\documentclass{amsart}
\usepackage{amssymb}
\usepackage{graphicx}
\usepackage{bm}

\allowdisplaybreaks[1]

\newcommand{\C}{\mathbb C}

\newcommand{\N}{\mathbb N}

\renewcommand{\H}{\mathbb H}
\renewcommand{\P}{\mathbb P}

\newcommand{\Z}{\mathbb Z}

\newcommand{\cE}{\mathcal E}
\newcommand{\CC}{\mathcal C}
\newcommand{\cF}{\mathcal F}

\newcommand{\cH}{\mathcal H}
\newcommand{\cI}{\mathcal I}
\newcommand{\cM}{\mathcal M}

\newcommand{\cR}{\mathcal R}

\newcommand{\cx}{\tilde x}

\newcommand{\re}{\mathrm{Re}}

\newcommand{\Sol}{\mathrm{Sol}}
\newcommand{\bu}{\bullet}
\newcommand{\na}{\nabla}
\newcommand{\pa}{\partial}
\newcommand{\ex}{\mathbf{e}}

\renewcommand{\a}{\alpha}

\newcommand{\g}{\gamma}
\renewcommand{\d}{\delta}

\renewcommand{\l}{\lambda}

\newcommand{\cL}{\mathcal{L}}

\newcommand{\f}{\varphi}
\newcommand{\G}{\mathit{\Gamma}}

\newcommand{\W}{\mathit{\Omega}}
\newcommand{\w}{\omega}

\newcommand{\la}{\langle}
\newcommand{\ra}{\rangle}
\newcommand{\tr}{\;^t}

\newcommand{\diag}{\mathrm{diag}}

\newcommand{\IZc}{I_{\Z^c}}
\newcommand{\ID}{I_{\N_0}^{\f_0}}
\newcommand{\IP}{I_{-\N}^{\f_0}}
\newcommand{\IZ}{I_{\Z}}

\newcommand{\res}{\mathrm{Res}} 

\newcommand{\ord}{\mathrm{ord}}
\newcommand{\comment}[1]{}
\newcommand{\wt}[1]{\widetilde #1}

\title[Relative twisted (co)homology groups]
{Relative twisted homology and cohomology groups associated with 
Lauricella's $F_D$}
\author{Keiji Matsumoto}
\address[Matsumoto]{
   Department of Mathematics,
   Faculty of Science,
   Hokkaido University,
   Sapporo 060-0810, Japan
}
\email{matsu@math.sci.hokudai.ac.jp}

\date{\today}

\theoremstyle{plain} 
\newtheorem{theorem}{\indent\sc Theorem}[section]
\newtheorem{lemma}[theorem]{\indent\sc Lemma}
\newtheorem{cor}[theorem]{\indent\sc Corollary}
\newtheorem{proposition}[theorem]{\indent\sc Proposition}

\theoremstyle{definition} 
\newtheorem{definition}[theorem]{\indent\sc Definition}

\newtheorem{remark}[theorem]{\indent\sc Remark}
\newtheorem{example}[theorem]{\indent\sc Example}

\numberwithin{equation}{section}

\begin{document}
\maketitle
\begin{abstract}
We introduce relative twisted homology and cohomology groups 
associated with Euler type 
integrals of solutions to Lauricella's system $\cF_D(a,b,c)$ of 
hypergeometric differential equations.
We define an intersection form between relative twisted homology groups
and that between relative twisted cohomology groups,  
and show their compatibility.
We prove that the relative twisted homology group is canonically 
isomorphic to the space of local solutions to $\cF_D(a,b,c)$ for 
any parameters $a,b,c$. Through this isomorphism, 
we study $\cF_D(a,b,c)$ 
by the relative twisted homology and cohomology groups 
and the intersection forms without any conditions on $a,b,c$.
\end{abstract}

\tableofcontents
\section{Introduction}
\label{sec:Intro}
There are many generalizations of the Gauss hypergeometric differential 
equation. It is known that 
Lauricella's system $\cF_D(a,b,c)$ given in (\ref{eq:LHGS}) is 
the simplest regular integrable system with multiple independent variables
$x_1,\dots,x_m$, where each of $(a,b,c)=(a,b_1,\dots,b_m,c)$ is a 
complex parameter. 
This system is of rank $m+1$, and local solutions to this system admit 
the path-integral representations of Euler type in (\ref{eq:int-rep}).  
By separating 
$$u(t)=u(t,x)=t^{b_1+\cdots+b_m-c}(t-x_1)^{-b_1}\cdots(t-x_m)^{-b_m}
(t-1)^{c-a}$$ 
and $\f_0=dt/(t-1)$ from the integrand in (\ref{eq:int-rep}), 
we have twisted homology and cohomology groups as in \cite[\S2]{AoKi}. 
We array the exponents of $u(t)$ at 
$0=x_0,x_1,\dots x_m,x_{m+1}=1$ and $x_{m+2}=\infty$ as
\begin{equation}
\label{eq:parameters}
\a=(\a_0,\a_1,\dots,\a_m,\a_{m+1},\a_{m+2})
=\big(\sum_{i=1}^m b_i-c,-b_1,\dots,-b_m,c-a,a\big),
\end{equation}
which satisfy 
$$\sum_{i=0}^{m+2}\a_i=0.$$
Under a non-integral condition $\a\in (\C-\Z)^{m+3}$, 
several properties of $\cF_D(a,b,c)$ are studied by 
twisted homology and cohomology groups, 
for details refer to \cite{AoKi}, \cite{CM}, \cite{KY}, 
\cite{M2}, \cite{OT}, \cite{Y2} and the references therein.
Moreover, the monodromy representation of $\cF_D(a,b,c)$ is studied 
in \cite{M3} under mild conditions. 
These studies are based on the key fact that the space $\Sol_x(a,b,c)$ 
of local solutions to $\cF_D(a,b,c)$ around $x$ 
is canonically isomorphic to that of sections of 
the trivial vector bundle over a small neighborhood of $x$ 
with fiber consisting of the twisted homology group, where 
$x$ is a point in the complement $X$ of the singular locus of $\cF_D(a,b,c)$.
However, in case of $\a\in \Z^{m+3}$, this key fact does not hold 
since the dimension of the twisted homology group is different from 
the rank of $\cF_D(a,b,c)$.

In this paper, we remove the non-integral condition 
$\a\in (\C-\Z)^{m+3}$ from the studies above by extending 
twisted homology and cohomology groups. 
As our extension of the twisted homology group, 
we introduce a relative twisted homology group $H_1(T,D;\cL)$, 
where the space $T=T_x$ is a subset of the complex projective line $\P^1$ 
consisting of points at which $u(t)\f_0$ is a
locally single-valued holomorphic $1$-form, the relative set $D=D_x$ is 
the intersection of $T$ and $\wt x=\{0,x_1,\dots,x_m,1,\infty\}$,  
$\cL=\cL_x$ is the local system associated with $u(t)$, and they depend on 
the parameters and the variables of $\cF_D(a,b,c)$.
We show that $H_1(T,D;\cL)$ is $m+1$ dimensional 
for any parameters $a,b,c$,  in particular, 
this property holds in the case $\a\in \Z^{m+3}$. 
By aligning points $x_1,\dots,x_m$ in convenient order,   
we give its basis $(\g_1^u,\dots,\g_{m+1}^u)$ in both two cases 
$\a\notin  \Z^{m+3}$ and $\a\in  \Z^{m+3}$.

We have a relative twisted cohomology group $H^1(T,D;\cL)$ as the 
dual of $H_1(T,D;\cL)$. 
We define 
$H^1_{alg}(T,D;\cL)$, $H^1_{C^\infty}(T,D;\cL)$ and 
$H^1_{C^\infty_V}(T,D;\cL)$ as the first cohomology groups of 
relative twisted de Rham complexes 
consisting of rational $k$-forms, smooth ones and 
those with certain vanishing property, respectively.  
We show that they are canonically isomorphic to $H^1(T,D;\cL)$
through the pairing 
$$\la \f,\g^u\ra=\int_\g u(t) \f,$$
where $\f\in H^1_*(T,D;\cL)$ ($*=alg,C^\infty,C^\infty_V$) 
is represented by a $1$-form, and $\g^u\in H_1(T,D;\cL)$ is represented by 
a $1$-chain $\g$ on which a branch of $u(t)$ is assigned.
We utilize the canonical isomorphisms among these relative twisted de Rham 
cohomology groups several times in this paper.

This pairing is extended to that between sections of 
the trivial vector bundles over a small neighborhood of $x$ with 
fibers  $H^1_{*}(T,D;\cL)$ and $H_1(T,D;\cL)$.
For simplicity, the spaces of sections of these trivial vector bundles 
are denoted by the same symbols $H^1_{*}(T,D;\cL)$ and $H_1(T,D;\cL)$ 
as their fibers.  
We show that the map 
$$\jmath_{\f_0}:H_1(T,D;\cL) \ni 
\g^u \mapsto \la \f_0,\g^u\ra =\int_\g u(t) \f_0\in \Sol_x(a,b,c)
$$
is isomorphic for any parameters $a,b,c$. 
This key fact enables us to study $\cF_D(a,b,c)$ by 
the relative twisted homology and cohomology groups 
$H_1(T,D;\cL)$ and $H^1_{*}(T,D;\cL)$.  
In our proof of the key fact, we consider 
$$\pa_j\la \f_0,\g^u\ra=\frac{\pa}{\pa x_j}\la \f_0,\g^u\ra\ 
(1\le j\le m).
$$
For the pairing $\la \phi,\g^u\ra$ between sections 
$\phi \in H^1_{C^\infty_V}(T,D;\cL)$ and $\g^u \in H_1(T,D;\cL)$, 
we have 
$$
\pa_j\la \phi,\g^u\ra
=\int_\g \big(u(t,x)\pa_j(\phi)+\pa_j(u(t,x))\phi\big)
=\int_\g u(t,x)\big(\pa_j \phi +\frac{\pa_j u(t,x)}{u(t,x)}\phi\big)
=\la \na_j\phi,\g^u\ra,
$$
where $\na_j=\pa_j-\dfrac{\pa_ju(t,x)}{u(t,x)}=\pa_j-\dfrac{\a_j}{t-x_j}$.
We can regard $\na_j$ as an action on $H^1_{C^\infty_V}(T,D;\cL)$. 
However, it cannot act directly on 
$H^1_{alg}(T,D;\cL)$ and $H^1_{C^\infty}(T,D;\cL)$
since their coboundary groups are not kept invariant under this action.
It acts on $H^1_{alg}(T,D;\cL)$ and $H^1_{C^\infty}(T,D;\cL)$
through the canonical isomorphisms 
$H^1_{alg}(T,D;\cL)\simeq H^1_{C^\infty_V}(T,D;\cL)$ and 
$H^1_{C^\infty}(T,D;\cL)\simeq H^1_{C^\infty_V}(T,D;\cL)$, respectively. 

Let  $D^\vee$ be the set of poles of the multivalued $1$-form $\f_0u(t)$, 
$T^\vee$ be the complement of $\wt x- D^\vee$ in $\P^1$ and 
$\cL^\vee$ be the local system associated with $1/u(t)$.
We have the relative twisted homology and cohomology groups 
$H_1(T^\vee,D^\vee;\cL^\vee)$ and $H^1_{*}(T^\vee,D^\vee;\cL^\vee)$,
where $*$ is $alg,C^\infty,C^\infty_V$ and the blank.
We define the intersection form $\cI_h$ between 
$H_1(T^\vee,D^\vee;\cL^\vee)$ and $H_1(T,D;\cL)$ 
by a similar way in \cite[\S1.4]{KY}.
We also define the intersection form $\cI_c$ between 
$H^1_{C^\infty_V}(T,D;\cL)$ and $H^1_{C^\infty_V}(T^\vee,D^\vee;\cL^\vee)$ by 
$$\iint_{T\cap T^\vee} \f\wedge \psi,$$ 
where $\f\in H^1_{C^\infty_V}(T,D;\cL)$ and 
$\psi\in H^1_{C^\infty_V}(T^\vee,D^\vee;\cL^\vee)$.
This double integral converges by the vanishing property 
of $1$-forms $\f$ and $\psi$. 
We can extend it to the intersection form $\cI_c$ between 
$H^1_{C^\infty}(T,D;\cL)$ and $H^1_{C^\infty}(T^\vee,D^\vee;\cL^\vee)$  
by utilizing the canonical isomorphisms between 
the relative twisted cohomology groups. 
We generalize the evaluation formula of $\cI_c$ in \cite[Theorem 1]{CM} 
so that it is valid without the condition $\a\in (\C-\Z)^{m+3}$.
By using this formula, we give bases of $H^1_{C^\infty}(T,D;\cL)$ 
and $H^1_{C^\infty}(T^\vee,D^\vee;\cL^\vee)$, which are dual to each other. 

We show the compatibility of the pairings between relative twisted 
homology and cohomology groups and the intersection forms $\cI_h$ and $\cI_c$.
Our proof is elementary one  based on Stokes' theorem.
This compatibility yields relations between period matrices and 
intersections matrices with respect to any bases of 
$H_1(T,D;\cL)$, $H^1_{C^\infty}(T,D;\cL)$, 
$H_1(T^\vee,D^\vee;\cL^\vee)$ and  $H^1_{C^\infty_V}(T^\vee,D^\vee;\cL^\vee)$.
These  are regarded as generalizations of results in \cite[\S3]{CM}.

A Pfaffian system and the monodromy representation of $\cF_D(a,b,c)$ are 
studied by the intersection forms $\cI_c$ and $\cI_h$ under the condition 
$\a\in (\C-\Z)^{m+3}$ in \cite{M2}.
By using the relative twisted cohomology and homology groups
and the intersection forms $\cI_c,\cI_h$,  
we express a connection matrix of a Pfaffian system  
and circuit matrices of $\cF_D(a,b,c)$ for any parameters $a,b,c$. 
We study invariant subspaces of $H^1_{C^\infty_V}(T,D;\cL)$  
under the actions $\na_1,\dots,\na_m$ and those of 
$H_1(T,D;\cL)$ under the monodromy representation. 
We show that if a parameter $\a_i$ belongs to $\Z$, 
then such a non-trivial subspace of $H^1_{C^\infty_V}(T,D;\cL)$
and that of $H_1(T,D;\cL)$ exist. 
Consequently, the monodromy representation of $\cF_D(a,b,c)$ is trivial 
if and only if $\a\in\Z^{m+3}$ and the number 
of the set $D$ is equal to $1$ or $m+2$. 

Other than those introduced here,  it is conceivable that 
there are many applications 
of the relative twisted homology and cohomology groups 
to various studies of $\cF_D(a,b,c)$, 
for examples,  the connection problem of local solutions, 
the study of difference equations for parameters, reduction formulas, etc.
The author expects that the relative twisted homology and cohomology groups
make progress in study of systems of hypergeometric differential equations.

\section{Lauricella's system $\cF_D(a,b,c)$}
\label{sec:FD}
In this section, we prepare facts on Lauricella's hypergeometric 
system $\cF_D(a,b,c)$ by referring \cite[\S9.1]{IKSY} and \cite[\S6]{Y1}.
Lauricella's hypergeometric series $F_D(a,b,c;x)$ is defined by 
$$
F_D(a,b,c;x)
=\sum_{n\in \N_0^m}
\frac{(a,\sum_{i=1}^m n_i)\prod_{i=1}^m(b_i,n_i)}
{(c,\sum_{i=1}^m n_i)\prod_{i=1}^m(1,n_i)}
\prod_{i=1}^m x_i^{n_i},
$$
where  $x_1,\dots,x_m$ are complex variables with $|x_i|<1$ $(1\le i\le m)$, 
$a$, $b=(b_1,\dots,b_m)$ and $c$ are complex parameters,  
$c\notin -\N_0=\{0,-1,-2,\dots\}$, and 
$(b_i,n_i)=b_i(b_i+1)\cdots(b_i+ n_i-1)$.
It admits an Euler type integral:
\begin{equation}
\label{eq:intrep}
\frac{\G(c)} 
{\G(a)\G(c\!-\! a)}
\int_1^\infty u(t,x)\f_0, \quad 
u(t,x)=t^{\sum_i b_i-c}(t\!-\!1)^{c-a}\prod_{i=1}^m(t\!-\! x_i)^{-b_i},\
\f_0=\frac{dt}{t\!-\!1}
\end{equation}
where the parameters $a$ and $c$ satisfy $0<\re(a)<\re(c)$.

The differential operators 
\begin{align}
\nonumber
x_i(1- x_i)\pa_i^2
+ (1- x_i)\sum\limits_{1\le j\le m}^{j\ne i}x_j\pa_i\pa_j 
+ [c- (a+ b_i+ 1)x_i]\pa_i 
- b_i\sum\limits_{1\le j\le m}^{j\ne i}x_j\pa_j- ab_i,&\\
\label{eq:LHGS}(1\le i\le m)& \\[2mm]
\nonumber
 (x_i-x_j)\pa_i\pa_j- b_j\pa_i+ b_i\pa_j,\hspace{3cm} (1\le i<j\le m)&
\end{align}
annihilate the series $F_D(a,b,c;x)$, where $\pa_i=\dfrac{\pa}{\pa x_i}$ 
$(1\le i\le m)$.
Lauricella's system $\cF_D(a,b,c)$ is defined by 
the ideal generated by these operators  
in the ring of differential operators with rational function coefficients
$\C(x_1,\dots,x_m)\la \pa_1,\dots,\pa_m\ra$. 
Though the series $F_D(a,b,c;x)$ is not defined 
when $c\in -\N_0$, the system $\cF_D(a,b,c)$ 
can be defined even in this case. 
It is a regular holonomic system of rank $m+1$ with singular locus 
\begin{equation}
\label{eq:sing-loc}
S=\big\{x\in \C^m\Big| \prod_{i=1}^m [x_i(1-x_i)]
\prod_{1\le i<j\le m} (x_i-x_j)=0\big\}\cup 
(\cup_{i=1}^\infty\{ x_i=\infty\})\subset (\P^1)^m.
\end{equation}
We set 
$$X=(\P^1)^m-S=\big\{(x_1,\dots,x_m)\in \C^m\mid \prod_{0\le i<j\le m+1}
(x_j-x_i)\ne 0\big\},$$
where $x_0=0$ and $x_{m+1}=1$. 
We introduce a notation 
$$\cx=(x_0,x_1,\dots,x_m,x_{m+1},x_{m+2})=(0,x_1,\dots,x_m,1,\infty)
=(0,x,1,\infty)$$
for $x\in X$.
Let $\Sol_x(a,b,c)$ be the vector space of solutions to $\cF_D(a,b,c)$ on 
a small simply connected neighborhood $W(\subset X)$ of $x$. 
It is called the local solution space to $\cF_D(a,b,c)$ around $x$, 
and it is $m+1$ dimensional. 
If a $1$-chain $\g$ satisfies certain vanishing properties for its boundary 
then the integral 
\begin{equation}
\label{eq:int-rep}
\int_{\g} u(t,x)\f_0
\end{equation}
gives an element of $\Sol_x(a,b,c)$. We remark that 
it happens that this integral degenerates into the zero solution.

\section{Relative twisted homology groups}
\label{sec:RTH}
Recall that 
$$
\a=(\a_0,\a_1,\dots,\a_m,\a_{m+1},\a_{m+2})
=\big(-c+\sum_{i=1}b_i,-b_1,\dots,-b_m,c-a,a\big),\quad
\ \sum_{i=0}^{m+2} \a_i=0,
$$
where $a$, $b_1,\dots,b_m$ and $c$ are the parameters of Lauricella's $F_D$
belonging to $\C$. 
We fix $\a$ and $x\in X$.
We divide the index set $I=\{0,1,2,\dots,m,m+1,m+2\}$ of $\a$ 
into two disjoint subsets
$$\IZ=\{i\in I\mid \a_i\in \Z\},\quad \IZc=\{i\in I\mid \a_i\notin \Z\}.
$$
Moreover, we divide $\IZ$ into two disjoint subsets 
\begin{equation}
\label{eq:index set}
\ID=
\{i\in \IZ\mid \ord_{x_i} (u(t)\f_0)\ge 0 \}, \quad 
\IP=\{i\in \IZ\mid \ord_{x_i} (u(t)\f_0)< 0 \},
\end{equation}
where $u(t)$ and $\f_0$ are in the integral (\ref{eq:intrep}),
and $\ord_{x_i}$ denotes the order of zero of meromorphic 
functions or $1$-forms at $t=x_i$.
We remark that though we have 
$$
\{\a_i\mid i\in \IZ\}=\{\a_i\in \a \mid \a_i\in \Z\},\quad 
\{\a_i\mid i\in \IZc\}=\{\a_i\in \a \mid \a_i\notin \Z\},$$
it happens that 
\begin{align*}
\{\a_i\mid i\in \ID\}&\varsubsetneq
\{\a_i\in \a \mid \a_i\in \N_0=\{0,1,2,\cdots\}\},
\\ 
\{\a_i\mid i\in \IP\}&\varsupsetneq 
\{\a_i\in \a \mid \a_i\in -\N=\{-1,-2,-3,\dots\}\},  
\end{align*}
since we count the order of zero  by not the function $u(t)$ but 
the $1$-form $u(t)\f_0$.
We set $\#\ID=r$, $\#\IP=s$, $\#\IZc=m+3-r-s$, and 
$$
\ID=\{i_1,\dots,i_r\},\quad 
\IP=\{i_{r+1},\dots,i_{r+s}\},\quad 
\IZc=\{i_0,i_{r+s+1},\dots,i_{m+2}\}.
$$
If the set $\IZc$ is empty, then 
neither $\ID$ nor $\IP$ is empty since 
the total  sum of the orders of zeros of $u(t)\f_0$ is $-2$; 
in this case  we regard the set 
$\IP$ as 
$$\{i_0,i_{r+1},\dots,i_{r+s-1}\}\quad (s=m+3-r,\ r>0,\ s>0).$$

We define a subspace $T$ of $\P^1$ and a subset $D$ in $T$ by 
\begin{align*}
&T=T_x=\P^1-\{x_i\mid i\in \IZc\cup \IP\}=\P^1-\{x_{i_0},x_{i_{r+1}},\dots,
x_{r+s},x_{r+s+1},\dots,x_{i_{m+2}}\}
, \\
&D=D_x=\{x_i\mid i\in \ID\}=\{x_{i_1},\dots,x_{i_{r}}\}.
\end{align*}
Note that the space $T$ consists points at which 
$u(t)\f_0$ is a locally single-valued holomorphic $1$-form. 
We set $B=B_x=\{x_{i_0},x_{i_{r+s+1}},\dots,x_{i_{m+2}}\}\subset T^c$ 
if $\IZc\ne \emptyset$.
Let $\cL=\cL_x$ be the locally constant sheaf on $T=T_x$ defined by 
$u(t)=u(t,x)$. 
We define $\CC_k(T;\cL)$ by the $\C$-vector space of twisted $k$-chains
which are finite linear combinations of $k$-simplices in $T$ on which 
branches of $u(t)$ are assigned.
Let $\CC_k(D;\cL)$ be the subspace defined by the restrictions of 
elements in $\CC_k(T;\cL)$ to $D$: 
$$\CC_k(D;\cL)=\CC_k(T;\cL)|_D.$$
It is clear that 
$$\CC_1(D;\cL)=\CC_2(D;\cL)=0.$$
Since the space $\CC_0(D;\cL)$ is generated by 
$x_{i}\in D$ with the germ $u(t)|_{x_i}$ of a branch $u(t)$ at $t=x_i$,
we have $$\dim \CC_0(D;\cL)=r.$$
Here note that the germ $u(t)|_{x_i}$ is non-zero even in the case $u(x_i)=0$.  
The space of relative twisted $k$-chains is defined by the quotient
$$
\CC_k(T,D;\cL)=\CC_k(T;\cL)/\CC_k(D;\cL).
$$
We have the boundary operator $\pa^u:\CC_k(T;\cL)\to \CC_{k-1}(T;\cL)$ by
extending 
$$\pa^u(\mu^{u(t)|_{\mu}})=(\pa\mu)^{u(t)|_{\pa\mu}}$$
linearly, 
where $\mu^{u(t)|_{\mu}}$ is a twisted $k$-chain given by 
a $k$-simplex $\mu$ in $T$ and a branch $u(t)|_{\mu}$ of $u(t)$ on $\mu$, 
$\pa$ is the usual boundary operator,
and $u(t)|_{\pa\mu}$ is the restriction of the branch $u(t)|_{\mu}$ to $\pa\mu$. 
We have an exact sequence of chain complexes 
\begin{equation}
\label{eq:chain-complex}
0 \longrightarrow  
\CC_\bullet(D;\cL)\longrightarrow 
\CC_\bullet(T;\cL)\longrightarrow 
\CC_\bullet(T,D;\cL)\longrightarrow 0,
\end{equation}
where the boundary operators of  $\CC_\bullet(D;\cL)$ and $\CC_\bullet(T,D;\cL)$
are naturally induced from $\pa^u$ on $\CC_\bullet(T;\cL)$.
We define 
$H_k(D;\cL)$,
$H_k(T;\cL)$ and 
$H_k(T,D;\cL)$
by the $k$-th homology groups of the complexes 
$\CC_\bullet(D;\cL)$, 
$\CC_\bullet(T;\cL)$ and 
$\CC_\bullet(T,D;\cL)$, respectively.
We call $H_k(T,D;\cL)$ the $k$-th relative twisted homology group. 
We have an exact sequence 
\begin{equation}
\label{eq:ex-seq-hom}
\begin{array}{rl}
0\longrightarrow &H_2(D;\cL) \longrightarrow  
H_2(T;\cL)\longrightarrow 
H_2(T,D;\cL)\\
\overset{\pa^u}{\longrightarrow}
&H_1(D;\cL)\longrightarrow  
H_1(T;\cL)\longrightarrow 
H_1(T,D;\cL)\\
\overset{\pa^u}{\longrightarrow}
&H_0(D;\cL)\longrightarrow 
H_0(T;\cL)\longrightarrow 
H_0(T,D;\cL)\longrightarrow 0.
\end{array}
\end{equation}
Here an element of $H_k(T,D;\cL)$ is represented by 
a $k$-chain $\ell^{u(t)|_{\ell}}\in \CC_k(T;\cL)$ with its boundary in 
$\CC_{k-1}(D;\cL)$, and 
the connection map $\pa^u$ is naturally defined by the boundary operator
as
$$H_k(T,D;\cL)\ni\ell^{u(t)|_\ell} \mapsto \pa \ell^{u(t)|_{\pa(\ell)}} \in 
\CC_{k-1}(D;\cL)
=\left\{ 
\begin{array}{ccc}
0 &\textrm{if} & k=2, \\
H_{0}(D;\cL)&\textrm{if} & k=1. 
\end{array}\right.
$$

\begin{theorem}
\label{th:dim}
For any parameters $\a$, we have 
$$
H_0(T,D;\cL)=H_2(T,D;\cL)=0,\quad \dim H_1(T,D;\cL)=m+1.
$$

\end{theorem}
\begin{proof}
By the definition, it is easy to see that $H_2(T;\cL)=0$.  
Since $H_2(D;\cL)=H_1(D;\cL)=0$, we have $H_2(T,D;\cL)\simeq H_2(T;\cL)=0$ 
by the exact sequence (\ref{eq:ex-seq-hom}).
Since $T$ is connected, the map $H_0(D;\cL)\to H_0(T;\cL)$ in 
the exact sequence (\ref{eq:ex-seq-hom}) is surjective. 
Thus the kernel of the surjective map $H_0(T;\cL)\to H_0(T,D;\cL)$ is 
the whole space $H_0(T;\cL)$, which means $H_0(T,D;\cL)=0$.
The exact sequence (\ref{eq:ex-seq-hom}) reduces to 
\begin{equation}
\label{eq:ess-ex-seq-hom}
0\longrightarrow 
H_1(T;\cL)\longrightarrow 
H_1(T,D;\cL)
\overset{\pa^u}{\longrightarrow}
H_0(D;\cL)\longrightarrow 
H_0(T;\cL)\longrightarrow 0.
\end{equation}
Thus we have 
$$
\dim H_1(T;\cL)-\dim H_1(T,D;\cL)
+\dim H_0(D;\cL)-\dim H_0(T;\cL)=0.
$$
Note that $\dim H_0(D;\cL)=r$ and 
\begin{align*}
&\dim H_1(T;\cL)-\dim H_0(T;\cL)=
-\dim H_2(T;\cL)+\dim H_1(T;\cL)-\dim H_0(T;\cL)\\
=&-\chi(T)=-(2-(m+3-r))=m+1-r,
\end{align*}
where $\chi(T)$ denotes the Euler number of $T$. 
Hence we have $\dim H_1(T,D;\cL)=m+1$.
\end{proof}

\begin{remark}
\label{rem:relative-cycle}
The quotient space $H_1(T,D;\cL)/H_1(T;\cL)$ is isomorphic to 
the image of the map $\pa^u$ in (\ref{eq:ess-ex-seq-hom}). 
It coincides with the kernel of 
the surjective map $H_0(D;\cL) \to H_0(T;\cL)$ in (\ref{eq:ess-ex-seq-hom}). 
If $\a\in \Z^{m+3}$ then $H_0(T;\cL)$ is one dimensional, 
otherwise $H_0(T;\cL)=0$. 
Thus we have 
$$
\dim H_1(T,D;\cL)/H_1(T;\cL)
=\tilde r =\left\{\begin{array}
{ccc}
r &\textrm{if} & \a\notin \Z^{m+3},\\
r-1&\textrm{if} & \a\in \Z^{m+3}.
\end{array}
\right.
  $$ 
We call an element $\g^u\in H_1(T,D;\cL)$ satisfying $0\ne \pa^u(\g^u)\in 
H_0(D;\cL)$ a relative cycle, which represents 
a non-zero element of the quotient space 
$H_1(T,D;\cL)/H_1(T;\cL)$.
\end{remark}

We give $m+1$ elements of $H_1(T,D;\cL)$.
We take a base point $\dot x\in X$ so that 
\begin{equation}
\label{eq:base-point}
\begin{array}{ccc}
x_{i_0}<x_{i_1}<\cdots
<x_{i_{m+2}}=x_{m+2}=\infty&\textrm{if} & m+2\in \IZc, \\
-\infty=x_{m+2}=x_{i_1}<\cdots
<x_{i_{m+2}}<x_{i_0} &\textrm{if} & m+2\in \ID,\\

x_{i_{r+s+1}}<\cdots<x_{i_{m+2}}<x_{i_0}<x_{i_1}<\cdots<
x_{i_{r+s}}=x_{m+2}=\infty
&\textrm{if} & m+2\in \IP.
\end{array}
\end{equation}
We choose a base point $\dot t$ in the upper half space $\H_T$ of $T$.
Let $\ell_{i_j}$ $(0\le j\le m+2)$ be a path from $\dot t$ to $x_{i_j}$
via $\H_T$. 
Let $\circlearrowleft_{i_j}$ $(0\le j\le m+2)$ 
be a loop starting from $\dot t$, approaching to 
$x_{i_j}$ in $\H_T$, turning once around $x_{i_j}$ 
positively, and tracing back to $\dot t$; 
see Figure \ref{fig:cycles}.

\begin{figure}[htb]
\includegraphics[width=12cm]{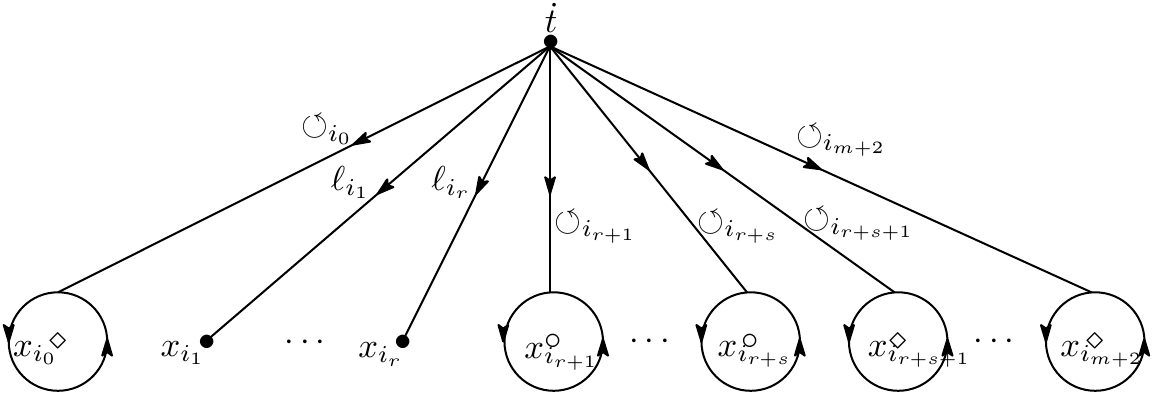}
\label{fig:cycles}
\caption{Chains and relative chains}
\end{figure}
We fix a branch of $u(t)$ on $\H_T$ 
by the assignment $0 < \arg(t-x_i)<\pi$ $(0\le i\le m+1)$ for $t\in \H_T$.

We consider two cases: (1) $\a\notin \Z^{m+3}$; (2) $\a\in \Z^{m+3}$.

\noindent
(1) In this case, we have $\a_{i_0},\a_{i_{r+s+1}}, \dots,\a_{i_{m+2}} \notin \Z$, 
$\a_{i_1}, \dots,\a_{i_r},\a_{i_{r+1}}, \dots,\a_{i_{r+s}}\in \Z$, and  
$x_{i_1},\dots,x_{i_r} \in D$.
We set 
\begin{equation}
\label{eq:hom-basis-1}
\g^u_{j}=
\left\{\begin{array}{cll} 
\ell_{i_{j}}^u-\dfrac{\circlearrowleft_{i_{0}}^u}{1-\l_{i_0}} 
& \textrm{if} & 1\le j\le r,\\[4mm]
\circlearrowleft_{i_j}^u & \textrm{if} & r+1\le j\le r+s,\\[2mm]
\circlearrowleft_{i_{j}}^u
-\dfrac{1-\l_{i_j}}{1-\l_{i_{0}}}\circlearrowleft_{i_{0}}^u 
& \textrm{if} & r+s+1\le j\le m+1.\\
\end{array}
\right.
\end{equation}
(2) In this case, we have
$x_{i_1},\dots,x_{i_r}\in D$, 
and $r+s=m+3$.
We set 
\begin{equation}
\label{eq:hom-basis-2}
\g^u_{j}=\left\{\begin{array}{cll} 
\ell_{i_{j+1}}^u-\ell_{i_1}^u & \textrm{if} & 1\le j\le r-1,\\[2mm]
\circlearrowleft_{i_{j+1}}^u & \textrm{if} & r\le j\le r+s-2=m+1.\\
\end{array}
\right.
\end{equation}

\begin{theorem}
\label{th:TRH-basis}
The elements $\g^u_1,\dots,\g^u_{m+1}$ form a basis of $H_1(T,D;\cL)$.
\end{theorem}
\begin{proof}
In the case (1), it is shown in \cite[\S3]{M3} that 
$\g_{r+1}^u,\dots,\g^u_{m+1}$ form a basis of $H_1(T;\cL)$. 
Since the image $\g^u_{j}$ $(1\le j\le r)$ under the map $\pa^u$ is 
a non-zero element of $H_0(D;\cL)$ given by the point $t=x_{i_j}$ 
with the germ $u(t)$ at $x_{i_j}$, 
$\g^u_1,\dots,\g^u_{r}$ are linearly independent 
and they do not belong to $H_1(T;\cL)$. Hence  $\g^u_1,\dots,\g^u_{m+1}$
form a basis of $H_1(T,D;\cL)$.

In the case (2), we see that $\g_{r}^u,\dots,\g^u_{m+1}$ form a basis of 
$H_1(T;\cL)$ similarly to the case (1). 
Recall that 
$$\dim H_0(D;\cL)=r,\quad \dim H_0(T;\cL)=1,\quad 
\dim \pa^u(H_1(T,D;\cL))=r-1$$
in this case. 
In the images $\pa^u(\g^u_1),\dots,\pa^u(\g^u_{r-1})\in H_0(D;\cL)$, 
the $0$-chain $x_{i_{j+1}}$ $(1\le j\le r-1)$ 
appears only in $\pa^u(\g^u_{j})$, 
$\g^u_1,\dots,\g^u_{r-1}$ are linearly independent, 
and they do not belong to $H_1(T;\cL)$.
Hence  $\g^u_1,\dots,\g^u_{m+1}$ form a basis of $H_1(T,D;\cL)$.
\end{proof}

\section{Relative twisted cohomology groups}
\label{sec:RTC}
We define $H^k(T,D;\cL)$, $H^k(T;\cL)$ and $H^k(D;\cL)$  by the $k$-th 
cohomology groups of the cochain complexes $\CC^\bullet(T,D;\cL)$, 
$\CC^\bullet(T;\cL)$ and $\CC^\bullet(D;\cL)$, 
which are the dual complexes of chain complexes in (\ref{eq:chain-complex}). 
We call $H^k(T,D;\cL)$ the $k$-th relative twisted cohomology group.
Since cochain complexes satisfy 
$$0 \longrightarrow  
\CC^\bullet(T,D;\cL)\longrightarrow 
\CC^\bullet(T;\cL)\longrightarrow 
\CC^\bullet(D;\cL)\longrightarrow 
0,
$$
we have an exact sequence 
\begin{equation}
\label{eq:ex-seq-cohom}
0 \longrightarrow  
H^0(T;\cL)\longrightarrow 
H^0(D;\cL)\longrightarrow 
H^1(T,D;\cL)\longrightarrow
H^1(T;\cL)\longrightarrow 
0.
\end{equation}
By Theorem \ref{th:dim}, we have the following corollary. 
\begin{cor}
\label{cor:dim}
For any parameters $\a$, we have 
$$
H^0(T,D;\cL)=H^2(T,D;\cL)=0,\quad \dim H^1(T,D;\cL)=m+1.
$$
\end{cor}

We give three kinds of relative twisted de Rham cohomology groups 
isomorphic to $H^k(T,D;\cL)$ in this section.
We define a twisted exterior derivative $\na_{t}$ by $d+\w\wedge$,
where 
$$\w=d\log u(t,x)=\sum_{i=0}^{m+1} \frac{\a_i dt}{t-x_i}.$$
Let $\W^k(\cx)$ ($k=0,1,2$) be the vector space of rational differential 
$k$-forms with poles only on entries of $\cx=(0,x_1,\dots,x_m,1,\infty)$. 
We define subspaces of $\W^k(\cx)$ by 
\begin{align*}
\W^0(T,D;\cL)&=\{f(t)\in \W^0(\cx)\mid \ord_{x_i}(u(t)\cdot f(t)) \ge 1
\textrm{ for any } x_i\in D\},\\
\W^1(T,D;\cL)&=\{\f(t)\in \W^1(\cx)\mid \ord_{x_i}(u(t)\cdot \f(t)) \ge 0 
\textrm{ for any } x_i\in D\},
\\
\W^2(T,D;\cL)&=0.
\end{align*}
Note that $d(u(t)\cdot f(t))=u(t)\cdot \na_{t} f(t)$ and that  
if $\na_{t} f(t)$ is not identically $0$ then 
$\ord_{x_i}(\na_{t} f(t))=\textrm{ord}_{x_i}(f(t))-1$ for 
$0\le i\le m+2$. 
Thus we see that $\na_{t} f(t)\in 
\W^1(T,D;\cL)$ for any $f(t)\in \W^0(T,D;\cL)$.
We define relative twisted algebraic de Rham cohomology groups by 
\begin{align*}
H_{alg}^0(T,D;\cL)&=\ker(\na_{t}:\W^0(T,D;\cL)\to \W^1(T,D;\cL)),\\
H_{alg}^1(T,D;\cL)&=\W^1(T,D;\cL)/\na_{t}(\W^0(T,D;\cL)),\\
H_{alg}^2(T,D;\cL)&=0.
\end{align*}

\begin{proposition}
\label{prop:alg-dim}
We have
$$H_{alg}^0(T,D;\cL)=0,\quad 
\dim H_{alg}^1(T,D;\cL)=m+1.$$
\end{proposition}

\begin{proof}
Since 
$$
\ker(\na_{t}:\W^0(\cx)\to \W^1(\cx))
=\left\{\begin{array}{ccl}
0 &\textrm{if} & \a\notin \Z^{m+3},\\
\la u(t)^{-1} \ra &\textrm{if} & \a\in \Z^{m+3},
\end{array}
\right.
$$
and $u(t)\cdot u(t)^{-1}=1$ does not vanish at $x_i\in D\ne \emptyset$ 
in the case $\a\in \Z^{m+3}$, we have 
$H_{alg}^0(T,D;\cL)=0$ for any $\a$.

We have a short exact sequence of complexes of sheaves
$$
\begin{array}{ccccccccc}
0& \longrightarrow& \W^0_T(D;\cL) &\longrightarrow & \W^0_T(\cL) &
\longrightarrow &\bigoplus\limits_{x_i\in D} \C\cdot x_i& \longrightarrow & 0\\
\downarrow & &\downarrow\na_{t} & &\downarrow\na_{t} & &\quad\downarrow\na_{t} & 
&\downarrow  \\[3mm]
0& \longrightarrow& \W^1_T(D;\cL) &\longrightarrow & \W^1_T(\cL) &
\longrightarrow &0 & \longrightarrow & 0,\\
\end{array}
$$
where $\W^k_T(D;\cL)$ and $\W^k_T(\cL)$ are sheaves over $T$ satisfying 
\begin{align*}
&H^0(\W^k_T(D;\cL))=\W^k(T,D;\cL),\\ 
&H^0(\W^k_T(\cL))=\W^k(T;\cL)
=\{\f\in \W^k(\cx)\mid \ord_{x_i}(u(t)\cdot \f(t)) \ge 0 
\textrm{ for any } x_i\in D\},
\end{align*}
and $\C\cdot x_i$ denotes the skyscraper sheaf at $x_i$.
As in \cite[Appendix]{EV}, it induces a long exact sequence of 
hypercohomology groups
$$
\begin{array}{cccccccccc}
0&\longrightarrow& \H^0(\W^\bullet_T(D;\cL))& \longrightarrow 
&\H^0(\W^\bullet_T(\cL)) & \longrightarrow 
& \H^0(\bigoplus\limits_{x_i\in D} \C\cdot x_i)\\
 &\longrightarrow 
&\H^1(\W^\bullet_T(D;\cL))
&\longrightarrow &\H^1(\W^\bullet_T(\cL)) &\longrightarrow&  0.
\end{array}
$$
Since $T$ is an affine space, 
the cohomology groups $H^1(\W^k_T(D;\cL))$ and $H^1(\W^k_T(\cL))$ vanish.  
As in \cite[Appendix]{EV}, hypercohomology groups reduce to 
\begin{align*}
\H^j(\W^\bullet_T(D;\cL))=&H^j(H^0(\W^\bullet_T(D;\cL)))
=H^j_{alg}(T,D;\cL),\\
\H^j(\W^\bullet_T(\cL))=&H^j(H^0(\W^\bullet_T(\cL)))
=H^j_{alg}(T;\cL)\\
=&\left\{
\begin{matrix}
\ker(\na_t:\W^0(T;\cL)\to \W^1(T;\cL))& \textrm{if} &j=0,\\
\W^1(T;\cL)/\na_t(\W^0(T;\cL)) & \textrm{if} &j=1.
\end{matrix}
\right.
\end{align*}
Hence we have an exact sequence of cohomology groups
\begin{equation}
\label{eq:ES-cohomo}
0\longrightarrow 
H_{alg}^0(T;\cL)\longrightarrow 
H_{alg}^0(D;\cL)\overset{\na_{t}}{\longrightarrow }
H_{alg}^1(T,D;\cL)\longrightarrow 
H_{alg}^1(T;\cL)\longrightarrow 0,
\end{equation}
where $H_{alg}^0(D;\cL)=H^0(\bigoplus\limits_{x_i\in D} \C\cdot x_i)$. 
Note that 
$$\dim H_{alg}^0(T;\cL)-
\dim H_{alg}^0(D;\cL)+
\dim H_{alg}^1(T,D;\cL)-
\dim H_{alg}^1(T;\cL)
=0.
$$
Since 
$$\dim H_{alg}^0(D;\cL)=\# D=r,
\quad 
\dim H_{alg}^0(T;\cL)-\dim H_{alg}^1(T;\cL)=\chi(T)=-m-1+r,$$
we have $\dim H_{alg}^1(T,D;\cL)=m+1$.
\end{proof}

\begin{remark}
Though the rational $1$-form $\w$ satisfies $\w=\na_{t}(1)$, 
the constant function $1$ does not belong to $\W^0(T,D;\cL)$ in general, 
$\w$ is not always the zero  of $H_{alg}^1(T,D;\cL)$.
Refer to Theorem \ref{th:cint-number} for details. 
\end{remark}

\begin{theorem}
\label{th:cohom-dual}
The space $H_{alg}^1(T,D;\cL)$ is dual to $H_1(T,D;\cL)$ 
by the integral 
\begin{equation}
\label{eq:dual-cohom-hom}
\la \f,\g^u\ra =\sum_i c_i\int_{\mu_i} u(t)|_{\mu_i} \f\in \C.
\end{equation}
Here an element $H_{alg}^1(T,D;\cL)$ is represented by 
$\f\in \W^1(T,D;\cL)$, and an element $\g^u$ of $H_1(T,D;\cL)$ 
is represented by $\sum_i c_i \mu_i^{u}\in \CC_1(T,D;\cL)$, 
where $c_i\in \C$ and $\mu_i^{u}$  denotes a $1$-simplex $\mu_i$ on which 
a branch $u(t)|_{\mu_i}$ of $u(t)$ is assigned. 
\end{theorem}

\begin{proof}
Since the integral (\ref{eq:dual-cohom-hom}) converges, 
there is a linear map from $H_1(T,D;\cL)$ to $\C$ given by 
$$
\jmath_{\f}: H_1(T,D;\cL) \ni  \g^u
\mapsto 
\la \f,\g^u\ra \in \C
$$
for any element $\f\in H_{alg}^1(T,D;\cL)$. 
We show the map 
$$\jmath : H_{alg}^1(T,D;\cL)\ni \f\mapsto \jmath_{\f}\in 
H_1(T,D;\cL)^*$$
is bijective, where $H_1(T,D;\cL)^*$ denotes 
the dual space of $H_1(T,D;\cL)$
isomorphic to $H^1(T,D;\cL)$.
Though we can show it by applying the five lemma to 
the exact sequences (\ref{eq:ES-cohomo}) and (\ref{eq:ex-seq-cohom}), 
we give a direct proof. 
Since $\dim H_1(T,D;\cL)=\dim H^1_{alg}(T,D;\cL)=m+1$ by 
Theorem \ref{th:dim} and Proposition \ref{prop:alg-dim},
we have only to show that $\jmath$ is injective. 
Suppose that $\f$ is an element of the kernel of $\jmath$.
This means that 
$$\la \f,\g^u_j\ra=0\quad (1\le j\le m+1),$$ 
where $(\g^u_1,\dots,\g^u_{m+1})$ is the basis of $H_1(T,D;\cL)$ 
given in (\ref{eq:hom-basis-1}) or (\ref{eq:hom-basis-2}).
By the exact sequence 
(\ref{eq:ES-cohomo}), $H_{alg}^1(T,D;\cL)$ is regarded as 
the direct sum of $H_{alg}^1(T;\cL)$ and $\na_{t}(H_{alg}^0(D;\cL))$.
Since the spaces $H_{alg}^1(T;\cL)$ and $H_1(T;\cL)$ are dual 
to each other and  $\la \f,\g^u_j\ra=0$ for $\g^u_j\in H_1(T;\cL)$, 
$\f$ belongs to $\na_{t}(H^0_{alg}(D;\cL))$. 
Thus there exists $f(t)\in \W(\cx)$ such that $\na_{t}(f(t))=\f$. 
We consider two cases (1) $\a\notin \Z^{m+3}$ and (2) $\a\in \Z^{m+3}$.

\noindent (1) $\a\notin \Z^{m+3}$. In this case, we have 
$$
0=\la \f,\g^u_j\ra=\la \na_{t} f(t),\g^u_j\ra
=\Big[u(t) f(t)\Big]_{t=\pa(\g_j)}
=u(x_{i_j}) f(x_{i_j}), 
$$
where $\g_j^u$ $(1\le j\le r)$  is a relative cycle with topological 
boundary $x_{i_j}\in D$. This means that $f$ belongs to $\W^0(T,D;\cL)$ and 
$\f$ is the zero of $H_{alg}^1(T,D;\cL)$.

\noindent
(2) $\a\in \Z^{m+3}$. In this case, we have 
$$
0=\la \f,\g_j\ra=\la \na_{t} f(t),\g_j\ra=\Big[u(t) f(t)\Big]_{\pa(\g_j)}
=u(x_{i_{j+1}})f(x_{i_{j+1}})-u(x_{i_1})f(x_{i_1}),  
$$
for $1\le j\le r-1$, 
where $\g_j$  is the relative cycles with topological 
boundary consisting of  $x_{i_{j+1}},x_{i_1}\in D$.
In this case, $H_{alg}^0(T;\cL)$ is a $1$-dimensional space spanned by 
$1/u(t)$, which satisfies $\na_{t}(1/u(t))=0$. 
Since the element $f(t)-f(x_{i_1})u(x_{i_1})/u(t)$ satisfies 
$$
\na_t(f(t)-f(x_{i_1})u(x_{i_1})/u(t))=\f,\quad 
\Big[u(t)\cdot\big(f(t)-f(x_{i_1})u(x_{i_1})/u(t)\big)\Big]_{t=x_{i_j}}
=0 
$$
for any $x_{i_j}\in D$,
it belongs to $\W^0(T,D;\cL)$ and $\f$ is the zero of $H_{alg}^1(T,D;\cL)$.
Therefore the map $\jmath:H^1_{alg}(T,D;\cL)\to H_1(T,D;\cL)^*$ is injective
for any $\a$.
\end{proof}

By Theorem \ref{th:cohom-dual} together with Proposition \ref{prop:alg-dim}
yields the following. 
\begin{cor}
\label{cor:cohom-alg-iso}
The relative twisted  algebraic de Rham cohomology group 
$H_{alg}^k(T,D;\cL)$ is canonically isomorphic to $H^k(T,D;\cL)$.
\end{cor}

Let $\cE^k(\cx)$ ($k=0,1,2$) be the vector space of $C^\infty$-differential 
$k$-forms on $T-D$. 
We define subspaces of $\cE^k(\cx)$ $(k=0,1,2)$ by
\begin{align*}
&\cE^0(T,D;\cL)=\{f(t)\in \cE^0(\cx)\mid 
u(t)\cdot f(t) \textrm{ is } C^\infty \textrm{ on } U_i, 
\lim_{t\to x_i} u(t)\cdot f(t) =0
\textrm{ for any } i\in \ID\},\\
&\cE^k(T,D;\cL)=\{\f(t)\in \cE^k(\cx)\mid 
u(t)\cdot \f(t) \textrm{ is } C^\infty \textrm{ on } U_i
\textrm{ for any } i\in \ID\}\quad (k=1,2),\\
&\cE^k_V(T,D;\cL)=\{\f(t)\in \cE^k(\cx)\mid 
u(t)\cdot \f(t) 
\textrm{ is identically } 0 \textrm{ on } V_i \textrm{ for any } 
i\in \ID\cup \IZc\},
\end{align*}
where $U_i$ and $V_i$ are sufficiently small neighborhood of $x_i$
satisfying $V_i\subset U_i$. 
We define relative twisted ($C^\infty$ de Rham) $k$-th cohomology groups  
$H_{C^\infty}^k(T,D;\cL)$ and 
$H_{C^\infty_V}^k(T,D;\cL)$ as the $k$-th cohomology groups of 
the complexes 
$$
\begin{array}{ccccccc}
\cE^0(T,D;\cL)&\overset{\na_{t}}{\longrightarrow}& \cE^1(T,D;\cL)
&\overset{\na_{t}}{\longrightarrow} &\cE^2(T,D;\cL)&
\overset{\na_{t}}{\longrightarrow} &0,\\
\cE^0_V(T,D;\cL)&\overset{\na_{t}}{\longrightarrow} &\cE^1_V(T,D;\cL)
&\overset{\na_{t}}{\longrightarrow} &\cE^2_V(T,D;\cL)
&\overset{\na_{t}}{\longrightarrow} &0,
\end{array}
$$
respectively, i.e.,
\begin{align*}
H^k_{C^\infty}(T,D;\cL)&=\ker(\na_{t}:\cE^k(T,D;\cL)\to \cE^{k+1}(T,D;\cL))/
\na_{t}(\cE^{k-1}(T,D;\cL)),\\
H^k_{C^\infty_V}(T,D;\cL)&=\ker(\na_{t}:\cE^k_V(T,D;\cL)\to 
\cE^{k+1}_V(T,D;\cL))/\na_{t}(\cE^{k-1}_V(T,D;\cL)).
\end{align*}

\begin{theorem}
\label{th:smooth-cohomology}
The natural inclusions 
$$H^k_{alg}(T,D;\cL)\hookrightarrow  H^k_{C^\infty}(T,D;\cL), 
\quad 
H^k_{C^\infty_V}(T,D;\cL)\hookrightarrow  H^k_{C^\infty}(T,D;\cL).
$$
are isomorphisms. The relative twisted cohomology groups 
$H^k_{alg}(T,D;\cL)$, $H^k_{C^\infty}(T,D;\cL)$ and 
$H^k_{C^\infty_V}(T,D;\cL)$ are canonically isomorphic to 
$H^k(T,D;\cL)$.
In particular, 
$$H^0_{C^\infty}(T,D;\cL)=
H^0_{C^\infty_V}(T,D;\cL)=0, \quad 
H^2_{C^\infty}(T,D;\cL)=H^2_{C^\infty_V}(T,D;\cL)=0,$$
$$
\dim H_{C^\infty}^1(T,D;\cL)=\dim H_{C^\infty_V}^1(T,D;\cL)=m+1.
$$
\end{theorem}

\begin{proof}
Let $\cE^k_T(D;\cL)$ and $\cE^k_{T_V}(D;\cL)$ be sheaves over $T$ satisfying 
$$H^0(\cE^k_T(D;\cL))=\cE^k(T,D;\cL),\quad 
  H^0(\cE^k_{T_V}(D;\cL))=\cE^k_V(T,D;\cL),
$$
respectively. The natural inclusions 
$$\W^k_T(D;\cL)\hookrightarrow  \cE^k_T(D;\cL),\quad 
 \cE^k_{T_V}(D;\cL)\hookrightarrow  \cE^k_T(D;\cL)
$$
induce quasi isomorphisms between complexes of sheaves
$$
\W^\bu_T(D;\cL)\to \cE_T^\bu(D;\cL),\quad 
 \cE_{T_V}^\bu(D;\cL)\to \cE_T^\bu(D;\cL)
$$
with differential $\na_t$. As in \cite[Appendix]{EV},
we have isomorphisms of hypercohomology groups 
$$\H^j(\W^\bu_T(D;\cL))\simeq \H^j(\cE_T^\bu(D;\cL))\simeq 
\H^j(\cE_{T_V}^\bu(D;\cL)).
$$
Since $\cE_T^k(D;\cL)$ and $\cE_{T_V}^k(D;\cL)$ 
are fine sheaves, $H^j(\cE^k_T(D;\cL))=H^j(\cE^k_{T_V}(D;\cL))=0$ for 
$j\ge 1$. Thus we have 
\begin{align*}
H^j_{alg}(T,D;\cL)=\H^j(\W^\bu_T(D;\cL))
&\simeq\H^j(\cE_T^\bu(D;\cL))=H^j(H^0(\cE_T^\bu(D;\cL)))
=H^j_{C^\infty}(T,D;\cL),\\
&\simeq\H^j(\cE_{T_V}^\bu(D;\cL))=H^j(H^0(\cE_{T_V}^\bu(D;\cL)))
=H^j_{C^\infty_V}(T,D;\cL).
\end{align*}
The rest can be obtained from Proposition \ref{prop:alg-dim} and 
Corollary \ref{cor:cohom-alg-iso}.
\end{proof}

Here we give an expression of the inverse 
\begin{equation}
\label{eq:iso-D}
\imath_D:H^1_{C^\infty}(T,D;\cL)\to H^1_{C^\infty_V}(T,D;\cL)
\end{equation}
of the natural inclusion by following \cite[\S4]{M1}. 
For any element $\f\in \cE^1(T,D;\cL)$, there exists a 
$C^\infty$ function $f_i(t)$ around $x_i$ $(i\in \ID\cup \IZc)$ such that 
$\na_{t}(f_i)=\f$ and $f_i(x_i)=0$ for $x_i\in D$. It admits 
the expression 
$$f_i(t)=\left\{\begin{array}{ccc}
\displaystyle{\frac{1}{u(t)}\int_{x_i}^t u(t)\f} &\textrm{if} & x_i\in D,\\
\displaystyle{\frac{1}{(\l_i-1)u(t)}\int_{\circlearrowleft_i(t)} u(t)\f} 
&\textrm{if} & x_i\in B,\\
\end{array}
\right.
$$
where $\circlearrowleft_i\!\!(t)$ is 
a positively oriented circle with center $x_i$ and terminal $t$.
In case of $\f\in \W^1(T,D;\cL)$, $f_i$ is a meromorphic function 
around $x_i$ and admits the Laurent expansion at $x_i$. 
Though $f_i$ is defined locally, the function 
$$\sum_{i\in \ID\cup \IZc} h_i(t)\cdot f_i(t)$$ 
can be regarded as defined on $T$, and it belongs to $\cE^0(T,D;\cL),$
where $h_i$ is a $C^\infty$ function on $T$ satisfying 
\begin{equation}
\label{eq:function-h}
h_i(t)=\left\{
\begin{array}{ccc}
1 &\textrm{if} & t\in V_i,\\
0 &\textrm{if} & t\in U_i^c,\\
\end{array}
\right.
\end{equation}
for $x_i\in V_i\subset U_i$.
The element 
\begin{equation}
\label{eq:imath-D}
\f- \na_{t}\Big(\sum_{i\in \ID\cup \IZc} h_i(t)\cdot f_i(t)\Big)
\end{equation}
belongs to $\cE^1_{V}(T,D;\cL)$ and represents 
$\imath_D(\f)\in H^1_{C^\infty_V}(T,D;\cL).$ 

We will give a basis of $H^1_{C^\infty}(T,D;\cL)$ 
in \S\ref{sec:dual-cohom}.

\section{Relative twisted dual homology groups}
\label{sec:RTDH}

We set 
\begin{align*}
&T^\vee=\P^1-\{x_i\mid i\in \IZc\cup \ID \}
=\P^1-\left\{\begin{array}{lll}
\{x_{i_0},x_{i_1},\dots,x_{i_r},x_{i_{r+s+1},\dots,x_{i_{m+2}}}\}
&\textrm{if} &\IZc\ne \emptyset,
\\
\{x_{i_1},\dots,x_{i_r}\}
&\textrm{if} &\IZc=\emptyset,
\\
\end{array}
\right.\\
&D^\vee=\{x_i\mid i\in \IP\}=
\left\{\begin{array}{lll}
\{x_{i_{r+1}},\dots,x_{i_{r+s}}\}
&\textrm{if} &\IZc\ne \emptyset,
\\
\{x_{i_0},x_{i_{r+1}},\dots,x_{i_{r+s-1}}\}
&\textrm{if} &\IZc=\emptyset.
\\
\end{array}
\right.
\end{align*}
\begin{remark}
Note that $T^\vee$ (resp. $D^\vee$) is different from the space $T'$ 
(resp. $D'$) defined by the differential $1$-form 
$$\frac{\f_0}{u(t,x)}=\frac{dt}{u(t,x)(t-1)}$$
as in \S\ref{sec:RTH}. For an example, in the case $m=1$ 
and $u(t)=t^0(t-x_1)^0(t-1)^0=1$,  we have 
$$T=\P^1-\{1,\infty\},\  D=\{0,x_1\},\quad 
T^\vee=\P^1-\{0,x_1\},\  D^\vee=\{1,\infty\},
$$
since $u(t)\f_0=\dfrac{dt}{t-1}$. On the other hand, 
$T'$ and $D'$ defined by $\f_0/u(t)$ as in \S\ref{sec:RTH}
are 
$$T'=\P^1-\{1,\infty\}=T,\  D'=\{0,x_1\}=D,$$
since $1/u(t)=u(t)=1$.
\end{remark}
Let $\cL^{\vee}$ be the locally constant sheaf defined by $1/u(t)$.
We define $\CC_k(T^\vee;\cL^\vee)$ by the vector space of 
finite linear combinations of $k$-simplices in $T^\vee$ on which 
a branch of $u(t)^{-1}$ is assigned.
As in the previous section, we have 
an exact sequence of chain complexes 
$$0 \longrightarrow  
\CC_\bullet(D^\vee;\cL^\vee)\longrightarrow 
\CC_\bullet(T^\vee;\cL^\vee)\longrightarrow 
\CC_\bullet(T^\vee,D^\vee;\cL^\vee)\longrightarrow 0
$$
with the boundary operator 
$$\pa^{u^{-1}}:\ell^{u(t)^{-1}|_{\ell}} \mapsto 
\pa(\ell)^{u(t)^{-1}|_{\pa(\ell)}},$$
which induces 
an exact sequence of the twisted homology groups: 
\begin{equation}
\label{eq:ex-seq-deRham}
0 \longrightarrow  
H_1(T^\vee;\cL^\vee)\longrightarrow 
H_1(T^\vee,D^\vee;\cL^\vee)\overset{\pa^{u^{-1}}}{\longrightarrow}
H_0(D^\vee;\cL^\vee)\longrightarrow H_0(T^\vee;\cL^\vee)\longrightarrow 0.
\end{equation}
By Theorem \ref{th:dim}, $H_1(T^\vee,D^\vee;\cL^\vee)$ is 
$m+1$ dimensional for any $\a$.

We define the intersection form between 
$H_1(X^\vee,D^\vee;\cL^\vee)$ and $H_1(X,D;\cL)$ as follows. 
\begin{definition}[The intersection form]
Let elements $\g^u\in H_1(X,D;\cL)$ and $\d^{u^{-1}} 
\in H_1(X^\vee,D^\vee;\cL^\vee)$ be represented by 
$$\sum_i c_i \mu_i^u\in \CC_1(T,D;\cL),\quad 
\sum_j d_j \nu_j^{u^{-1}}\in \CC_1(T^\vee,D^\vee;\cL^\vee),$$
where $c_i,d_j\in \C$, and 
$\mu_i$ and $\nu_j$ are $1$-simplices in $T$ and $T^\vee$, respectively. 
We suppose that if $\mu_i\cap \nu_j\ne \emptyset$ then 
$\mu_i$ and $\nu_j$ intersect transversely at a point $p_{ij}$.
The intersection form $\cI_h$ is defined by 
$$\cI_h(\d^{u^{-1}},\g^u)=
-\sum_{p_{ij}\in \nu_j\cap \mu_i}
(d_j\cdot c_i)\times[\nu_j\cdot \mu_i]_{p_{ij}}
\times(u^{-1}|_{\nu_j}(p_{ij})\cdot u|_{\mu_i}(p_{ij})),
$$
where $[\nu_j\cdot \mu_i]_{p_{ij}}(=\pm1)$ 
is the topological intersection number
of $\mu_i$ and $\nu_j$ at $p_{ij}$, and 
$u|_{\mu_i}(p_{ij})$ and $u^{-1}|_{\nu_j}(p_{ij})$ are 
the values of $u|_{\mu_i}(t)$ and $u^{-1}|_{\nu_j}(t)$ at $p_{ij}$.
\end{definition}

\begin{remark}
If $\a\in (\C-\Z)^{m+3}$ then $\cI_h(\d^{u^{-1}},\g^u)$ is equal to 
the intersection number $\g^u\cdot \d^{u^{-1}}$  
defined in  \cite[\S2.3.3]{AoKi} and \cite[\S4.7]{Y2}.
Pay your attention to the layout of $\g^u$ and $\d^{u^{-1}}$ in our 
intersection form $\cI_h$ and to the construction of the intersection 
matrix $H$ in Proposition \ref{prop:dual-basis}.
There is an advantage of our setting in the study of 
twisted period relations in \S\ref{sec:TPR}.
\end{remark}

We give $m+1$ elements $\d_1^{u^{-1}},\dots,\d_{m+1}^{u^{-1}}$ 
of $H_1(X^\vee,D^\vee;\cL^\vee)$.

\noindent 
(1) In the case $\a\notin \Z^{m+3}$, 
\begin{equation}
\label{eq:hom-d-basis-1}
\d_j^{u^{-1}}=
\left\{\begin{array}
{ccl}
-\circlearrowleft_{i_j}^{u^{-1}} &\textrm{if} & 1\le j\le r,
\\[2mm]
\ell_{i_j}^{u^{-1}}-
\dfrac{\circlearrowleft_{i_{m+2}}^{u^{-1}}}{1-\l_{i_{m+2}}^{-1}}
&\textrm{if} & r+1\le j\le r+s,
\\[4mm]
\circlearrowleft_{i_j}^{u^{-1}}-
\dfrac{1-\l_{i_j}^{-1}}{1-\l_{i_{m+2}}^{-1}}\circlearrowleft_{i_{m+2}}^{u^{-1}}
&\textrm{if} & r+s+1\le j\le m+1.
\end{array}
\right.
\end{equation}
(2) In the case $\a\in \Z^{m+3}$, 
\begin{equation}
\label{eq:hom-d-basis-2}
\d_j^{u^{-1}}=
\left\{\begin{array}
{ccl}
-\circlearrowleft_{i_{j+1}}^{u^{-1}} &\textrm{if} & 1\le j\le r-1,\\[2mm]
\ell_{i_{j+1}}^{u^{-1}}-\ell_{i_{0}}^{u^{-1}}&\textrm{if} & r\le j\le r+s-2=m+1.
\end{array}
\right.
\end{equation}

\begin{proposition}
\label{prop:dual-basis}
The intersection matrix 
$H=\big(\cI_h(\d_i^{u^{-1}},\g_j^u)\big)_{\substack{1\le i\le m+1\\
1\le j\le m+1}}$
is as follows. \\
$(1)$ In the case $\a\notin \Z^{m+3}$, it is

$$
\begin{pmatrix}
 E_r & O   & O \cr
 O   & E_s & H_{32} \cr
 O   & O   & H_{33} \cr
\end{pmatrix},\quad 
H_{32}=\begin{pmatrix}\l_{i_{r+s+1}}-1 & \cdots &\l_{i_{m+1}}-1\\ 
\vdots & \vdots &\vdots \\
\l_{i_{r+s+1}}-1 & \cdots & \l_{i_{m+1}}-1\end{pmatrix},
$$

$$
H_{33}=\begin{pmatrix}
\l_{i_{r+s+1}}-1 &(\l_{i_{r+s+2}}-1)(1-\l_{i_{r+s+1}}^{-1})& \cdots &
(\l_{i_{m+1}}-1)(1-\l_{i_{r+s+1}}^{-1})
\\ 
0 & \l_{i_{r+s+2}}-1 & \cdots &(\l_{i_{m+1}}-1)(1-\l_{i_{r+s+2}}^{-1}) 
\\ 
\vdots & O &\ddots & \vdots
\\
0& 0 & \cdots & \l_{i_{m+1}}-1\end{pmatrix},
$$
where $E_r$ is the unit matrix of size $r$. \\
$(2)$ In the case $\a\in \Z^{m+3}$, it is the unit matrix $E_{m+1}$.\\ 
The elements $\d_1^{u^{-1}},\dots,\d_{m+1}^{u^{-1}}$ form a basis of 
$H_1(X^\vee,D^\vee;\cL^\vee)$. 
\end{proposition}

\begin{proof}
We can easily evaluate $\cI_h(\d_h^{u^{-1}},\g_i^u)$ with  
following variations of branches of $u(t)$ and $u(t)^{-1}$.
The regularity of the intersection matrix shows that 
$\d_1^{u^{-1}},\dots,\d_{m+1}^{u^{-1}}$ 
is a basis of $H_1(X^\vee,D^\vee;\cL^\vee)$.
\end{proof}

\section{Relative twisted dual cohomology groups}
\label{sec:dual-cohom}
The relative twisted dual cohomology group $H^k(X^\vee,D^\vee;\cL^\vee)$ 
is defined as the dual space of $H_k(X^\vee,D^\vee;\cL^\vee)$. 
We define three complexes 
$\W^\bu(T^\vee,D^\vee;\cL^\vee)$, $\cE^\bu(T^\vee,D^\vee;\cL^\vee)$, 
and $\cE^\bu_V(T^\vee,D^\vee;\cL^\vee)$ by changing the roles in 
$\W^\bu(T,D;\cL)$, $\cE^\bu(T,D;\cL)$ and $\cE^\bu_V(T,D;\cL)$ as 
$$T\to T^\vee,\quad D\to D^\vee,\quad  u(t) \to 1/u(t), \quad 
\na_t \to \na_t^\vee=d-\w\wedge;$$
that is, 
\begin{align*}
\W^0(T^\vee,D^\vee;\cL^\vee)&
=\{f(t)\in \W^0(\cx)\mid \ord_{x_i}f(t)/u(t)\ge 1
\textrm{ for any }x_i\in D^\vee\},\\
\W^1(T^\vee,D^\vee;\cL^\vee)&
=\{\f(t)\in \W^0(\cx)\mid \ord_{x_i}\f(t)/u(t)\ge 0
\textrm{ for any }x_i\in D^\vee\},\\
\cE^0(T^\vee,D^\vee;\cL^\vee)&=\{\f(t)\in \cE^0(\cx)\mid 
\begin{array}{l}
f(t)/u(t)\textrm{ is } C^\infty \textrm{ on } U_i \\
\lim\limits_{t\to x_i} f(t)/u(t) =0
\end{array}
\textrm{ for any } i\in \IP\},\\
\cE^k(T^\vee,D^\vee;\cL^\vee)&=\{\f(t)\in \cE^k(\cx)\mid 
\f(t)/u(t)\textrm{ is } C^\infty \textrm{ on } U_i 
\textrm{ for any } i\in \IP\} \ (k=1,2),\\
\cE^k_V(T^\vee,D^\vee;\cL^\vee)&=\{\f(t)\in \cE^k(\cx)\mid 
\f(t) \textrm{ is identically }0\textrm{ on } V_i \textrm{ for any } 
i\in \IP\cup \IZc\}.
\end{align*}
We have relative twisted de Rham dual cohomology groups 
$$
H_{alg}^k(T^\vee,D^\vee;\cL^\vee),\quad 
H_{C^\infty}^k(T^\vee,D^\vee;\cL^\vee),\quad 
H_{C^\infty_V}^k(T^\vee,D^\vee;\cL^\vee),
$$
as the $k$-th cohomology groups  of the complexes 
$\W^\bu(T^\vee,D^\vee;\cL^\vee)$, $\cE^\bu(T^\vee,D^\vee;\cL^\vee)$, 
and $\cE^\bu_V(T^\vee,D^\vee;\cL^\vee)$, respectively.
There is a natural pairing between $H_1(X^\vee,D^\vee;\cL^\vee)$ and 
$H_{C^\infty}^1(T^\vee,D^\vee;\cL^\vee)$ (resp. 
$H_{alg}^1(T^\vee,D^\vee;\cL^\vee)$ and  
$H_{C^\infty_V}^1(T^\vee,D^\vee;\cL^\vee)$)
defined by
\begin{equation}
\label{eq:dual-pairing}
\la \d^{u^{-1}}, \psi \ra=\int_{\d} \frac{\psi}{u(t)}
\end{equation}
for $\d^{u^{-1}}\in H_1(X^\vee,D^\vee;\cL^\vee)$ and 
$\psi\in H_{C^\infty}^1(T^\vee,D^\vee;\cL^\vee)$.
Here note that the layout of $\d^{u^{-1}}$ and $\psi$ 
in this pairing is different from that in (\ref{eq:dual-cohom-hom}).
As is shown in \S\ref{sec:RTC}, $H^k(X^\vee,D^\vee;\cL^\vee)$ 
is canonically isomorphic to 
$H_{alg}^k(T^\vee,D^\vee;\cL^\vee)$,
$H_{C^\infty}^k(T^\vee,D^\vee;\cL^\vee)$ and 
$H_{C^\infty_V}^k(T^\vee,D^\vee;\cL^\vee)$.

\begin{definition}[Intersection form between 
$H^1(T,D;\cL)$ and $H^1(T^\vee,D^\vee;\cL^\vee)$]
The intersection form $\cI_c$ between 
$H^1(T,D;\cL)$ and $H^1(T^\vee,D^\vee;\cL^\vee)$ is 
defined by 
$$
\cI_c(\f,\psi) =\iint_{T\cap T^\vee} \f\wedge \psi
$$
for $\f\in H^1_{C^\infty_V}(T,D;\cL)$ and 
$\psi\in H^1_{C^\infty_V}(T^\vee,D^\vee;\cL^\vee)$.
\end{definition}

Though the supports of 
$\f\in H^1_{C^\infty_V}(T,D;\cL)$ and 
$\psi\in H^1_{C^\infty_V}(T^\vee,D^\vee;\cL^\vee)$ are not necessarily 
compact, $\f\wedge \psi$ is a $C^\infty$ $2$-form with 
a compact support included in $\P^1-\bigcup_{i=0}^{m+2} V_i$.
Thus the intersection form $\cI_c(\f,\psi)$ is well-defined.
Even in the case where
$$\iint_{T\cap T^\vee} \f\wedge \psi$$
is not well-defined for elements $\f\in H_{C^\infty}^1(T,D;\cL)$ and 
$\psi\in H_{C^\infty}^1(T^\vee,D^\vee;\cL^\vee)$, 
$\cI_c(\f,\psi)$ is always defined 
by elements $\f'\in \cE_{V}^1(T,D;\cL)$ and 
$\psi'\in \cE_{V}^1(T^\vee,D^\vee;\cL^\vee)$
cohomologous to $\f$ and $\psi$ as elements of 
$H_{C^\infty}^1(T,D;\cL)$ and 
$H_{C^\infty}^1(T^\vee,D^\vee;\cL^\vee)$, respectively.
In particular, we have the following.

\begin{theorem}
\label{th:intersection-cohom}
The isomorphisms $\imath_D$ 
in (\ref{eq:iso-D}) and 
$$
\imath_{D^\vee}:H^1_{C^\infty}(T^\vee,D^\vee;\cL^\vee)\to 
H^1_{C^\infty_V}(T^\vee,D^\vee;\cL^\vee)
$$
induce the intersection form between $H^1_{alg}(T,D;\cL)$ 
and $H_{alg}^1(T^\vee,D^\vee;\cL^\vee)$, which is expressed as  
\begin{align*}
&\cI_c(\f,\psi) =\cI_c(\imath_D(\f),\imath_{D^\vee}(\psi)) =
\iint_{T\cap T^\vee} \imath_D(\f)\wedge \imath_{D^\vee}(\psi)\\
=&
2\pi \sqrt{-1} 
\left(
\sum_{i\in \ID} \res_{x_i} (f_i\cdot \psi)
-\sum_{i\in \IP}\res_{x_i} (g_i\cdot \f)
+\frac{1}{2}\sum_{i\in \IZ^c}\res_{x_i} (f_i \cdot \psi-g_i\cdot \f)
\right),
\end{align*}
where $\f\in H^1_{alg}(T,D;\cL)$,
$\psi\in H_{alg}^1(T^\vee,D^\vee;\cL^\vee)$, 
$\na_{t} f_i=\f$ around $x_i$ for $i\in \IZc\cup \ID$, 
$f_i(x_i)=0$ for $i\in \ID$,
$\na_t^\vee g_i =\psi$ around $x_i$ for $i\in \IZc\cup \IP$, 
$g_i(x_i)=0$ for $i\in \IP$, and 
$\res_{x_i}(\eta)$ denotes the residue of a meromorphic $1$-form $\eta$
at $t=x_i$.
\end{theorem}

\begin{proof}
By the expression (\ref{eq:imath-D}), we see that 
the support of $\imath_D(\f)\wedge \imath_{D^\vee}(\psi)$ 
is included in the closure of $\bigcup\limits_{i=0}^{m+2} (U_i-V_i)$. 
The restriction of $\imath_D(\f)\wedge \imath_{D^\vee}(\psi)$ to $U_i-V_i$ 
becomes as follows:
$$
\begin{array}{lll}
 i\in \ID &\Rightarrow&\quad
(\f-\na_{t} (h_if_i))\wedge \psi=-d(h_if_i)\wedge \psi=-d(h_if_i\psi),
\\
 i\in \IP &\Rightarrow&\quad 
\f\wedge(\psi-\na_t^\vee(h_ig_i))=-\f\wedge d(h_ig_i)=d(h_ig_i\f),
\\
 i\in \IZc &\Rightarrow&\quad 
(\f-\na_{t} (h_if_i))\wedge(\psi-\na_t^\vee(h_ig_i))\\
&&=d(-h_if_i\psi+h_ig_i\f)+\na_t(h_if_i)\wedge \na_t^\vee(h_ig_i)\\
&&=d(-h_if_i\psi+h_ig_i\f)+(f_idh_i+h_i\f)\wedge (g_idh_i+h_i\psi)\\
&&=d(-h_if_i\psi+h_ig_i\f)+(h_idh_i)\wedge (f_i\psi)+(g_i\f)\wedge (h_idh_i)\\
&&=d\big(h_i(-f_i\psi+g_i\f)\big)
+\dfrac{1}{2}d \big(h_i^2(f_i\psi-g_i\f)\big).
\end{array}
$$
Thus we have 
\begin{align*}
 &\iint_{U_i-V_i} \imath_D(\f)\wedge \imath_{D^\vee}(\psi) \\
=&\left\{
\begin{array}{lcc}
\displaystyle{\iint_{U_i-V_i} -d(h_if_i\psi)=\int_{\pa V_i} f_i\psi}, 
&\textrm{if}& i\in \ID,\\
\displaystyle{\iint_{U_i-V_i} d(h_ig_i\f)=\int_{\pa V_i}-g_i\f}, 
&\textrm{if}& i\in \IP,\\
\displaystyle{\iint_{U_i-V_i} d\Big((1-\dfrac{h_i}{2})h_i(-f_i\psi+g_i\f)\Big)
=\int_{\pa V_i}\dfrac{1}{2}(f_i\psi-g_i\f)},
&\textrm{if}& i\in \IZc,
\end{array}
\right.
\end{align*}
by Stokes' theorem, since $h_i$ is identically $1$ on $\pa(V_i)$ and 
identically $0$ on $\pa(U_i)$. Apply the residue theorem to the 
integrals along $\pa V_i$. 
\end{proof}

\begin{remark}
\label{rem:res-nonint}
Since 
$\res_{x_i} (f_i \cdot \psi)=\res_{x_i} (-g_i\cdot \f)$ for 
$i\in \IZc$, we have 
\begin{align*}
\cI_c(\f,\psi) &=
2\pi \sqrt{-1} 
\left(
\sum_{i\in \ID\cup \IZc} \res_{x_i} (f_i\cdot \psi)
-\sum_{i\in \IP}\res_{x_i} (g_i\cdot \f)\right)\\
&=
2\pi \sqrt{-1} 
\left(
\sum_{i\in \ID} \res_{x_i} (f_i\cdot \psi)
-\sum_{i\in \IP\cup \IZc}\res_{x_i} (g_i\cdot \f)\right).
\end{align*}
\end{remark}

We give $(m+2)$ elements $\f_{i,m+2}$ $(0\le i\le m+1)$ of 
$H^1_{C^\infty}(T,D;\cL)$ by 
\begin{align}
\nonumber 
\f_{i,m+2}&=\left\{
 \begin{array}{ccc}
\dfrac{\a_idt}{t-x_i} &\textrm{if} & \a_i\ne 0,\\
\dfrac{-u(x_i)dh_i(t)}{u(t)} &\textrm{if} & \a_i= 0,
\end{array}
\right. 
\ (0\le i\le m),\\[-3mm]
\label{eq:frame}
\\
\nonumber
\f_{m+1,m+2}&=\f_0=\frac{dt}{t-1}.
\end{align}
We check that $\f_{i,m+2}\in \cE^1(T,D;\cL)$ for $0\le i\le m+1$. 
If $x_i$ $(0\le i\le m)$ belongs to $D$, then $\a_i=0$ or $\a_i\in \N$, 
and we can see that $u(t)\f_{i,m+2}$ is smooth around $x_i$ in both cases.
If $x_{m+1}=1$ belongs to $D$, then $\a_{m+1}\in \N$  
and $u(t)\f_{m+1,m+2}$ is holomorphic around $x_{m+1}=1$.
If $x_{m+2}=\infty$ belongs to $D$, 
then $\a_{m+2}\in \N$, and 
$u(t)\f_{i,m+2}$ $(0\le i\le m+1)$ is smooth around $x_{m+2}=\infty$, 
since $\f_{i,m+2}$ has a simple pole at $t=\infty$ or vanishes identically
around $t=\infty$. Note that if $\a_i=0$ then $u(t)$ is non-zero holomorphic 
around $t=x_i$. Thus 
$\f_{i,m+2}$ belongs to $\cE^1(T,D;\cL)$ in any cases. 
It is clear that $\na_t\f_{m+1,m+2}=0$ and 
$\na_t\f_{i,m+2}=0$ for $\a_i\ne 0$. For $\a_i=0$, we have 
$$
\na_t\f_{i,m+2}=-u(x_i)\na_t\big( \frac{1}{u(t)}\cdot dh_i(t)\big)
=-u(x_i)\Big(\big(\na_t\frac{1}{u(t)}\big)\wedge dh_i(t)+
\frac{d(dh_i(t))}{u(t)}
\Big)=0
$$
since $\na_t \dfrac{1}{u(t)}=0$. Thus $\f_{i,m+2}$'s represent 
elements of $H^1_{C^\infty}(T,D;\cL)$.

We also give $(m+2)$ elements $\psi _{0,i}$ $(1\le i\le m+2)$
of $H_{C^\infty}^1(T^\vee,D^\vee;\cL^\vee)$ by 
\begin{align}
\nonumber
\psi _{0,i}&=\dfrac{dt}{t-x_i}-\dfrac{dt}{t}\ (1\le i\le m),
\\
\label{eq:dual-frame}
\psi _{0,m+1}&=\left\{
 \begin{array}{ccc}
\a_{m+1}\left(\dfrac{dt}{t-1}-\dfrac{dt}{t}\right) &\textrm{if} 
& \a_{m+1}\ne 0,\\
\dfrac{u(t)dh_{m+1}(t)}{u(1)} &\textrm{if} & \a_{m+1}= 0,
\end{array}
\right.
\\  
\nonumber
\psi _{0,m+2}&=\left\{
 \begin{array}{ccc}
-\a_{m+2}\dfrac{dt}{t}&\textrm{if} & \a_{m+2}\ne 0,\\
u(t)dh_{m+2}(t) &\textrm{if} & \a_{m+2}= 0. 
\end{array}
\right.
\end{align}
As shown previously, 
we can check that $\psi _{i,m+2}\in H^1(T^\vee,D^\vee;\cL^\vee)$ 
for $1\le i\le m+2$. Here we use the property $\na_t^\vee u(t)=0$.

\begin{theorem}
\label{th:cint-number}
\begin{enumerate}
\item For $1\le i,j\le m+1$, we have
$$
\cI_c(\f_{i,m+2},\psi _{0,j})=2\pi\sqrt{-1}\d_{[i,j]}, 
$$
where $\d_{[i,j]}$ denotes Kronecker's symbol.
In particular, 
$\f_{i,m+2}$'s and $\psi _{0,i}$'s $(1\le i\le m+1)$ are 
bases of $H^1_{C^\infty}(T,D;\cL)$ and 
$H^1_{C^\infty}(T^\vee,D^\vee;\cL^\vee)$, respectively.
\item We have 
$$
\cI_c(\f_{0,m+2},\psi _{0,j})=-2\pi\sqrt{-1}\ (1\le j\le m),\quad 
\cI_c(\f_{0,m+2},\psi _{0,m+1})=-2\pi\sqrt{-1}\a_{m+1}.
$$
In particular, there is a linear relation 
$$\f_{0,m+2}+\sum_{i=1}^m \f_{i,m+2}+\a_{m+1}\f_{m+1,m+2}
=\w+\sum_{i\in \ID,\a_i=0} \f_{i,m+2}=0
$$
as elements of $H^1_{C^\infty}(T,D;\cL)$, and $m+1$ elements 
$$\f_{0,m+2},\ \dots\  ,\ \f_{i-1,m+2},\ \f_{i+1,m+2},\ \dots\ ,\ 
\f_{m+1,m+2} \quad 
(1\le i\le m)$$
are linearly independent.
\item We have 
$$
\cI_c(\f_{i,m+2},\psi _{0,m+2})=-2\pi\sqrt{-1}\a_i\ (1\le i\le m),\quad 
\cI_c(\f_{m+1,m+2},\psi _{0,m+2})=-2\pi\sqrt{-1}.
$$
In particular, there is a linear relation 
$$\sum_{i=1}^m \a_i\psi _{0,i}+\psi _{0,m+1}+\psi _{0,m+2}
=\w+\sum_{i\in \IP,\a_i=0}\psi _{0,i}=0
$$
as elements of $H^1_{C^\infty}(T^\vee,D^\vee;\cL^\vee)$.
\end{enumerate}
\end{theorem}
\begin{proof}
(1) In case of $\a_i\ne 0$, 
we use Remark \ref{rem:res-nonint}.
At least one of differential equations 
$\na_t f(t)=\f_{i,m+2}$ or $\na_t^\vee g(t)=\psi_{0,j}$ 
admits a meromorphic local solution $f_k(t)$ or $g_k(t)$ around 
$x_k$ $(0\le k\le m+2)$. 
If $f_k(t)$ (or $g_k(t)$) is a solution, then it satisfies
$$\ord_{x_k} f_k(t)=1+\ord_{x_k} \f_{i,m+2}(t)\quad  
(\textrm{or}\quad \ord_{x_k} g_k(t)=1+\ord_{x_k} \psi _{0,j}(t)).$$
Since $\f_{i,m+2}$ and $\psi _{0,j}$ admit simple poles only on 
$t=x_i,\infty$ and on $t=0,x_j$, 
if $i\ne j$ then $\res_{x_k}f_k(t)\psi _{0,j}=0$ or 
$\res_{x_k}g_k(t)\f_{i,m+2}=0$ holds. 
Thus we have $\cI_c(\f_{i,m+2},\psi _{0,j})=0$ for $i\ne j$.
In case of $i=j=k$, 
if $f_i(t)$ (or $g_i(t)$) is a solution, then it takes a form 
$$f_i(t)=\left\{
\begin{array}{lcc} 
1+O(t-x_i) &\textrm{if}& 1\le i\le m,\\
\frac{1}{\a_{m+1}}+O(t-1) &\textrm{if}&  i=m+1,\\
\end{array}
\right.
$$
$$
\Big(\textrm{or}\quad  
g_i(t)=\left\{
\begin{array}{ccc} 
-\frac{1}{\a_i}+O(t-x_i) &\textrm{if}& 1\le i\le m,\\
-1+O(t-1) &\textrm{if}&  i=m+1,\\
\end{array}
\right.
\Big)
$$
where $O$ denotes Landau's symbol.
The intersection number $\cI_c(\f_{i,m+2},\psi _{0,i})$ is equal to 
$2\pi\sqrt{-1}$ times 
$\res_{x_k}f_k(t)\psi _{0,j}$ or $-\res_{x_k}g_k(t)\f_{i,m+2}$; 
it becomes $2\pi\sqrt{-1}$ in both cases. 

In case of $\a_i=0$ $(1\le i\le m)$, note that 
$$\f_{i,m+2}\wedge \psi _{0,j}=-u(x_i)dh_i(t)\wedge \frac{\psi_{0,j}}{u(t)}
=-u(x_i)d\Big(
h_i(t)\frac{\psi_{0,j}}{u(t)}
\Big),
$$
and that its support is the closure of $U_i-V_i$.
By Stokes' theorem and the residue theorem, we have
\begin{align*}
\cI_c(\f_{i,m+2},\psi _{0,j})&=
\iint_{U_i-V_i} \f_{i,m+2}\wedge \psi _{0,j}
=-u(x_i)\int_{\pa(U_i-V_i)}h_i(t)\frac{\psi_{0,j}}{u(t)}
=u(x_i)\int_{\pa V_i}\frac{\psi_{0,j}}{u(t)}\\
&=2\pi\sqrt{-1}\cdot  u(x_i) \cdot \res_{x_i} \frac{\psi_{0,j}}{u(t)}
=\left\{
\begin{array}{ccc}
2\pi\sqrt{-1} &\textrm{if} & i=j,\\
0 &\textrm{if} & i\ne j.\\
\end{array}
\right.
\end{align*}
Here note that $u(t)$ is non-zero holomorphic around $x_i$ by $\a_i=0$.
In case of $\a_{m+1}=0$, we can similarly show.

Since $H^1_{C^\infty}(T,D;\cL)$ and $H^1_{C^\infty}(T^\vee,D^\vee;\cL^\vee)$ 
are $m+1$ dimensional, $\f_{i,m+2}$'s and $\psi_{0,j}$'s 
are bases of these spaces.

\smallskip\noindent
(2) We can evaluate the intersection number similarly to (1). 
We can express $\f_{0,m+2}$ as a linear combination: 
$$\f_{0,m+2}=c_1\f_{1,m+2}+\cdots +c_m\f_{m,m+2}+c_{m+1}\f_{m+1,m+2}.$$
By comparing $\cI_c(\f_{0,m+2},\psi _{0,j})$ with 
$$\cI_c(\sum_{i=1}^{m+1} c_i\f_{i,m+1},\psi _{0,j}),$$
we have $c_1=\cdots=c_m=-1$ and $c_{m+1}=-\a_{m+1}$.
Thus we obtain the linear relation. This relation is also obtained by 
the property that 
$$\na_t\big(1-\sum_{i\in \ID} u(x_i)\frac{h_i(t)}{u(t)}\big)
$$
is the zero of $H^1_{C^\infty}(T,D;\cL)$.
In fact, since
$$1-\sum_{i\in \ID} u(x_i)\frac{h_i(t)}{u(t)}\in \cE^0(T,D;\cL),\quad  
\na_t1=\w,$$ 
$$\na_t\big(u(x_i)\frac{h_i(t)}{u(t)}\big)
=u(x_i)h_i(t)\na_t\big(\frac{1}{u(t)}\big)
+u(x_i)\frac{dh_i(t)}{u(t)}
=u(x_i)\frac{dh_i(t)}{u(t)},
$$
we have 
$$0=\na_t\big(1-\sum_{i\in \ID} u(x_i)\frac{h_i(t)}{u(t)}\big)
=\w-\sum_{i\in \ID} u(x_i)\frac{dh_i(t)}{u(t)}
=\w+\sum_{i\in \ID,\a_i=0} \f_{i,m+2}.
$$
Here note that $u(x_i)=0$ for $i\in \ID$ with $\a_i\ne 0$. 

For $1\le i\le m$, $\f_{i,m+2}$ can be expressed as 
a linear combination of the others, 
we have the linear independence of the $m+1$ elements.

\smallskip\noindent
(3) We can show the claims similarly to (2).
\end{proof}

\begin{remark}
\begin{enumerate}
\item 
By Theorem \ref{th:cint-number} (2), the $m+1$ elements 
$\f_{0,m+2},\f_{1,m+2},\dots,\f_{m,m+2}$ 
are linearly dependent if $\a_{m+1}=0$.
\item 
By Theorem \ref{th:cint-number} (3), the $m+1$ elements 
$\psi _{0,1},\dots,\psi _{0,i-1},\psi _{0,i+1}\dots,
\psi _{0,m+1}, \psi _{0,m+2}$ 
are linearly dependent if $\a_{i}=0$ for $1\le i\le m$.
\end{enumerate}
\end{remark}

\begin{proposition}
\label{prop:relative-forms}
\begin{enumerate}
\item If $\a_i=0$ for $0\le i\le m$ then $\f_{i,m+2}$ is cohomologous to 
$$
\na_t\Big(\prod_{j\in \ID}^{j\ne i} 
\big(\frac{t-x_j}{x_i-x_j}\big)^{1-\a_j}\Big )\in \W^1(T,D;\cL)
$$
as elements of $H^1_{C^\infty}(T,D;\cL)$.

\item Suppose that one of $\a_{m+1}$ and $\a_{m+2}$ is $0$.  
If $\a_i=0$ $(i=m+1,m+2)$ 
then $\psi _{0,i}$ is cohomologous to $-\w\in 
\W^1(T^\vee,D^\vee;\cL^\vee)$
as elements of $H^1_{C^\infty}(T^\vee,D^\vee;\cL^\vee)$.

\item 
If $\a_{m+1}=\a_{m+2}=0$ then $\psi _{0,m+1}$, 
$\psi _{0,m+2}$ are cohomologous to 
$$
\na_t^\vee\big(\frac{1-x_j}{t-x_j}\big),\  
\na_t^\vee\big(\frac{t-1}{t-x_j}\big)\in \W^1(T^\vee,D^\vee;\cL^\vee)
$$
as elements of 
$H^1_{C^\infty}(T^\vee,D^\vee;\cL^\vee)$, respectively,
where $j$ is an element of $\IZc\cup \ID$. 
\end{enumerate}
\end{proposition}

\begin{proof}
(1) If $\a_i=0$ then 
$$\la \f_{i,m+2},\g^u\ra=\int_\g u(t) u(x_i)\frac{dh_i(t)}{u(t)}
=\big[u(x_i)h_i(t)\big]_{\pa(\g)}.
$$
Thus it vanishes for $\g^u\in H^1(T,D;\cL)$ such that $x_i\notin \pa \g$, 
and becomes $u(x_i)$ for $\g^u\in H^1(T,D;\cL)$ 
with a path $\g$ ending at $x_i$. On the other hand, we have
\begin{align*}
\la \na_t\Big(\prod_{j\in \ID}^{j\ne i} 
\big(\frac{t-x_j}{x_i-x_j}\big)^{1-\a_j}\Big ),\g^u\ra
&=\int_\g u(t)\na_t \Big(\prod_{j\in \ID}^{j\ne i} 
\big(\frac{t-x_j}{x_i-x_j}\big)^{1-\a_j}\Big)\\
&=\Big[u(t)\prod_{j\in \ID}^{j\ne i} 
\big(\frac{t-x_j}{x_i-x_j}\big)^{1-\a_j}\Big]_{\pa\g}.
\end{align*}
The last term vanishes for $\g^u\in H^1(T,D;\cL)$ such that 
$x_i\notin \pa \g$, and becomes $u(x_i)$ for $\g^u\in H^1(T,D;\cL)$
with a path $\g$ ending at $x_i$. 
Here we regard 
$(\dfrac{t-x_{m+2}}{x_i-x_{m+2}}\big)^{1-\a_{m+2}}$ as $1$ 
when $m+2\in \ID$.
Hence $\f_{i,m+2}$ is cohomologous to this algebraic $1$-form 
as elements of $H^1_{C^\infty}(T,D;\cL)$.

\smallskip\noindent
(2) The assertion is obvious from Theorem \ref{th:cint-number} (3).

\smallskip\noindent
(3)
If $\a_{m+1}=\a_{m+2}=0$ then 
we have
$\lim\limits_{t\to\infty}u(t)=1$, 
$$\Big[\frac{1-x_j}{t-x_j}\Big]_{t=1}=1,\quad 
\Big[\frac{1-x_j}{t-x_j}\Big]_{t=\infty}=0,\quad 
\Big[\frac{t-1}{t-x_j}\Big]_{t=1}=0,\quad 
\Big[\frac{t-1}{t-x_j}\Big]_{t=\infty}=1,
$$
\begin{align*}
\la \g^{u^{-1}},\psi_{0,m+1}\ra&=\int_\g \frac{1}{u(t)} 
\frac{u(t)dh_{m+1}(t)}{u(1)}
=\Big[\frac{h_{m+1}(t)}{u(1)}\Big]_{\pa(\g)},\\
\la \g^{u^{-1}}, \na_t^\vee\Big(\frac{1-x_j}{t-x_j}\Big)\ra
&=\int_\g \frac{1}{u(t)} 
\na_t^\vee\Big(\frac{1-x_j}{t-x_j}\Big)
=\Big[\frac{1}{u(t)}\cdot \frac{1-x_j}{t-x_j}\Big]_{\pa(\g)},\\
\la \g^{u^{-1}},\psi_{0,m+2}\ra&=\int_\g \frac{1}{u(t)} 
u(t)dh_{m+1}(t)
=\Big[h_{m+2}(t)\Big]_{\pa(\g)},\\
\la \g^{u^{-1}}, \na_t^\vee\Big(\frac{t-1}{t-x_j}\Big)\ra
&=\int_\g \frac{1}{u(t)} 
\na_t^\vee\Big(\frac{t-1}{t-x_j}\Big)
=\Big[\frac{1}{u(t)}\cdot \frac{t-1}{t-x_j}\Big]_{\pa(\g)}.
\end{align*}
Note that if $x_k\in D^\vee$ $(0\le k\le m)$ then $\a_k\in -\N$ and 
$1/u(x_k)=0$. Hence we have 
$$\la \g^{u^{-1}},\psi_{0,m+1}\ra=
\la \g^{u^{-1}}, \na_t^\vee\Big(\frac{1-x_j}{t-x_j}\Big)\ra,
\quad 
\la \g^{u^{-1}},\psi_{0,m+2}\ra
=\la \g^{u^{-1}}, \na_t^\vee\Big(\frac{t-1}{t-x_j}\Big)\ra, 
$$
which yield the assertion. 
\end{proof}

\section{Twisted period relations}
\label{sec:TPR}
In this section, we show the compatibility of the pairings 
between relative twisted homology and cohomology groups and the intersection 
forms $\cI_h$ and $\cI_c$. 

\begin{theorem}
\label{th:compatible}
The intersection form $\cI_c$ is compatible with $\cI_h$ 
through the isomorphisms $H^1(T,D;\cL)\simeq H_1(T^\vee, D^\vee;\cL^\vee)$ 
and $H^1(T^\vee,D^\vee;\cL^\vee)\simeq H_1(T, D;\cL)$. 
\end{theorem}
\begin{proof}
By the perfectness of the pairing between 
$H_1(T,D;\cL)$ and $H^1(T,D;\cL)$, and that of $\cI_h$ 
between $H_1(T,D;\cL)$ and $H_1(T^\vee,D^\vee;\cL^\vee)$, 
there exists an isomorphism 
\begin{equation}
\label{eq:PD-iso}
\kappa:H_1(T,D;\cL)\to H^1(T^\vee,D^\vee;\cL^\vee)
\end{equation}
such that
$$\la \f,\g^u\ra=\cI_c(\f,\kappa(\g^u))$$
for any $\f\in H^1(T,D;\cL)$. 
We show that this isomorphism also satisfies 
$$\cI_h(\d^{u^{-1}}, \g^u)=\la \d^{u^{-1}},\kappa(\g^u)\ra $$
for any $\d^{u^{-1}}\in H_1(T^\vee,D^\vee;\cL^\vee)$.

For a twisted cycle $\circlearrowleft_i^u$ $(i\in \IP)$ and any element 
$\f\in H^1_{C^\infty_V}(T,D;\cL)$, we have 
$$
\la \f,\circlearrowleft_i^u\ra =
\int_{\pa V_i} u(t)\f=\iint_{U_i-V_i}  u(t)\f \wedge dh_i(t)
=\iint_{{T\cap T^\vee}} \f\wedge \zeta_i^\vee=\cI_c(\f, \zeta_i^\vee),
$$
where $\zeta_i^\vee=\na_t^\vee(u(t)h_i(t))=u(t)dh_i(t)
\in H^1_{C^\infty_V}(T,D;\cL)$.
Thus we have $\kappa(\circlearrowleft_i^u)=\zeta_i^\vee$.
For  $\d^{u^{-1}}_j\in H_1(T^\vee,D^\vee;\cL^\vee)$ 
$(1\le j\le m+1)$ given in 
(\ref{eq:hom-d-basis-1}) or (\ref{eq:hom-d-basis-2}),
we have 
$$
\la \d^{u^{-1}}_j,\kappa(\circlearrowleft_i^u)\ra=
\la \d^{u^{-1}}_j,\zeta_i^\vee\ra=\int_{\d_j} \frac{\zeta_i^\vee}{u(t)} 
=\int_{\d_j} dh_i(t) 
=\big[h_i(t)\big]_{\pa(\d_j)}=
\cI_h(\d^{u^{-1}}_j,\circlearrowleft_i^u).
$$

Let $\g^u\in H_1(T,D;\cL)$ be represented by 
a path $\g$ connecting $x_i$ and $x_j$ $(i,j\in \IZc\cup \ID)$ 
with a branch of $u(t)$ on it. 
Though we consider a small circle with center $x_i$ for $x_i\in B$  
in our construction of a basis of $H_1(T,D;\cL)$, 
we may ignore it since $\f\in H^1_{C^\infty_V}(T,D;\cL)$ 
is identically $0$ around $x_i$.
We define a $C^\infty$ function $h_\g(t)$ on $\P^1-\g$ 
satisfying 
$$h_\g(t)=\left\{
\begin{array}{lll}
1 &\textrm{if} &V_\g, \\
0 &\textrm{if} &U_\g^c,
\end{array}
\right.
$$
where open sets $V_\g\subset U_\g$ are in the right side of $\g$ 
with respect to its orientation, see Figure \ref{fig:VUgamma}.
\begin{figure}[htb]
\includegraphics[width=6cm]{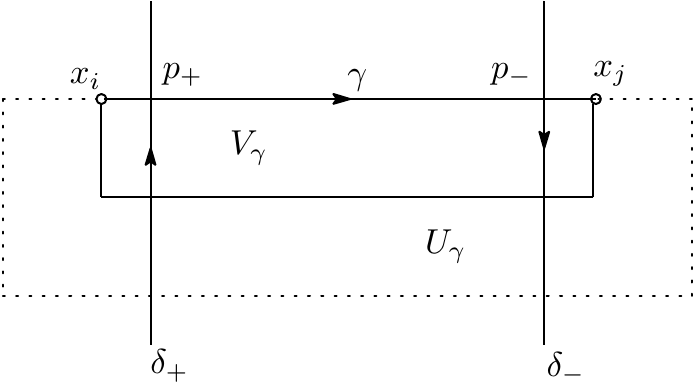}
\caption{Regions with respect to a twisted cycle $\g^u$}
\label{fig:VUgamma}
\end{figure}
We define $\zeta_\g^\vee\in \cE^1(T^\vee,D^\vee;\cL^\vee)$ by
$$\zeta_\g^\vee=\na_t^\vee(u(t)h_\g(t))=
u(t) dh_\g(t)+h_\g(t) \na_t^\vee(u(t))
=u(t)dh_\g(t).$$
Here note that $u(t)=u|_{\g}(t)$ is a (single valued) branch on $U_\g$, 
$u(t)h_\g(t)$ can be regarded as single valued on $T^\vee-\g$, and  
$\zeta_\g^\vee$ is extended to $0$ on $\g$ and it is $C^\infty$ on $T^\vee$.
Then we have 
$$\iint_{{T\cap T^\vee}} \f\wedge \zeta_\g^\vee
=\iint_{U_\g} \f\wedge \zeta_\g^\vee
=\int_{\pa U_\g} -(u(t)h_\g(t))\cdot \f
=\int_{\g} u(t)\f=\la \f,\g^u\ra,
$$
by Stokes' theorem since
\begin{align*} 
&-d((u(t) h_\g(t))\cdot \f)=-(du(t) h_\g(t)+u(t)dh_\g(t))
\wedge \f-(u(t) h_\g(t))d\f\\
=&\f\wedge u(t)dh_\g(t)
-u(t)h_\g(t)(\w \wedge \f+d\f)=\f\wedge \zeta_\g^\vee-u(t)h_\g(t)\na_t\f
=\f\wedge \zeta_\g^\vee.
\end{align*}
Thus we have $\kappa(\g^u)=\zeta_\g^\vee$.

Let $\d^{u^{-1}}_{\pm}$ be elements of $H_1(T^\vee,D^\vee;\cL^\vee)$ 
with paths $\d_+$ and $\d_-$ intersecting $\g$ with topological intersection 
number $+1$ and $-1$, respectively.
We assume that the branches $u|_{\d_{\pm}}(t)^{-1}$ on $\d_{\pm}$ satisfy 
$u|_{\g}(p_{\pm})\cdot u|_{\d_{\pm}}(p_{\pm})^{-1}=1$ at 
the intersection points $p_{\pm}=\g\cap \d_{\pm}$. We have 
$$\int_{\d_{\pm}} u|_{\d_{\pm}}(t)^{-1} \zeta_\g^\vee
=\int_{\d_{\pm}\cap U_\g} u|_{\d_{\pm}}(t)^{-1}u(t)dh_\g(t)=
\big[h_\g(t) \big]_{\pa(\d_{\pm}\cap U\g)}=\pm 1,
$$
which means that 
$\la \d_{\pm}^{u^{-1}},\kappa(\g^u)\ra=
\cI_h(\d_{\pm}^{u^{-1}},\g^u)
$.
\end{proof}

Let 
$(\g_1^u,\dots, \g_{m+1}^u)$, 
$\tr(\f_1,\dots,\f_{m+1})$,
$\tr(\d_1^{u^{-1}},\dots, \d_{m+1}^{u^{-1}})$ 
and 
$(\psi_1,\dots,\psi_{m+1})$ 
be any bases of
$H_1(T,D;\cL)$, 
$H^1_{C^\infty_V}(T,D;\cL)$, 
$H_1(T^\vee,D^\vee;\cL^\vee)$,
and 
$H^1_{C^\infty_V}(T^\vee,D^\vee;\cL^\vee)$, respectively.
We define four matrices by
$$
\Phi =\left( \la \f_i,\g^u_j\ra \right)_{\substack{1\le i\le m+1\\
1\le j\le m+1}},\quad 
\Psi =\left( \la \d^{u^{-1}}_i,\psi_j,\ra \right)_{\substack{1\le i\le m+1\\
1\le j\le m+1}},$$
$$
H=\left( \cI_h(\d_i^{u^{-1}},\g^u_j) \right)_{\substack{1\le i\le m+1\\
1\le j\le m+1}},\quad 
C=\left( \cI_c(\f_i,\psi_j) \right)_{\substack{1\le i\le m+1\\
1\le j\le m+1}}.
$$

\begin{theorem}
\label{th:period-rel}
The matrices $\Phi ,\Psi,H$ and $C$ satisfy a twisted period relation 
\begin{equation}
\label{eq:TPR}
H=\Psi C^{-1} \Phi \quad (\Leftrightarrow\  C=\Phi  H^{-1} \Psi) .
\end{equation}
\end{theorem}

\begin{proof}
Let $K$ be the representation matrix of $\kappa$ in (\ref{eq:PD-iso}) 
with respect to the bases $(\g_1^u,\dots, \g_{m+1}^u)$ and 
$(\psi_1,\dots,\psi_{m+1})$
of $H_1(T,D;\cL)$ and  $H^1_{C^\infty_V}(T^\vee,D^\vee;\cL^\vee)$.
Then the matrix $K$ satisfies 
$$(\kappa(\g_1^u),\dots, \kappa(\g_{m+1}^u))=(\psi_1,\dots,\psi_{m+1})K.$$
Since 
$$\cI_c(\f_i,\kappa(\g_j^u))=\la \f_i,\g_j^u\ra,\quad 
\la \d_i^{u^{-1}},\kappa(\g_j^u)\ra =\cI_h(\d_i^{u^{-1}},\g_j^u),
$$
we have 
\begin{align*}
\Phi &=\tr(\f_1,\dots,\f_{m+1})\cdot 
(\g_1^u,\dots, \g_{m+1}^u)=
\tr(\f_1,\dots,\f_{m+1})\cdot
(\kappa(\g_1^u),\dots, \kappa(\g_{m+1}^u))
\\
 &=\tr(\f_1,\dots,\f_{m+1})\cdot (\psi_1,\dots,\psi_{m+1})K=CK,
\\
H&=\tr(\d^{u^{-1}}_1,\dots,\d^{u^{-1}}_{m+1})\cdot (\g_1^u,\dots, \g_{m+1}^u)
=\tr(\d^{u^{-1}}_1,\dots,\d^{u^{-1}}_{m+1})\cdot 
(\kappa(\g_1^u),\dots, \kappa(\g_{m+1}^u))\\
&=
\tr(\d^{u^{-1}}_1,\dots,\d^{u^{-1}}_{m+1})\cdot(\psi_1,\dots,\psi_{m+1})K
=\Psi K,
\end{align*}
where $\f_i \cdot \g^u_j$, $\f_i \cdot \psi_j$, 
$\d^{u^{-1}}_i \cdot \g^u_j$ and $\d^{u^{-1}}_i \cdot \psi_j$ are 
regarded as  
$\la\f_i, \g^u_j\ra$, $\cI_c(\f_i,\psi_j)$, 
$\cI_h(\d^{u^{-1}}_i, \g^u_j)$ and $\la \d^{u^{-1}}_i, \psi_j\ra$, 
respectively. 
By eliminating $K$ from these, we obtain the twisted period relation.
\end{proof}

\begin{example}
We give examples of twisted period relations for $m=2$.\\ 
(1) $\a_0=\a_1=\cdots=\a_4=0$. \\
In this case, we have $u(t)=1$, $\w=0$, $\na_{t}=d$, 
$I=\IZ=\{0,1,\dots,4\}$, $\IZc=\emptyset$, $\ID=\{0,1,2\}$, 
$\IP=\{3,4\}$, $T=\P^1-\{1,\infty\}=\C-\{1\}$, $D=\{0,x_1,x_2\}$, 
$T^\vee=\P^1-\{0,x_1,x_2\}$ and $D^\vee=\{1,\infty\}$.
We array $x_i$ $(0\le i\le 4)$ as 
$$0=x_0< x_1<x_2<x_3=1<x_4=\infty.$$
We give bases of $H_1(T,D;\cL)$, $H^1_{alg}(T,D;\cL)$, 
$H_1(T^\vee ,D^\vee;\cL^\vee)$ and $H^1_{alg}(T^\vee,D^\vee;\cL^\vee)$ as
$$\begin{array}{cc}
(\circlearrowleft^u_3,\ell^u_1-\ell^u_0,\ell^u_2-\ell^u_0),
& 
\begin{pmatrix}
\f_0\\
\na_{t}(t)\\
\na_{t}(t^2)
\end{pmatrix},
\\
\begin{pmatrix}
\ell^{u^{-1}}_4-\ell^{u^{-1}}_3\\
\circlearrowleft^{u^{-1}}_1\\
\circlearrowleft^{u^{-1}}_2
\end{pmatrix},
&
\Big(\na_t^\vee\big(\dfrac{-1}{t}\big),
\dfrac{dt}{t(t-x_1)}, 
\dfrac{dt}{t(t-x_2)}\Big). 
\end{array}
$$
Then the period matrices $\Phi,\Psi$ and the intersection matrices $H,C$ become
$$
\begin{array}{ll}
\Phi=\begin{pmatrix}
2\pi\sqrt{-1} & \log (1-x_1) & \log(1-x_2)\\
0 & x_1 &x_2\\
0 & x_1^2 & x_2^2
\end{pmatrix},
&
H=\begin{pmatrix}
-1 & 0 & 0 \\
 0 &-1 & 0 \\
 0 & 0 &-1  \\
\end{pmatrix},
\\
\Psi
=\begin{pmatrix}
1 & -\dfrac{\log(1-x_1)}{x_1} & -\dfrac{\log(1-x_2)}{x_2}\\
0 & \dfrac{2\pi\sqrt{-1}}{x_1} &0\\
0 & 0 &\dfrac{2\pi\sqrt{-1}}{x_2}\\
\end{pmatrix},
&
C=-2\pi\sqrt{-1}
\begin{pmatrix}
1 & 0 &0 \\
0 & 1 &1 \\
0 &x_1&x_2
\end{pmatrix},
\end{array}
$$
which satisfies (\ref{eq:TPR}).
\end{example}

\smallskip\noindent
(2) $\a_0=\a_1=\a_2=0$, $\a_3=1$, $\a_4=-1$. \\
In this case, we have $u(t)=t-1$, $\w=\dfrac{dt}{t-1}$, 
$\na_{t}=d+\dfrac{dt}{t-1}\wedge$, $\f_0=\na_t(1)$, 
$I=\IZ=\{0,1,\dots,4\}$, $\IZc=\emptyset$, 
$\ID=\{0,1,2,3\}$, $\IP=\{4\}$, $T=\P^1-\{\infty\}=\C$, 
$D=\{0,x_1,x_2,x_3\}$, $T^\vee=\P^1-\{0,x_1,x_2,1\}$ and $D^\vee=\{\infty\}$.  
Note that $H_1(T;\cL)=0$. 
We give bases of $H_1(T,D;\cL)$, $H_{alg}^1(T,D;\cL)$, 
$H_1(T^\vee,D^\vee;\cL\vee)$ and $H_{alg}^1(T^\vee,D^\vee;\cL^\vee)$
by 
$$
\begin{array}{cc}
(\ell^u_0-\ell^u_3,\ell^u_1-\ell^u_3,\ell^u_2-\ell^u_3), &
\begin{pmatrix}
\f_0\\
\na_{t}(t-1)\\
\na_{t}(t-1)^2
\end{pmatrix},
\\
\begin{pmatrix}
\circlearrowleft_0^{u^{-1}}\\ 
\circlearrowleft_1^{u^{-1}}\\ 
\circlearrowleft_2^{u^{-1}} 
\end{pmatrix},
&
(\dfrac{dt}{t},\dfrac{dt}{t-x_1},\dfrac{dt}{t-x_2}).
\end{array}
$$
Then the period matrices $\Phi,\Psi$ and the intersection matrices 
$H,C$ become
$$\begin{array}{ll}
\Phi=\begin{pmatrix}
-1 & x_1-1 & x_2-1\\
 1 & (x_1-1)^2 & (x_2-1)^2\\
-1 & (x_1-1)^3 & (x_2-1)^3
\end{pmatrix}, &
H=\begin{pmatrix}
-1&   &  \\
  &-1 &  \\
  &   &-1\\
\end{pmatrix},
\\
\Psi=
2\pi\sqrt{-1}\begin{pmatrix}
-1&   &  \\
  &\dfrac{1}{x_1-1} &  \\
  &   &\dfrac{1}{x_2-1}\\
\end{pmatrix},&
 C=
 -2\pi\sqrt{-1}\begin{pmatrix}
  1&  1  & 1 \\
 -1&x_1-1&x_2-1\\
  1&(x_1-1)^2&(x_2-1)^2\\
 \end{pmatrix},
\end{array}
$$
which satisfy (\ref{eq:TPR}).
 
\section{The isomorphism between $H_1(T,D;\cL)$ and $\Sol_x(a,b,c)$}
Let $x$ vary in a small simply connected domain $W$ in $X$. 
We have trivial vector bundles 
$$\prod_{x\in W} H_1(T_x,D_x;\cL_x),\quad 
\prod_{x\in W} H^1_{*}(T_x,D_x;\cL_x),$$
where $*$ is $alg$, $C^\infty$, $C^\infty_V$ and the blank. 
We can extend the pairing $\la \f,\g^u\ra$ to that between 
sections of these trivial vector bundles.
Hereafter, we identify the spaces of local sections of 
these trivial vector bundles with 
the fibers $H^1_{*}(T,D;\cL)$ and $H_1(T,D;\cL)$ at $x\in W$
by the local triviality.

The partial differential operator $\pa_j=\dfrac{\pa}{\pa x_j}$ ($1\le j\le m$) 
acts on $\la \f,\g^u\ra$ 
for $\f\in H^1_{C^\infty_V}(T,D;\cL)$ and $\g^u\in H_1(T,D;\cL)$
as 
$$
\pa_j\la \f,\g^u \ra=\int_\g \pa_j(u(t)\f)=
\int_\g  u(t,x)( \pa_j\f +\frac{\pa_j u(t,x)}{u(t)}\f)
=\la \na_j\f,\g^u\ra,
$$
where $\na_j$ denotes the operator 
$$\pa_j+\frac{\pa_j u(t,x)}{u(t)}=\pa_j+\frac{-\a_j}{t-x_j}
$$
in $\C(t,x_1,\dots,x_m)\la \pa_1,\dots,\pa_m\ra.$
The operator $\pa_j$ induces a linear transformation $\na_j$ on 
$H^1_{C^\infty_V}(T,D;\cL)$.

Though the identity
$$\pa_j\la \f,\g^u\ra=\la\na_j \f,\g^u\ra
$$
holds for any element $\f\in \cE^1(T,D;\cL)$ and 
any element $\g^u$ in $H_1(T;\cL)\subset H_1(T,D;\cL)$, 
the operator $\na_j$ cannot act directly on the spaces 
$$
H^1_{alg}(T,D;\cL), \quad  H^1_{C^\infty}(T,D;\cL),
$$
since $\W^0(T,D;\cL)$ and $\cE^0(T,D;\cL)$ are not kept 
invariant under the action of $\na_j$.
In fact, in case of $\a_j=1$ and $\a_i\notin \Z$ ($0\le i\le m+2$, $i\ne j$), 
$\w=\na_t(1)$ is the zero of $H^1_{C^\infty}(T,D;\cL)$ by 
$\lim\limits_{t\to x_j}u(t)\cdot 1=0$ by $\a_j=1$. 
However 
$$\na_j(\w)=\na_j\na_t(1)=\na_t\na_j(1)=\na_t(\frac{-1}{t-x_j})
=\sum_{0\le i\le m+1}^{i\ne j} \frac{-\a_i}{(t-x_j)(t-x_i)}
\in \cE^1(T,D;\cL)
$$
is not the zero of $H^1_{C^\infty}(T,D;\cL)$ since 
$\lim\limits_{t\to x_j}u(t)\cdot\dfrac{-1}{t-x_j}\ne0$.

By using the isomorphism $\imath_D:H^1_{C^\infty}(T,D;\cL)\to 
H^1_{C^\infty_V}(T,D;\cL)$, 
we define an action 
\begin{equation}
\label{eq:extend-na-j}
\na_j\circ \imath_D:H^1_{C^\infty}(T,D;\cL)\to 
H^1_{C^\infty_V}(T,D;\cL)\hookrightarrow H^1_{C^\infty}(T,D;\cL)
\end{equation}
on $H^1_{C^\infty}(T,D;\cL)$. 
Similarly, we have an action on $H^1_{alg}(T,D;\cL)$ defined by 
the composition of $\na_j\circ \imath_D$ and 
the canonical isomorphism
$H^1_{C^\infty}(T,D;\cL)\to H^1_{alg}(T,D;\cL)$. 
These actions are simply denoted by $\na_j$.

\begin{lemma}
\label{lem:diff-phi0}
Let $\phi_0$ be an element of $H^1_{C^\infty_V}(T,D;\cL)$ cohomologous to 
$\f_0=\f_{m+1,m+2}=\dfrac{dt}{t-1}$ as elements of $H^1_{C^\infty}(T,D;\cL)$. 
Then $\na_j \phi_0$ is cohomologous to 
$$
\frac{\a_j\f_{m+1,m+2}-\f_{j,m+2}}{x_j-1}=
\left\{
\begin{array}{lcc}
 \dfrac{-\a_j dt}{(t-x_j)(t-1)} 
=\dfrac{\a_j}{x_j-1}\Big(\dfrac{dt}{t-1}-\dfrac{dt}{t-x_j}\Big)
& \textrm{if} & \a_j\ne 0,\\[4mm]
\na_t\Big(\dfrac{h_j(t)}{u(t)}\cdot \dfrac{u(x_j)}{x_j-1}\Big) 
=\dfrac{u(x_j)}{x_j-1}\na_t\Big(\dfrac{h_j(t)}{u(t)}\Big) 
& \textrm{if} & \a_j=0,
\end{array}
\right.
$$ 
as elements of $H^1_{C^\infty}(T,D;\cL)$. 
\end{lemma}
\begin{proof}
Note that 
$$\phi_0=\f_0-\sum_{i\in \ID\cup\IZc} \na_t(h_i(t)f_i(t)),$$
where 
$f_i(t)$ is a single valued meromorphic function on $U_i$ satisfying 
$\na_t(f_i(t))=\f_0$ and $f_i(x_i)=0$ for $i\in \ID$. 
Since $\na_j\na_t=\na_t\na_j$ and 
$\pa_j h_i(t)=0$ on $V_k$ for any $k\in \ID\cup \IZc$, we have 
\begin{align*}
\na_j \phi_0 =&\dfrac{-\a_jdt}{(t-x_j)(t-1)}-
\sum_{i\in \ID\cup\IZc} \na_t\na_j( h_i(t)f_i(t) )
\\
=&\dfrac{-\a_jdt}{(t-x_j)(t-1)}-
\sum_{i\in \ID\cup\IZc}\Big[
\na_t\Big(\pa_j  h_i(t) \cdot f_i(t)\Big) 
+ \na_t\Big(h_i(t) \na_j f_i(t)\Big)\Big]\\
=&\dfrac{-\a_jdt}{(t-x_j)(t-1)}-
\sum_{i\in \ID\cup\IZc}  \na_t\Big(h_i(t) \na_j f_i(t)\Big).
\end{align*}
Here note that $\pa_j  h_i(t) \cdot f_i(t)\in 
\cE^0_{C^\infty_V}(T,D;\cL)\subset \cE^0_{C^\infty}(T,D;\cL)$.
It is easy to see  that 
$h_i(t) \na_j f_i(t)$ belongs to $\cE^0_{C^\infty}(T,D;\cL)$ for $i\in \IZc$.
For $i\in \ID$, if 
$$\lim_{t\to x_i} u(t)\cdot \na_j f_i(t)=0$$ 
then 
$h_i(t) \na_j f_i(t)$ belongs to $\cE^0_{C^\infty}(T,D;\cL)$.
Since $\ord_{x_i}(\na_j\f_0)=\ord_{x_i}(\dfrac{-\a_j}{(t-x_j)(t-1)})\ge -1$ 
and $\na_j f_i(t)$ satisfies $\na_t(\na_j f_i(t))=\na_j\f_0$, 
we have $\ord_{x_i}\na_j f_i(t)\ge 0$. 
Note also that $\ord_{x_i} u(t)\ge 0$ for any $i\in \ID$. 
Thus if $\ord_{x_i} u(t)>0$ or $\ord_{x_i} f_i(t)>0$ then 
$\na_t\Big(h_i(t) \na_j f_i(t)\Big)$ is the zero of $H^1_{C^\infty}(T,D;\cL)$.
If $\ID\ni i\ne j,m+1$ then $\ord_{x_i} f_i(t)>0$, 
if $\ID\ni i=j, \a_j>0$ then $\ord_{x_j} u(t)>0$, 
and if $\ID\ni i=m+1$ then $\ord_{x_{m+1}}u(t)>0$ 
by  $\ord_{x_{m+1}}(u(t)\f_0)\ge 0$.  
Hence if $\a_j\ne 0$  then 
$$\sum_{i\in \ID\cup\IZc}  \na_t\Big(h_i(t) \na_j f_i(t)\Big)
$$
is the zero of $H^1_{C^\infty}(T,D;\cL)$ and 
$\na_t\phi_0$ is cohomologous to $\dfrac{-\a_jdt}{(t-x_j)(t-1)}$. 
If $\a_j= 0$  
then 
$$\sum_{i\in \ID\cup\IZc}  \na_t\Big(h_i(t) \na_j f_i(t)\Big)
$$
is cohomologous to $\na_t\Big(h_j(t) \na_j f_j(t)\Big)$.
In this case, 
$f_j(t)$ admits an integral representation 
$$\frac{1}{u(t)}\int_{x_j}^t \frac{u(t')dt'}{t'-1}.$$
Here note that an integral 
$$\frac{1}{u(t)}\int_{p}^t \frac{u(t')dt'}{t'-1}$$
is a solution to $\na_t f(t)=\f_0$ for any starting point $p$, 
and that $p$ should be $x_j$ by the vanishing property at $t=x_j$
for the condition $h_j(t)f_j(t)\in \cE^0(T,D;\cL)$. 
Hence we have 
$$
\na_j f_j(t)=\na_j\left(\frac{1}{u(t)}\right)\cdot 
\int_{x_j}^t \frac{u(t')dt'}{t'-1}
+\frac{1}{u(t)}\cdot \pa_j\left(\int_{x_j}^t \frac{u(t')dt'}{t'-1}\right)
=\frac{1}{u(t)}\frac{-u(x_j)}{x_j-1},
$$
since $u(t)$ is independent of $x_j$. 
Therefore,  
$\na_j(\phi_0)$ is cohomologous to 
$\na_t\Big(\dfrac{h_j(t)}{u(t)}\cdot \dfrac{u(x_j)}{x_j-1}\Big)$ 
as elements of $H^1_{C^\infty}(T,D;\cL)$ in the case $\a_j=0$. 
\end{proof}

\begin{theorem}
\label{th:hom-sol-iso}
The space of sections of the trivial vector bundle $H_1(T,D;\cL)$ 
around $x$ is isomorphic to the space $\Sol_x(a,b,c)$ of 
local solutions to $\cF_D(a,b,c)$ around $x\in X$ by the map 
$$
\jmath_{\f_0} :H_1(T,D;\cL)\ni \g^u \mapsto \la \f_0,\g^u\ra=
\int_\g u(t) \f_0 \in \Sol_x(a,b,c).
$$
\end{theorem}

\begin{proof}
By similar way to \cite[\S6.4]{Y1}, we can show that 
$\jmath_{\f_0}(\g^u)=\la \f_0,\g^u\ra$
is a local solution to $\cF_D(a,b,c)$ around $x\in X$ 
for any $\g^u\in H_1(T,D;\cL)$.   
Since $H_1(T,D;\cL)$ and $\Sol_x(a,b,c)$ are $m+1$ dimensional, 
we show that the map $\jmath_{\f_0}$ is surjective. 
Let $\g_i^u$ $(1\le i\le m+1)$ be the basis of $H_1(T,D;\cL)$ given in 
(\ref{eq:hom-basis-1}) or (\ref{eq:hom-basis-2}). 
Since 
$$\la \f_0,\g_i\ra=\la \phi_0,\g_i\ra,\quad 
\pa_j \la \phi_0,\g_i\ra=\la \na_j \phi_0,\g_i\ra
=\frac{1}{x_j-1}\la \a_j\f_0-\f_{j,m+1},\g_i\ra,
$$
for $1\le i\le m+1$ and $1\le j\le m$,
the Wronskian 
\begin{align*}
&\begin{vmatrix}
\quad \la \phi_0,\g_1^u\ra 
& \cdots&
\quad  \la \phi_0,\g_{m+1}^u\ra \\
\pa_1\la \phi_0,\g_1^u\ra 
& \cdots& \pa_1\la \phi_0,\g_{m+1}^u\ra \\
\vdots 
& \ddots &\vdots\\
\pa_m\la \phi_0,\g_1^u\ra 
& \cdots& \pa_m\la \phi_0,\g_{m+1}^u\ra \\
\end{vmatrix}
=
\begin{vmatrix}
\la \f_0,\g_1^u\ra 
& \cdots& \la \f_{0},\g_{m+1}^u\ra \\
\dfrac{\la \a_1\f_0\!-\!\f_{1,m+2},\g_1^u\ra}{x_1-1}
& \cdots&  \dfrac{\la\a_1\f_0\!-\!\f_{1,m+2},\g_{m+1}^u\ra}{x_1-1} \\
\vdots 
& \ddots &\vdots\\
\dfrac{\la \a_m\f_0\!-\!\f_{m,m+2},\g_1^u\ra}{x_m-1} 
& \cdots& \dfrac{\la \a_m\f_0\!-\!\f_{m,m+2},\g_{m+1}^u\ra}{x_m-1} \\
\end{vmatrix}\\
&=
\frac{(-1)^m}{(x_1-1)\cdots(x_m-1)}
\begin{vmatrix}
\la \f_{m+1,m+2},\g_1^u\ra 
& \cdots& \la \f_{m+1,m+2},\g_{m+1}^u\ra \\
\la \f_{1,m+2},\g_1^u\ra
& \cdots& \la \f_{1,m+2},\g_{m+1}^u\ra\\
\vdots 
& \ddots &\vdots\\
\la \f_{m,m+2},\g_1^u\ra
& \cdots& \la\f_{m,m+2},\g_{m+1}^u\ra\\
\end{vmatrix}
\end{align*}
does not vanish by the perfectness of the pairing between the relative 
twisted homology and cohomology groups and Theorem \ref{th:cint-number}.
Hence $\la \f_0,\g_i\ra$ $(1\le i\le m+1)$ are linearly independent as 
functions in $x_1,\dots,x_m$, and the map $\jmath_{\f_0}$ is surjective. 
\end{proof}

\section{Invariant subspaces of $H^1_{C^\infty_V}(T,D;\cL)$ 
under partial differentials}
\label{sec:ISPD}
\begin{proposition}
\label{prop:inv-space-D}
The space 
$$
\na_t\cE^0(T;\cL)=
\{\f\in H^1_{C^\infty_V}(T,D;\cL)\mid \;^\exists f\in 
\cE^0(T;\cL) 
\textrm{ s.t. }
\f=\na_t f\}
$$
is invariant under the action $\na_j$ $(j=1,\dots,m)$, 
where 
$$\cE^0(T;\cL)=\{f(t)\in \cE^0(\wt x)\mid u(t)\cdot f(t) \textrm{ is }
C^\infty \textrm{ on }U_i \textrm{ for any } i\in \ID\}.
$$ 
This space is spanned by 
$
\na_t\big(\dfrac{h_i(t)}{u(t)}\big)$  for  $i\in \ID$.
If $\a\in \Z^{m+3}$ then they satisfy 
$$
\sum_{i\in \ID} \na_t\Big(\frac{h_i(t)}{u(t)}\Big)=0.
$$
The dimension of this space is 
$$\wt r=\left\{
\begin{array}{ccc}
r  & \textrm{if} & \a\notin \Z^{m+3},\\
r-1& \textrm{if} & \a\in \Z^{m+3},
\end{array}
\right.
\quad (r=\#D).
$$
Each $1$-dimensional span of $\na_t\Big(\dfrac{h_i(t)}{u(t)}\Big)$ 
is invariant under the action $\na_j$ $(j=1,\dots,m)$.
\end{proposition}
\begin{proof} Note that if $D$ is empty then 
$\na_t\cE^0(T;\cL)=0$, since $\na_t(f)$ is the zero of 
$H^1_{C^\infty}(T,D;\cL)\simeq H^1_{C^\infty_V}(T,D;\cL)$ in this case. 
Let $\f$ be any element of $\na_t\cE^0(T;\cL)$. Then there exists 
$f\in \cE^0(T;\cL)$ such that $\f=\na_t f$.
Note that $\na_j\f$ belongs to $\cE^1_V(T,D;\cL)$ and admits 
an expression 
$$\na_j\f=\na_j(\na_tf)=\na_t(\na_j f).$$
Since $\f$ vanishes identically around $x_i\in B\cup D$, 
$f$ is identically $0$ around $x_i\in B$ and 
$f$ takes a form $g(x)/u(t)$ around $x_i\in D$, where 
$g(x)$ is a function independent of $t$.
This local property of $f$ is preserved under $\na_j$ by  
$$\na_j\Big(\frac{g(x)}{u(t)}\Big)= g(x)\na_j\frac{1}{u(t)}
+\frac{\pa_j g(x)}{u(t)}=\frac{\pa_j g(x)}{u(t)}, 
$$
$\na_t\cE^0(T;\cL)$ is invariant under the action $\na_j$ $(j=1,\dots,m)$. 
We also see that this space is spanned by 
$$
\na_t\Big(\frac{h_i(t)}{u(t)}\Big)=\frac{dh_i(t)}{u(t)},\quad i\in \ID. 
$$
We show that if $\a\in \Z^{m+3}$ then they satisfy the linear relation
$$
\sum_{i\in \ID} \na_t\Big(\frac{h_i(t)}{u(t)}\Big)=0.$$
Under this condition, $1/u(t)$ becomes single valued on $T$ and 
satisfies $\na_t(1/u(t))=0$. Thus this linear combination is equal to 
$$
-\na_t\frac{1}{u(t)}+\sum_{i\in \ID} \na_t\Big(\frac{h_i(t)}{u(t)}\Big)
=\na_t\Big(-\frac{1}{u(t)}+\sum_{i\in \ID} \frac{h_i(t)}{u(t)}\Big)
.
$$
Since 
$$\Big(-\frac{1}{u(t)}+\sum_{i\in \ID} \frac{h_i(t)}{u(t)}\Big)u(t)
=-1+\sum_{i\in \ID} h_i(t)$$
vanishes identically around $x_i\in D$, the function 
$\displaystyle{-\frac{1}{u(t)}+\sum_{i\in \ID} \frac{h_i(t)}{u(t)}}$
belongs to $\cE^0_V(T,D;\cL)$ and its $\na_t$-image is the zero of 
$H^1_{C^\infty_V}(T,D;\cL)$.
Hence we have the linear relation, and the claim on the dimension of this space.

Note that 
$$\na_j\frac{dh_i(t)}{u(t)}=\na_t\na_j\frac{h_i(t)}{u(t)}=
\na_t\Big(h_i(t)\na_j\frac{1}{u(t)}+\frac{\pa_jh_i(t)}{u(t)}
\Big)
=\na_t\frac{\pa_jh_i(t)}{u(t)}
$$
is cohomologous to $0$, since $\pa_jh_i(t)$ vanishes identically around 
$x_i\in B\cup D$ for $1\le j\le m$.
Thus we have 
$$\na_j\Big(g(x)\frac{dh_i(t)}{u(t)}\Big)=
\frac{dh_i(t)}{u(t)}\pa_j g(x)+g(x)\na_j\frac{dh_i(t)}{u(t)}
=\big(\pa_j g(x)\big)\cdot \frac{dh_i(t)}{u(t)},$$
which means that 
each $1$-dimensional span of $\na_t\Big(\dfrac{h_i(t)}{u(t)}\Big)$ 
is invariant under the action $\na_j$ $(1\le j\le m)$.
\end{proof}

\begin{cor}
\label{cor:inv-space-hom-dual}
The space $\na_t\cE^0(T;\cL)$ coincides with 
$$H_1(T;\cL)^\vdash=
\{\f\in H^1_{C^\infty_V}(T,D;\cL)\mid \la \f,\g^u\ra=0 \textrm{ for any }
\g^u\in H_1(T;\cL)\subset H_1(T,D;\cL)\}.
$$
\end{cor}
\begin{proof}
For any elements $\f=\na_t f\in \na_t\cE^0(T;\cL)$ 
and $\g^u\in H_1(T;\cL)$, 
we have 
$$\la \f,\g^u\ra= \la \na_t f,\g^u\ra= \la f,\pa^u \g^u\ra= 0,$$
which yields that $\na_t\cE^0(T;\cL)\subset H_1(T;\cL)^\vdash$.
Since they are of same dimension, they coincide.
\end{proof}

\begin{remark}
Since the spaces $H^1_{alg}(T,D;\cL)$, $H^1_{C^\infty}(T,D;\cL)$ and 
$H^1_{C^\infty_V}(T,D;\cL)$ are canonically isomorphic to $H^1(T,D;\cL)$, 
the subspaces 
$$
\{\f\in H^1_{alg}(T,D;\cL)\mid \la \f,\g^u\ra=0 \textrm{ for any }
\g^u\in H_1(T;\cL)\},
$$
$$
\{\f\in H^1_{C^\infty}(T,D;\cL)\mid \la \f,\g^u\ra=0 \textrm{ for any }
\g^u\in H_1(T;\cL)\},
$$
$$
\{\f\in H^1_{C^\infty_V}(T,D;\cL)\mid \la \f,\g^u\ra=0 \textrm{ for any }
\g^u\in H_1(T;\cL)\}
$$
are isomorphic to one another. 
 \end{remark}

For the relative twisted dual homology and cohomology groups, 
we have trivial vector bundles 
$$
\prod_{x\in X} H_1(T^\vee,D^\vee;\cL^\vee),\quad
\prod_{x\in W} H^1_{*}(T^\vee,D^\vee;\cL^\vee),
$$
over a simply connected domain $W$ in $X$, where $*$ is $alg$, 
$C^\infty$, $C^\infty_V$ and the blank. 
We can regard the natural pairing 
$$\la \d^{u^{-1}},\psi \ra=\int_\d \frac{\psi}{u(t)}
$$
between $H_1(T^\vee,D^\vee;\cL^\vee)$ and 
$H^1_{C^\infty_V}(T^\vee,D^\vee;\cL^\vee)$ 
as that between the spaces of local sections of these trivial vector bundles. 
As mentioned previously, the partial differential operator $\pa_j$ induces 
a linear transformation 
$$\na_j^\vee=\pa_j-\frac{\pa_j u(t,x)}{u(t)}=\pa_j+\frac{\a_j}{t-x_j}
$$
on $H^1_{C^\infty_V}(T^\vee,D^\vee;\cL^\vee)$, 
$H^1_{C^\infty}(T^\vee,D^\vee;\cL^\vee)$ and 
$H^1_{alg}(T^\vee,D^\vee;\cL^\vee)$. 
It also acts on $\cI_c(\f,\psi)$ as 
\begin{align}
\nonumber
\pa_j \cI_c(\f,\psi)
=&\iint_{T\cap T^\vee} \Big[\pa_j(u(t)\f)\wedge \frac{\psi}{u(t)}
+u(t)\f\wedge \pa_j\frac{\psi}{u(t)}\Big]\\
\label{eq:act-on-intno}
=&\iint_{T\cap T^\vee} u(t)\na_j(\f)\wedge \frac{\psi}{u(t)}
+\iint_{T\cap T^\vee}u(t)\f\wedge \frac{\na_j^\vee(\psi)}{u(t)}\\
\nonumber
=&\cI_c(\na_j\f,\psi)+\cI_c(\f,\na_j^\vee \psi),
\end{align}
where $\f\in H^1_{C^\infty_V}(T,D;\cL)$ and 
$\psi\in H^1_{C^\infty_V}(T^\vee,D^\vee;\cL^\vee)$. 
\begin{cor}
\label{cor:inv-space-P}
Suppose that $k\in \IP$. 
Then the space 
 $$(u(t)dh_k(t))^\perp=
\{\f\in H^1_{C^\infty_V}(T,D;\cL)\mid \cI_c(\f,u(t)dh_k(t))=0\}$$
 is invariant under the action $\na_j$ $(j=1,\dots,m)$, and 
 $$\dim (u(t)dh_k(t))^\perp=
\left\{
  \begin{array}{cl}
 m+1 & \textrm{if } \a\in \Z^{m+3},\#(\IP)=1,\\
 m   & \textrm{otherwise}. 
 \end{array}
 \right.
 $$
The space 
$(u(t)dh_k(t))^\perp$ coincides with 
$$(\circlearrowleft_k^u)^\vdash
=\{\f\in H^1_{C^\infty_V}(T,D;\cL)\mid \la \f, \circlearrowleft_k^u\ra=0\}.$$
\end{cor}
\begin{proof}
Let $\f$ be any element of $(u(t)dh_k(t))^\perp$.
By (\ref{eq:act-on-intno}), we have
$$\pa_j\cI_c(\f,u(t)dh_k(t))
=\cI_c(\na_j\f,u(t)dh_k(t))+\cI_c(\f,\na_j^\vee(u(t)dh_k(t)))=0
$$
for $1\le j\le m$. Since 
\begin{align*}
&\na_j^\vee (u(t)dh_k(t))=
\na_j^\vee \na_t^\vee (u(t)h_k(t))=
\na_t^\vee \na_j^\vee(u(t)h_k(t))\\
=&\na_t^\vee (
h_k(t)\na_j^\vee(u(t)) +u(t)\pa_j h_k(t))
=\na_t^\vee (u(t)\pa_j h_k(t)),
\end{align*}
and $\pa_j h_k(t)$ vanishes identically around $x_i$ for 
$i\in \IP$, $\na_j^\vee(u(t)dh_k(t))$ is the zero of 
$H^1_{C^\infty_V}(T^\vee,D^\vee;\cL^\vee)$.
Hence we have 
$\cI_c(\na_j\f,u(t)dh_k(t))=0$  and 
$\na_j\f\in (u(t)dh_k(t))^\perp$.

By the perfectness of $\cI_c$, 
if 
$u(t)dh_k(t)$ is not the zero of $H^1_{C^\infty_V}(T^\vee,D^\vee;\cL^\vee)$
then  $\na_j\f\in (u(t)dh_k(t))^\perp$ is $m$ dimensional. 
We show that $u(t)dh_k(t)$ degenerates only 
the case $\a\in \Z^{m+3}$ and $\IP=\{k\}$.  
In this case, we have 
$$
\na_t^\vee (u(t)h_k(t))=
\na_t^\vee (u(t)(h_k(t)-1))
$$
by $\na_t^\vee u(t)=0$. Since $\IP=\{k\}$ and 
$u(t)(h_k(t)-1)$ vanishes identically around $x_k$, 
it belongs to $\cE^0_V(T^\vee,D^\vee,L^\vee)$, 
which means that  $\na_t^\vee (u(t)(h_k(t)-1))$ is the zero of 
$H^1_{C^\infty_V}(T^\vee,D^\vee;\cL^\vee)$.
Except in this case, we can make a relative cycle $\d^{u^{-1}}$ 
by $\ell^{u^{-1}}_k$,  
it satisfies $\la \d^{u^{-1}},u(t)dh_k(t)\ra=1$.
Thus $u(t)dh_k(t)$ is different from the zero of  
$H^1_{C^\infty_V}(T^\vee,D^\vee;\cL^\vee)$.

We have shown in Proof of Theorem \ref{th:compatible} that 
$\la \f, \circlearrowleft_k^u\ra=\cI_c(\f, (u(t)dh_k(t)))$
which yields that 
$(u(t)dh_k(t))^\perp=(\circlearrowleft_k^u)^\vdash$.
 \end{proof}

\section{The Gauss-Manin connection and a Pfaffian system of 
$\cF_D(a,b,c)$}
\label{sec:GM-connection}
Let $\{W_n\}_{n\in N} $ be an open covering of $X$,
where $W_n$ are small simply connected domain in $X$. 
By patching the trivial vector bundles 
$$\prod\limits_{x\in W_n} H^1_{C^\infty}(T_x,D_x;\cL_x),\quad 
\prod\limits_{x\in W_n} H^1_{C^\infty}(T^\vee_x,D^\vee_x;\cL^\vee_x),$$ 
we have local systems
$$\cH^1(\cL)=\bigcup_{n\in N} \prod_{x\in W_n} H^1_{C^\infty}(T_x,D_x;\cL_x),
\quad 
\cH^1(\cL^\vee)=\bigcup_{n\in N}\prod_{x\in W_n} 
H^1_{C^\infty}(T^\vee_x,D^\vee_x;\cL^\vee_x) 
$$
over $X$.

\begin{lemma}
\label{lem:frame}
We can extend the local sections $\f_{i,m+2}$ in (\ref{eq:frame}) and 
$\psi_{0,i}$  in (\ref{eq:dual-frame}) 
to global sections of $\cH^1(\cL)$ and $\cH^1(\cL^\vee)$, respectively. 
The spaces $\cH^1(\cL)$ and $\cH^1(\cL^\vee)$ admit the structure of 
a trivial vector bundle over $X$.
The sections $\f_{i,m+2}$'s and $\psi _{0,i}$'s $(1\le i\le m+1)$ 
form a frame of $\cH^1(\cL)$ and that of $\cH^1(\cL^\vee)$, respectively. 
They are dual to each other with respect to the intersection form $\cI_c$.
\end{lemma}

\begin{proof}
It is obvious that $\f_{i,m+2}$ and $\psi _{0,i}$ are global sections 
for $\a_i\ne 0$.
In case of $\a_i=0$, 
we can regard $\f_{i,m+2}=u(x_i)dh_i/u(t)$ as 
a global section of $\cH^1(\cL)$ since $u(x_i)/u(t)$ is single valued 
in a tubular neighborhood of $t=x_i$. 
By Theorem \ref{th:cint-number} (1),  we see that 
$\f_{i,m+2}$'s and $\psi _{0,i}$'s are dual frames of 
$\cH^1(\cL)$ and $\cH^1(\cL^\vee)$.
\end{proof}

\begin{remark}
\label{rem:non-vectorbundle}
We cannot regard local systems 
$$\bigcup_{n\in N} \prod_{x\in W_n} H^1(T_x,D_x;\cL_x),
\quad 
\bigcup_{n\in N}\prod_{x\in W_n} 
H^1(T^\vee_x,D^\vee_x;\cL^\vee_x) 
$$
as vector bundles over $X$, 
since their monodromy representations are not trivial in general. 
\end{remark}

Note that 
$$d_x\la \f,\g^u\ra=\sum_{i=1}^m dx_i\wedge \pa_i\la \f,\g^u\ra=
\sum_{i=1}^m dx_i\wedge \la \na_i\f,\g^u\ra$$
for local sections $\f\in H^1_{C^\infty}(T,D;\cL)$ and 
$\g^u\in H_1(T,D;\cL)$, 
where $d_x$ is the exterior derivative on the space $X$, and 
$\na_i$ means the operator given in (\ref{eq:extend-na-j}).
Thus the Gauss-Manin connection on the vector bundle 
$\cH^1(\cL)$ is expressed as
$$\na_x=\sum_{i=1}^m dx_i\wedge \na_i:
\f \mapsto 
\sum_{i=1}^m dx_i\wedge \na_i(\f), 
$$
which is a map from the space of local sections of $\cH^1(\cL)$ 
to that of the tensor product 
of the holomorphic cotangent bundle over $X$ and $\cH^1(\cL)$.

We have also the dual connection 
$$\na_x^\vee=\sum_{i=1}^m dx_i\wedge \na_i^\vee:
\f \mapsto 
\sum_{i=1}^m dx_i\wedge \na_i^\vee(\f) 
$$
on the dual vector bundle
$\cH^1(\cL^\vee)$
of $\cH^1(\cL)$
with respect to the intersection form $\cI_c$.

In this section, we express the connection $\na_x$ by 
the intersection form $\cI_c$, and represent its connection matrix 
with respect to a frame of $\cH^1(\cL)$, which can be regarded  
as that of a Pfaffian system of $\cF_D$.

Let $\cH^1_\C(\cL)$ and $\cH^1_\C(\cL^\vee)$ be the $\C$-spans of 
$\f_{i,m+2}$ and $\psi _{0,i}$ $(1\le i\le m+1)$, respectively.
By Theorem \ref{th:cint-number} (2),(3), we have 
$$\f_{0,m+2}\in \cH^1_\C(\cL),\quad \psi _{0,m+2}\in \cH^1_\C(\cL^\vee).$$

\begin{lemma} 
\label{lem:naj-action}
For $1\le i\le m$, we have
$$\na_i(\f_{m+1,m+2})=-\dfrac{\f_{i,m+2}-\a_i\f_{m+1,m+2}}{x_i-1}
$$
and 
$$\na_i(\f_{j,m+2})=-\dfrac{\a_j\f_{i,m+2}-\a_i\f_{j,m+2}}{x_i-x_j}
\quad (0\le j\le m,\ j\ne i).  
$$

\end{lemma}
\begin{proof}
The first identity is essentially shown in Lemma \ref{lem:diff-phi0}.
We show the second identity. 
In case of $\a_i\a_j\ne 0$ $(0\le j\le m)$, we have
$$\na_i(\f_{j,m+2})=\frac{-\a_i}{t-x_i}\cdot \frac{\a_jdt}{t-x_j}
=\frac{-\a_i\a_j}{x_i-x_j}\left(\frac{dt}{t-x_i}-\frac{dt}{t-x_j}
\right)=-\frac{\a_j\f_{i,m+2}-\a_i\f_{j,m+2}}{x_i-x_j}.
$$
In case of $\a_i\ne 0, \a_j= 0$, we have 
\begin{align*}
&\na_i(\f_{j,m+2})=\na_i(\frac{-u(x_j)dh_j}{u(t)})
=-\na_i\na_t(\frac{u(x_j)h_j}{u(t)})
=-\na_t\na_i(\frac{u(x_j)h_j}{u(t)})\\
=&-\na_t\Big(u(x_j)h_j\na_i(\frac{1}{u(t)})+\frac{1}{u(t)}\pa_i(u(x_j)h_j)
\Big)=-\na_t\Big(\frac{h_j\pa_i(u(x_j))+u(x_j)\pa_i(h_j)}{u(t)}\Big).
\end{align*}
Here note that $\na_i(1/u(t))=0$.
Since $\pa_i(h_j)$ is identically zero around $x_k$ $(0\le k\le m+2)$, 
$\na_t\dfrac{u(x_j)\pa_i(h_j)}{u(t)}$  is the zero of 
$H^1_{C^\infty}(T,D;\cL)$. Thus $\na_i(\f_{j,m+2})$ is equal to 
$$
-\pa_i(u(x_j))\na_t(\frac{h_j}{u(t)})=\a_i\frac{u(x_j)}{x_j-x_i}
\frac{dh_j}{u(t)}=\frac{\a_i}{x_i-x_j}\f_{j,m+2}
$$
as elements of $H^1_{C^\infty}(T,D;\cL)$.

In case of $\a_i= 0,\a_{j}\ne 0$, By following 
Proof of Lemma \ref{lem:diff-phi0}, we have 
\begin{align*}
&\na_i\f_{j,m+2}=\frac{-\a_idt}{(t-x_i)(t-x_j)}
-\sum_{k\in \ID\cup \IZc}  \na_t\big(h_k(t)\na_i f_k(t)\big)
=-\na_t\big(h_i(t)\na_i f_i(t)\big)\\
=&-\na_t\Big(h_i(t)\na_i\big(
\frac{1}{u(t)} \int_{x_i}^t\frac{\a_ju(t')dt'}{t'-x_j}\big)\Big)
=\na_t\Big(\frac{h_i(t)}{u(t)}\frac{\a_ju(x_i)}{x_i-x_j}\Big)=
-\frac{\a_j}{x_i-x_j}\f_{i,m+2},
\end{align*}
where $f_k(t)$ is a holomorphic solution to $\na_t(f(t))=\f_{j,m+2}$ 
around $x_k$ ($k\in \ID\cup \IZc$).

In case of $\a_i=\a_j=0$, we have
$$\na_i(\f_{j,m+2})=-\na_i\Big(\na_t\big(\frac{u(x_j)h_j(t)}{u(t)}\big)\Big)
=-\na_t\Big(\na_i\big(\frac{u(x_j)h_j(t)}{u(t)}\big)\Big)
=-\na_t\Big(\frac{\pa_i(u(x_j)h_j(t))}{u(t)}\Big).
$$
Since $\pa_iu(x_j)=0$ by $\a_i=0$, the numerator of the last term reduces to
$$\pa_i(u(x_j)h_j(t))=h_j(t)\pa_iu(x_j)+u(x_j)\pa_ih_j(t)
=u(x_j)\pa_ih_j(t),
$$
which vanishes identically around $x_k$ $(0\le k\le m+2)$. Hence
$\na_i(\f_{j,m+2})$ is the zero of $H^1_{C^\infty_V}(T,D;\cL)$.
\end{proof}

By Lemma \ref{lem:naj-action} together with Theorem \ref{th:cint-number} (2), 
we can define linear transformations 
\begin{equation}
\label{eq:cRij}
\cR_{i,j}: \cH^1_\C(\cL) \ni \f \mapsto 
\lim_{x_i\to x_j}(x_i-x_j)\na_i(\f)\in \cH^1_\C(\cL)
\end{equation}
for $1\le i\le m$ and $j=0,\dots,i-1,i+1,\dots, m+1$,
and decompose the operator $\na_i$ into 
\begin{equation}
\label{eq:decomp-nai}
\na_i=\sum_{0\le j\le m+1}^{j\ne i} \frac{\cR_{i,j}}{x_i-x_j}.
\end{equation}

\begin{lemma}
\label{lem:eigenvals-cRij}
The eigenvalues of $\cR_{i,j}$  are $0$ and $\a_i+\a_j$. 
If $\a_i=\a_j=0$ then $\cR_{i,j}$ is the zero map. Otherwise,
the $0$-eigenspace of $\cR_{i,j}$ is $m$ dimensional and 
an $(\a_i+\a_j)$-eigenvector of $\cR_{i,j}$ is given by

\begin{equation}
\label{eq:roots}
\f_{i,j}=\left\{
\begin{array}{lcc}
\a_j\f_{i,m+2}-\a_i\f_{j,m+2} 
=\left\{\begin{array}{ccc}
\a_i\a_j(\dfrac{dt}{t-x_i}-\dfrac{dt}{t-x_j})
& \textrm{if} & j\le m,\a_i\ne 0,\a_j\ne0, \\
\a_i\dfrac{u(x_j)dh_j(t)}{u(t)}
& \textrm{if} & j\le m,\a_i\ne 0,\a_j=0, \\
\a_j\dfrac{-u(x_i)dh_i(t)}{u(t)}
& \textrm{if} & j\le m,\a_i=0,\a_j\ne0, \\[4mm]
0
& \textrm{if} & j\le m,\a_i=0,\a_j=0, 
\end{array}
\right.\\
\f_{i,m+2}-\a_i\f_{m+1,m+2} 
=\left\{\begin{array}{ccc}
\a_i(\dfrac{dt}{t-x_i}-\dfrac{dt}{t-1})
& \textrm{if} & j=m\!+\!1, \a_i\ne0, \\
\dfrac{-u(x_i)dh_{i}(t)}{u(t)}
& \textrm{if} & j=m\!+\!1,\a_i=0,
\end{array}
\right.
\end{array}
\right.
\end{equation}
where $\f_{i,m+2}$ are given in (\ref{eq:frame}).

\end{lemma}

\begin{remark}
\label{rem:non-diag-cRij}
\begin{enumerate}
\item Note that $\f_{j,i}=-\f_{i,j}$ for $1\le i,j\le m$, $i\ne j$. 
We set 
$$\f_{0,i}=-\f_{i,0},\quad \f_{m+1,i}=-\f_{i,m+1}.$$
\item If $\a_i+\a_j=0$ and $\a_i\a_j\ne0$ then 
the $0$-eigenspace of $\cR_{i,j}$ is $m$ dimensional and  
this space includes $\f_{i,j}$. 
In this case, $\cR_{i,j}$ is not diagonalizable. 
\end{enumerate}
\end{remark}

\begin{proof}
We fix $i$ and $j$ satisfying $1\le i\le m$, $0\le j\le m+1$ and $j\ne i$.
For $0\le k\le m+1$, $k\ne i,j$, we have 
$$\na_i(\f_{k,m+2})=\left\{
\begin{array}{ccc}
-\dfrac{\a_k\f_{i,m+2}-\a_i\f_{k,m+2}}{x_i-x_k}
&\textrm{if} & 0\le k\le m,\\
-\dfrac{\f_{i,m+2}-\a_i\f_{m+1,m+2}}{x_i-1}&\textrm{if} &  k=m+1,
\end{array}
\right.
$$
by Lemma \ref{lem:naj-action}. Thus we have 
$$\lim_{x_i\to x_j}(x_i-x_j)\na_i(\f_{k,m+2})=0,
$$
which means that $\f_{k,m+2}$ is a $0$-eigenvector of $\cR_{i,j}$. 
By Theorem \ref{th:cint-number} (2), the dimension of 
the $0$-eigenspace of $\cR_{i,j}$ is greater than or equal to $m$.
Since we have  
$$
\f_{i,m+2}=-\sum_{0\le k\le m}^{k\ne i}\f_{k,m+2}-\a_{m+1}\f_{m+1,m+2}
$$
by Theorem \ref{th:cint-number} (2), 
$\f_{i,j}$ is expressed as
$$
\f_{i,j}=\left\{
\begin{array}{ccc}
-(\a_i+\a_j)\f_{j,m+2}
-\sum\limits_{0\le k\le m}^{k\ne i,j}\a_j\f_{k,m+2}
-\a_j\a_{m+1}\f_{m+1,m+2}
 & \textrm{if} & j<m+1,\\
-(\a_i+\a_{m+1})\f_{m+1,m+2}-\sum_{0\le k\le m}^{k\ne i}\f_{k,m+2}
& \textrm{if} & j=m+1.\\
\end{array}
\right.
$$
By Lemma \ref{lem:naj-action}, we have 
\begin{align*}
&\lim_{x_i\to x_j}(x_i-x_j)\na_i(\f_{i,j})
=-(\a_i+\a_j)(-\a_j\f_{i,m+2}+\a_i\f_{j,m+2})=(\a_i+\a_j)\f_{i,j},
\\
&\lim_{x_i\to 1}(x_i-1)\na_i(\f_{i,m+1})
=-(\a_i+\a_{m+1})(-\f_{i,m+2}+\a_i\f_{m+1,m+2})=(\a_i+\a_{m+1})\f_{i,m+1}.
\end{align*}
Thus $\f_{i,j}$ is an $(\a_i+\a_j)$-eigenvector of $\cR_{i,j}$. 

If $\a_i+\a_j=0$ and $\a_i\a_j\ne 0$ then $\f_{i,j}$ is different from 
the zero vector and it belongs to the $0$-eigenspace of $\cR_{i,j}$. 
By Lemma \ref{lem:naj-action},  
$$\cR_{i,j}(\f_{j,m+2})=
\left\{
\begin{array}{ccl}
-\a_j\f_{i,m+2}+\a_i\f_{j,m+2}=-\f_{i,j}& \textrm{if} & 0\le j\le m,j\ne i,\\
-\f_{i,m+2}+\a_i\f_{m+1,m+2}=-\f_{i,m+1}& \textrm{if} & j=m+1,\\
\end{array}
\right.
$$
$\cR_{i,j}$ is not the zero map. 
Since the $0$-eigenspace of $\cR_{i,j}$ is $m$ dimensional, 
any element $\f\in \cH^1_\C(\cL)$ is expressed as a linear combination of 
$\f_{j,m+2}$ and elements of the $0$-eigenspace of $\cR_{i,j}$. 
Thus $\cR_{i,j}(\f)$ is a scalar multiple of 
$\f_{i,j}$, which belongs to the $0$-eigenspace of $\cR_{i,j}$. 
Hence $\cR_{i,j}^2$ is the zero map and the set of eigenvalues of $\cR_{i,j}$ 
consists of $0$.

If $\a_i=\a_{m+1}=0$ then $\f_{i,m+1}$ is different from the zero vector. 
Thus the $0$-eigenspace of $\cR_{i,m+1}$ is $m+1$ dimensional, 
which means that $\cR_{i,m+1}$ is the zero map.
If $\a_i=\a_j=0$, $0\le j\le m$, $j\ne i$ then $\f_{i,j}$ degenerates 
to the zero vector. In this case, 
$\f_{j,m+2}$ satisfies 
$$\cR_{i,j}(\f_{j,m+2})=-\f_{i,j}=0$$
by $\a_i=\a_j=0$, it is a $0$-eigenvector of $\cR_{i,j}$. 
Hence the $0$-eigenspace of $\cR_{i,j}$ is $m+1$ dimensional, and 
$\cR_{i,j}$ is  the zero map.
\end{proof}

\begin{theorem}
\label{th:Gauss-Manin-connection}
The linear transformation $\cR_{i,j}$ in (\ref{eq:cRij}) is expressed by the intersection form $\cI_c$ as
$$
\cR_{i,j}:\cH^1_\C(\cL)\ni \f \mapsto \frac{-1}{2\pi\sqrt{-1}}
\cI_c(\f,\psi_{i,j})
\f_{i,j}\in \cH^1_\C(\cL), 
$$
where $\f_{i,j}$ are given in (\ref{eq:roots}) and 
\begin{align}
\nonumber
\psi_{i,j}&=\left\{
\begin{array}{lcl}
-\psi _{0,i}=\dfrac{dt}{t}-\dfrac{dt}{t-x_i}
& \textrm{if} & j=0,\\
\psi _{0,j}-\psi _{0,i}=\dfrac{dt}{t-x_j}-\dfrac{dt}{t-x_i}
& \textrm{if} & 1\le j\le m,j\ne i,\\
\end{array}
\right.\\
\label{eq:dual-roots}
\\
\nonumber
\psi_{i,m+1}&=\psi _{0,m+1}-\a_{m+1}\psi _{0,i}=\left\{
\begin{array}{lcl}
\a_{m+1}\big(\dfrac{dt}{t-1}-\dfrac{dt}{t-x_i}\big)
& \textrm{if} & \a_{m+1}\ne 0,\\
\dfrac{u(t)dh_{m+1}(t)}{u(1)}
& \textrm{if} & \a_{m+1}= 0.\\
\end{array}
\right.
\end{align}
Here $\psi_{0,i}$ are given in (\ref{eq:dual-frame}). 
The Gauss-Manin connection $\na_x$ restricted to $\cH^1_\C(\cL)$ 
is expressed as 
\begin{align*}
\na_x(\f)&=\sum_{i=1}^m \sum_{0\le j\le m+1}^{j\ne i}
\frac{dx_i}{x_i-x_j}\wedge \cR_{i,j}(\f)
=\sum_{0\le i<j\le m+1} 
\frac{dx_i-dx_j}{x_i-x_j}\wedge\cR_{i,j}(\f)
\\
&=\frac{-1}{2\pi\sqrt{-1}}
\sum_{0\le i< j\le m+1} 
\frac{dx_i-dx_j}{x_i-x_j}\wedge\cI_c(\f,\psi _{i,j})\f_{i,j},
\end{align*}
where $\f\in \cH^1_\C(\cL)$ and $dx_0=dx_{m+1}=0$.
\end{theorem}

\begin{remark}
\begin{enumerate}
\item The kernel of $\cR_{i,j}$ is 
$$(\psi _{i,j})^\perp
=\{\f\in \cH^1_\C(\cL)\mid \cI_c(\f,\psi _{i,j})=0\}.$$
\item If $\a_i+\a_j\ne 0$ then $\cR_{i,j}$ admits the expression 
$$\cR_{i,j}:\f\mapsto (\a_i+\a_j)\frac{\cI_c(\f,\psi _{i,j})}
{\cI_c(\f_{i,j},\psi _{i,j})}\f_{i,j},
$$
which is invariant under non-zero scalar multiples of 
$\f_{i,j}$ and $\psi _{i,j}$.
We can see that 
$\f_{i,j}$ is an $(\a_i+\a_j)$-eigenvector of $\cR_{i,j}$
by this expression.
\item Any section $\f(x)$ of $\cH^1(\cL)$ can be expressed as 
$$\sum_{i=1}^m c_i(x)\f_{i,m+2},
$$
where $c_i(x)$ are holomorphic functions on $X$. 
Its image under $\na_x$ is 
$$\na_x \f(x)=\sum_{i=1}^m [c_i(x)\na_x\f_{i,m+2}+(d_xc_i(x))\f_{i,m+2}].
$$

\end{enumerate}
\end{remark}

\begin{proof}
We set 
$$\cR_{i,j}':
\f \mapsto \frac{-1}{2\pi\sqrt{-1}}
\cI_c(\f,\psi_{i,j})\f_{i,j}$$
and study its eigenspaces.
By Lemma \ref{lem:frame} together with (\ref{eq:roots}) 
and (\ref{eq:dual-roots}),  
we have
$$\cI_c(\f_{k,m+2},\psi_{i,j})=0$$ 
for $0\le k\le m+1$, $k\ne i,j$ and 
$$\cI_c(\f_{i,j},\psi_{i,j})=-2\pi\sqrt{-1}(\a_i+\a_j).$$
Thus $\f_{k,m+2}$'s belong to the $0$-eigenspace of $\cR_{i,j}'$ and 
$\f_{i,j}$ is an $(\a_i+\a_j)$-eigenvector of $\cR_{i,j}'$ unless 
$\a_i=\a_j=0$.
Hence if $\a_i+\a_j\ne0$ then 
the eigenspaces of $\cR_{i,j}$ coincide with those of $\cR_{i,j}'$, 
and this means that $\cR_{i,j}=\cR_{i,j}'$.
If $\a_i+\a_j=0$, $\a_i\a_j\ne0$, then we have seen that 
$\cR_{i,j}(\f_{j,m+2})=-\f_{i,j}$ in Proof of Lemma \ref{lem:eigenvals-cRij}.
In this case,  we have 
$$\cR_{i,j}'(\f_{j,m+2})=-\f_{i,j}$$
by
$$\cI_c(\f_{j,m+2},\psi _{i,j})=2\pi\sqrt{-1};$$
hence $\cR_{i,j}=\cR_{i,j}'$ holds.
If $\a_i=\a_j=0$ then $\cR_{i,j}$ is the zero map. 
In this case, if $j\le m$ then $\f_{i,j}$ is the zero vector
and $\cR_{i,j}'$ is the zero map; otherwise,  
$\f_{i,m+1}$ is an eigenvector of $\cR_{i,m+1}'$ 
of eigenvalue $\a_i+\a_j=0$, and $\cR_{i,m+1}'$ is the zero map.
Therefore, $\cR_{i,j}=\cR_{i,j}'$ holds for any case.

By the expression of $\cR_{i,j}$ and $\f_{j,i}=-\f_{i,j}$, 
$\psi _{j,i}=-\psi _{i,j}$ for $i\ne j$,  
$\cR_{j,i}$ coincides with $\cR_{i,j}$.
We can unite $\dfrac{dx_i}{x_i-x_j}\cR_{i,j}$ and 
$\dfrac{dx_j}{x_j-x_i}\cR_{j,i}$ to 
$\dfrac{dx_i-dx_j}{x_i-x_j}\cR_{i,j}.$ 
\end{proof}

\begin{cor}
\label{cor:cRij-matrix}
Let $R_{i,j}$ be the representation matrix of 
the linear transformation $\cR_{i,j}$ with respect to the frame 
$\tr(\f_{1,m+2},\dots, \f_{m,m+2},\f_{m+1,m+2})$ 
of $\cH^1(\cL)$. Then it admits an expression 
$$R_{i,j}=-w_{i,j}v_{i,j},$$
where 
$$v_{i,j}=\left\{
\begin{array}{lcl}
\a_0\ex_i-\a_i\ex_0
=(\a_i,\dots,\overset{i\textrm{-th}}{\a_0+\a_i},\dots \a_i
,\a_i\a_{m+1})
&\textrm{if}& j=0,\\
\a_j\ex_i-\a_i\ex_j
=(0,\dots,\overset{i\textrm{-th}}{\a_j},\dots 
,-\overset{j\textrm{-th}}{\a_i},\dots,0)
&\textrm{if}& 1\le i<j\le m,\\
\ex_i-\a_i\ex_{m+1}
=(0,\dots,\overset{i\textrm{-th}}{1},\dots ,0,-\a_i)
&\textrm{if}& j=m+1,\\
\end{array}
\right.
$$
$$w_{i,j}=\left\{
\begin{array}{lcl}
 -\tr\ex_i=\tr(0,\dots,\overset{i\textrm{-th}}{-1},0,\dots,0)
&\textrm{if} & j=0,\\
\tr\ex_j -\tr\ex_i
=\tr(0,\dots,\overset{i\textrm{-th}}{-1},\dots 
,\overset{j\textrm{-th}}{1},\dots,0)
&\textrm{if}& 1\le i<j\le m,\\
\tr\ex_{m+1}-\a_{m+1}\ex_i
=\tr(0,\dots,\overset{i\textrm{-th}}{-\a_{m+1}},\dots ,0,1)
&\textrm{if} & j=m+1,\\
\end{array}
\right.$$
and $v_{j,i}=-v_{i,j}$, $w_{j,i}=-w_{i,j}$.
Here $\ex_k$ $(1\le k\le m+1)$ is the $k$-th unit row vector of size $m+1$,
and $\ex_0=(-1,\dots,-1,-\a_{m+1})$.
The Gauss-Manin connection is represented as 
$$\na_x\tr(\f_{1,m+2},\dots, \f_{m,m+2},\f_{m+1,m+2})
=R(x)\tr(\f_{1,m+2},\dots, \f_{m,m+2},\f_{m+1,m+2}),$$  
where 
$$
R(x)=\sum_{0\le i<j\le m+1} \frac{dx_i-dx_j}{x_i-x_j}R_{i,j}.
$$
\end{cor}

\begin{proof}
We identify elements 
$$\f=\sum_{k=1}^{m+1} v_k \f_{k,m+2}\in \cH^1(\cL),\quad 
\psi =\sum_{k=1}^{m+1} w_k \psi_{0,k}\in \cH^1(\cL^\vee),
$$ 
with a row vector $v=(v_1,\dots,v_m,v_{m+1})$
and a column vector $w=\tr(w_1,\dots,w_m,w_{m+1})$, respectively.
Since 
$$\cI_c(\f,\psi _{i,j})=2\pi\sqrt{-1}\cdot v\cdot w_{i,j}, \quad 
\f_{i,j}=v_{i,j}\cdot \tr(\f_{1,m+2},\dots, \f_{m,m+2},\f_{m+1,m+2}),$$
we have
$$\cR_{i,j}(\f)= v\cdot (-w_{i,j}v_{i,j})\cdot 
\tr(\f_{1,m+2},\dots, \f_{m,m+2},\f_{m+1,m+2}),$$
which means $R_{i,j}=-w_{i,j}v_{i,j}$.
\end{proof}

\begin{cor}
\label{cor:dual-connection}
\begin{enumerate}
\item 
The dual connection $\na_x^\vee$ on $\cH^1(\cL^\vee)$ is expressed as 
$$
\na_x^\vee(\psi )=\frac{1}{2\pi\sqrt{-1}}
\sum_{0\le i< j\le m+1} 
\frac{dx_i-dx_j}{x_i-x_j}\wedge \cI_c(\f_{i,j},\psi )\psi _{i,j},
$$
where $\psi \in \cH^1(\cL^\vee)$, $dx_0=dx_{m+1}=0$ and 
$\f_{i,j}$ and $\psi _{i,j}$ are given in (\ref{eq:roots}) 
and (\ref{eq:dual-roots}), respectively.
\item 
Let $R_{i,j}^\vee$ be the representation matrix of 
the linear transformation 
$$\cR_{i,j}^\vee(\psi )=\dfrac{\cI_c(\f_{i,j},\psi )\psi _{i,j}}
{2\pi\sqrt{-1}}$$
with respect to the frame 
$(\psi _{0,1},\dots, \psi _{0,m},\psi _{0,m+1})$ 
of $\cH^1(\cL^\vee)$. 
Then it coincides with $-R_{i,j}$ in 
Corollary \ref{cor:cRij-matrix}.
\end{enumerate}
\end{cor}

\begin{remark}
\begin{enumerate}
\item 
Since $T\ne T^\vee$ in general,  
we cannot regard the dual connection $\na_x^\vee$ as 
the Gauss-Manin connection for the parameter $-\a$.
\item 
The matrix $R_{i,j}$ in Theorem \ref{th:Gauss-Manin-connection} acts 
on 
the frame 
$\tr(\f_{1,m+2},\dots, \f_{m,m+2},\f_{m+1,m+2})$ 
of $\cH^1(\cL)$ from the left, on the other hand, 
the matrix $-R_{i,j}$ in Corollary \ref{cor:dual-connection} 
acts on the frame 
$(\psi _{0,1},\dots, \psi _{0,m},\psi _{0,m+1})$ 
of $\cH^1(\cL^\vee)$  from the right.
\end{enumerate} 
\end{remark}

\begin{proof}
Since $\cI_c(\f,\psi )$ is independent of $x_1,\dots,x_m$
for $\f\in \cH^1_\C(\cL)$ and $\psi \in \cH^1_\C(\cL^\vee)$, 
we have
$$0=d_x\cI_c(\f,\psi )=\cI_c(\na_x\f,\psi )+\cI_c(\f,\na_x^\vee\psi )$$
by (\ref{eq:act-on-intno}).
Thus $\na_x^\vee$ admits a decomposition 
$$\na_x^\vee=\sum_{0\le i<j\le m+1}\frac{dx_i-dx_j}{x_i-x_j}
\wedge \cR_{i,j}^\vee,
$$
and each $\cR_{i,j}^\vee$  satisfies 
$$\cI_c(\f,\cR_{i,j}^\vee(\psi ))=-\cI_c(\cR_{i,j}(\f),\psi ).
$$
Array the identities for $\f_{k,m+2}$'s and $\psi_{0,l}$'s for 
$1\le k,l\le m+1$ as
$$\Big(\cI_c(\f_{k,m+2},\cR_{i,j}^\vee(\psi_{0,l}))\Big)
_{\substack{1\le k\le m+1\\ 1\le l\le m+1}}
=-\Big(\cI_c(\cR_{i,j}(\f_{k,m+2}),\psi_{0,l})\Big)_{\substack{
1\le k\le m+1\\ 1\le l\le m+1}}.
$$
Since 
$$
(\dots, \cR_{i,j}^\vee(\psi_{0,l}),\dots)
=(\dots,\psi_{0,l},\dots)R_{i,j}^\vee,
\quad 
\begin{pmatrix}
\vdots\\
\cR_{i,j}(\f_{k,m+2})\\
\vdots 
\end{pmatrix}=R_{i,j}
\begin{pmatrix}
\vdots\\
\f_{k,m+2}\\
\vdots
\end{pmatrix},
$$
we have an identity
$$
\Big(\cI_c(\f_{k,m+2},\psi_{0,l})\Big)_{\substack{
1\le k\le m+1\\ 1\le l\le m+1}}
R_{i,j}^\vee=
-R_{i,j}\Big(\cI_c(\f_{k,m+2},\psi_{0,l})\Big)_{\substack{
1\le k\le m+1\\ 1\le l\le m+1}}.
$$
By the duality $\cI_c(\f_{k,m+2},\psi_{0,l})=2\pi\sqrt{-1}\d_{[k,l]}$, 
$R_{i,j}^\vee=-R_{i,j}$ is obtained. 

Since the set of eigenvalues of $R_{i,j}$ consists 
of $0$ and $\a_i+\a_j$, that of $R_{i,j}^\vee$ consists 
of $0$ and $-(\a_i+\a_j)$. 
The identity 
$$R_{i,j}^\vee w_{i,j}=(w_{i,j}v_{i,j}) w_{i,j}=
w_{i,j}(v_{i,j}w_{i,j})=w_{i,j}\cI_c(\f_{i,j},\psi_{i,j})=
-(\a_i+\a_j)w_{i,j},$$
means that 
$w_{i,j}$ is a $-(\a_i+\a_j)$-eigenvector of $R_{i,j}^\vee$, i.e.,  
$\psi_{i,j}$ is a $-(\a_i+\a_j)$-eigenvector of $\cR_{i,j}^\vee$. 
Since $v_{i,j}w=\cI_c(\f_{i,j},\psi)$ for 
$\psi =w_1\psi_{0,1}+\cdots+w_{m+1}\psi_{0,m+1}$, 
if $\cI_c(\f_{i,j},\psi )=0$ then $R_{i,j}^\vee w=(w_{i,j}v_{i,j})w=0$ 
i.e.,  $\cR_{i,j}^\vee(\psi )=0$. 
Hence if $\a_i+\a_j\ne 0$ then the $-(\a_i+\a_j)$-eigenspace of 
$\cR_{i,j}^\vee$ 
is spanned by $\psi_{i,j}$ and the $0$-eigenspace of $\cR_{i,j}^\vee$ is 
$$\{\psi \in \cH^1_\C\mid \cI_c(\f_{i,j},\psi )=0\}.$$
Therefore, $\cR_{i,j}^\vee$ admits the expression 
$$\cR_{i,j}^\vee(\psi )=-(\a_i+\a_j)
\frac{\cI_c(\f_{i,j},\psi )}
{\cI_c(\f_{i,j},\psi_{i,j})}
\psi _{i,j}
=
\frac{1}{2\pi\sqrt{-1}}
\cI_c(\f_{i,j},\psi )\psi _{i,j}.
$$
If $\a_i+\a_j=0$, then we can get this expression by case study in 
Proof of Theorem \ref{th:Gauss-Manin-connection}.
\end{proof}

A Pfaffian system of $\cF_D(a,b,c)$ is a first order differential equation 
\begin{equation}
\label{eq:Pfaffian}
d_x \mathbf{F}(x)=\Xi(x) \mathbf{F}(x)
\end{equation}
of a vector-valued unknown function 
$\mathbf{F}(x)=\tr(f_0(x),f_1(x),\dots,f_m(x))$ 
equivalent to the system $\cF_D(a,b,c)$.
Here, $f_0(x)$ is supposed to be a local solution to $\cF_D(a,b,c)$ 
and each $f_i(x)$ $(1\le i\le m)$ is given by an action of 
$\mathcal{O}(X)\la \pa_1,\dots,\pa_m\ra$ on $f_0(x)$, 
where $\mathcal{O}(X)$ is the ring of holomorphic function on $X$. 
Each entry of the connection matrix $\Xi(x)$ in (\ref{eq:Pfaffian}) 
belongs to the space $\W^1(X)$ of holomorphic $1$-forms on $X$, and 
the integrability condition 
$$d_x \Xi(x)=\Xi(x)\wedge \Xi(x)$$ 
holds.

Since $\la \f_{m+1,m+2},\g^u\ra$ is a local solution to $\cF_D(a,b,c)$ for 
any local section $\g^u\in H_1(T,D;\cL)$ and 
$\pa_i\la \f_{m+1,m+2},\g^u\ra=\la \na_i\f_{m+1,m+2},\g^u\ra$, 
we obtain a Pfaffian system of $\cF_D(a,b,c)$ from 
the Gauss-Manin connection by relating $\mathbf{F}(x)$ to 
our frame of $\cH^1(\cL)$
(in other words, 
by determining a Gauss-Manin vector introduced in \cite{GM}).

\begin{theorem}
\label{th:Pfaffian}
Let $f_0(x)$ be a local solution to $\cF_D(a,b,c)$.  
We define a vector valued function 
$\mathbf{F}(x)=\tr(f_0(x),f_1(x),\dots,f_m(x))$ 
by 
$$f_i(x)=(x_i-1)\pa_i f_0(x)\quad (1\le i\le m).$$
Then $\mathbf{F}(x)$ satisfies a Pfaffian system 
$$d_x \mathbf{F}(x)=\Xi(x) \mathbf{F}(x),$$
where 
$$\Xi(x)=P R(x) P^{-1}, \quad 
P=\begin{pmatrix} \mathbf{0}_m & 1 \\
-E_m &
\begin{matrix} \a_1\\
\vdots \\
\a_m
\end{matrix}
\end{pmatrix}
\in GL_{m+1}(\C).
$$
\end{theorem}
\begin{proof}
There exists $\g^u\in H_1(T,D;\cL)$ such that 
$$f_0(x)=\la \frac{dt}{t-1},\g^u\ra.$$
By Lemma \ref{lem:diff-phi0} together with (\ref{eq:frame}), 
we have 
$$\pa_if_0(x)=\la \na_i\frac{dt}{t-1},\g^u\ra
=\frac{1}{x_i-1}\la \a_i\f_{m+1,m+2}-\f_{i,m+2},\g^u\ra.
$$
Thus we can identify $\mathbf{F}(x)$ with
$$P\tr (\f_{1,m+2},\dots,\f_{m,m+2},\dots,\f_{m+1,m+2}).
$$
Since the matrix $P$ is independent of $x_1,\dots,x_m$, 
the connection matrix $\Xi(x)$ 
is obtained by $PR(x)P^{-1}$.
\end{proof}
\begin{cor}
\label{cor:Pfaffian}
Let $f_0(x)$ be a local solution to $\cF_D(a,b,c)$.  
A vector valued function 
$\bm{f}(x)=\tr(f_0(x),\pa_1f_0(x),\dots,\pa_mf_0(x))$ 
satisfies a Pfaffian system 
$$d_x \bm{f}(x)=\Theta(x) \bm{f}(x),$$
where 
$$\Theta(x)=Q(x) \Xi(x) Q(x)^{-1}+[d_xQ(x)] Q(x)^{-1}, 
\quad Q(x)=\diag\Big(1,\frac{1}{x_1-1},\dots,\frac{1}{x_m-1}\Big).
$$
\end{cor}
\begin{proof}
Since $\bm{f}(x)=Q(x)\mathbf{F}(x)$, it satisfies
\begin{align*}
&d_x\bm{f}(x)=[d_xQ(x)]\mathbf{F}(x)+Q(x)d_x\mathbf{F}(x)\\
=&[d_xQ(x)Q(x)^{-1}]Q(x)\mathbf{F}(x)+Q(x)\Xi(x)Q(x)^{-1}Q(x)\mathbf{F}(x)
\\=&[d_xQ(x)Q(x)^{-1}]\bm{f}(x)+Q(x)\Xi(x)Q(x)^{-1}\bm{f}(x)
=\Theta(x) \bm{f}(x);
\end{align*}
we have this corollary.
\end{proof}

\section{The monodromy representation of $\cF_D(a,b,c)$}
By patching the trivial vector bundles 
$$\prod\limits _{x\in W_n}H_1(T_x,D_x;\cL_x)$$
for an open covering $\{W_n\}_{n\in N}$ of $X$, we have a local system 
$$
\cH_1(\cL)=\bigcup_{n\in N}
\prod\limits _{x\in W_n}H_1(T_x,D_x;\cL_x)
$$
of rank $m+1$ over $X$.
We take a base point $\dot x\in W_0\in \{W_n\}_{n\in N}$ so that 
$$(\dot x_0,\dot x_1,\dots,\dot x_m,\dot x_{m+1}, \dot x_{m+2}) 
= (0,\dot x_1,\dots,\dot x_m,1, \infty) 
$$
are aligned as in (\ref{eq:base-point}) for a fixed parameter $\a$. 
By the continuation of any section of the trivial vector bundle  
$H_1(T,D;\cL)=\prod\limits _{x\in W_0}H_1(T_x,D_x;\cL_x)$  
along a path in $X$ from $\dot x$ to any point $x'\in X$,
we have a linear isomorphism from $H_1(T_{\dot x},D_{\dot x};\cL_{\dot x})$ to 
$H_1(T_{x'},D_{x'};\cL_{x'})$.
In particular, for a loop $\rho$ in $X$ with terminal $\dot x$,  
we can make the continuation $\rho_*(\g^u)$ of any section 
$\g^u\in H_1(T,D;\cL)$. 
The linear transformation 
$$\cM_\rho:H_1(T,D;\cL)\ni \g^u\mapsto \rho_*(\g^u)\in H_1(T,D;\cL)
$$
is called the circuit transformation of $H_1(T,D;\cL)$ along $\rho$, and 
the homomorphism 
$$\cM:\pi_1(X,\dot x)\ni \rho \mapsto \cM_\rho\in GL(H_1(T,D;\cL))
$$
is called the monodromy representation of $\cH_1(\cL)$.
Thanks to Theorem \ref{th:hom-sol-iso}, 
we can study the monodromy representation of 
$\cF_D(a,b,c)$ as that of the local system $\cH_1(\cL)$.

We give generators of $\pi_1(X,\dot x)$. 
We set 
$$\C_p(\dot x)=\{
(\dot x_1,\dots, \dot x_{p-1},x_p, \dot x_{p+1},\dots, \dot x_m)
\mid x_p\in \C\}\quad (1\le p\le m),$$
which are lines in $\C^m$ passing through $\dot x$. 
For distinct indices $1\le p\le  m$ and $0\le q\le m+1$, 
let  $\rho_{p,q}$ be a loop 
in $X\cap \C_p(\dot x)$ starting from $x_p=\dot x_p$, approaching to 
$\dot x_q$ via the upper half space in $\C_p(\dot x)$, 
turning $\dot x_q$ once positively, and tracing back to $\dot x_p$.
It is known that 
$\pi_1(X,\dot x)$ is generated by the loops 
$\rho_{p,q}$, 
where $0\le p<q\le m+1$, $(p,q)\ne(0,m+1)$, and 
$\rho_{0,p}$ is regarded as the loop $\rho_{p,0}$ in $X\cap \C_p(\dot x)$.

\begin{theorem}
\label{th:monod-rep}
The circuit transformation $\cM_{p,q}=\cM(\rho_{p,q})$ is expressed as 
$$ \g^u\mapsto \g^u-\cI_h(\d^{u^{-1}}_{p,q},\g^u)\g^u_{p,q},
$$
where $\g^u_{p,q}$ and $\d^{u^{-1}}_{p,q}$ are given in Table 
\ref{tab:roots}.
\begin{table}[htb]
$$
\begin{array}{|c|c|r|}
\hline
\g^u_{p,q} &\d^{u^{-1}}_{p,q} & 
\textrm{conditions on } p,q
\\
\hline
\ell_{q}^u-\ell^u_{p} &0 & q\in \ID\\
\circlearrowleft_q^{u} &\circlearrowleft^{u^{-1}}_{p} &p\in \ID,\ q\in \IP\\
\frac{1}{1-\l_q}\circlearrowleft_q^{u}-\ell^u_{p} 
&(1-\l_q)\circlearrowleft^{u^{-1}}_{p} 
& q\in \IZc\\
\hline
-\circlearrowleft_p^{u} 
&-\circlearrowleft_q^{u^{-1}}
&q\in \ID \\
0
&\ell_{q}^{u^{-1}}-\ell^{u^{-1}}_{p} 
&p\in \IP,\ q\in \IP \\
-\circlearrowleft_p^{u} 
&-\l_q\circlearrowleft_q^{u^{-1}}-(1-\l_q)\ell^{u^{-1}}_{p} 
& q\in \IZc \\
\hline
\ell_q^{u}-\frac{1}{1-\l_p}\circlearrowleft_p^{u}
&-(1-\l_p)\circlearrowleft_{q}^{u^{-1}}
& q\in \ID\\
\circlearrowleft_q^{u}
&(1-\l_p)\ell_{q}^{u^{-1}}+\l_p\circlearrowleft^{u^{-1}}_{p} 
&p\in \IZc,\  q\in \IP\\
\frac{1}{1-\l_q}\circlearrowleft_q^{u}-\frac{1}{1-\l_p}\circlearrowleft_p^{u}
&-\l_q(1-\l_p)\circlearrowleft_q^{u^{-1}}
+\l_p(1-\l_q)\circlearrowleft_p^{u^{-1}} 
&q\in \IZc\\
\hline
%
\end{array}
$$
\caption{$\g_{p,q}^u$ and $\d_{p,q}^{u^{-1}}$}
\label{tab:roots}
\end{table}
\end{theorem}

\begin{remark}
\begin{enumerate}
\item 
Our proof of this theorem is based on \cite[Theorem 5.4]{M3}, 
in which 
$$\g_{\imath(p)\imath(q)}^{-\a}=(\l_p^{-1}-1)\g_{\imath(q)}^{-\a}
-(\l_q^{-1}-1)\g_{\imath(p)}^{-\a}
$$ 
should be multiplied by $(-1)$.

\item  For $\a_p\notin \Z$,
we have 
$$\lim_{t\to x_p}\int_{\overrightarrow{\dot t t}} u(t)\phi_0=
\frac{1}{1-\l_p}\int_{\circlearrowleft_p} u(t)\f_0$$
where $\phi_0=\imath_D(\f_0)\in H^1_{C^\infty_V}(T,D;\cL)$, and 
$\overrightarrow{\dot t t}$ denotes the segment from $\dot t$ to $t$.
Thus we can regard $\frac{1}{1-\l_p}\circlearrowleft_p^u$ as $\ell_p^u$. 
Under the limit as $\a_p$ converges to 
an integer such that $x_p\in D$, we see that this regard is valid. 
Similarly, $\frac{1}{1-\l_p^{-1}}\circlearrowleft_p^{u^{-1}}$ 
can be regarded as $\ell_p^{u^{-1}}$ if $\a_p$ is a non-integer or
an integer such that $x_p\in D^\vee$. 
For examples, 
$\frac{1}{1-\l_q}\circlearrowleft_q^{u}-\frac{1}{1-\l_p}\circlearrowleft_p^{u}$
can be regarded as $\ell_q^u-\ell_p^u$ for $p,q\in \ID\cup\IZc$, 
and $-\l_q(1-\l_p)\circlearrowleft_q^{u^{-1}}
+\l_p(1-\l_q)\circlearrowleft_p^{u^{-1}}$ 
can be regarded as 
$(1-\l_p)(1-\l_q)(\ell_q^{u^{-1}}-\ell_p^{u^{-1}})$ for 
$p,q\in \IP\cup\IZc$. 

\item 
In case of $p\in \IP$ or $q\in \IP$,
the cycles $\g^u_{p,q}$ and $\d^{u^{-1}}_{p,q}$ are given by the multiplications 
\begin{align*}
&\big(\prod_{i\in \{p,q\}\cap\IP}(1-\l_i)\big)\cdot 
\big(\frac{1}{1-\l_q}\circlearrowleft_q^{u}-\frac{1}{1-\l_p}\circlearrowleft_p^{u}\big),
\\
&\big(\prod_{i\in \{p,q\}\cap\IP}\frac{1}{1-\l_i}\big)\cdot 
\big(-\l_q(1-\l_p)\circlearrowleft_q^{u^{-1}}
+\l_p(1-\l_q)\circlearrowleft_p^{u^{-1}} \big)
\end{align*}
with the regard in (2). 
For examples, 
in case of $p\in \IZc$ and $q\in \IP$, they are 
$$\circlearrowleft_q^{u}-\frac{1-\l_q}{1-\l_p}\circlearrowleft_p^{u}, 
\quad 
-\frac{\l_q(1-\l_p)}{1-\l_q}\circlearrowleft_q^{u^{-1}}
+\l_p\circlearrowleft_p^{u^{-1}}
$$
regarded as $\circlearrowleft_q^{u}$ and 
$ (1-\l_p)\ell_q^{u^{-1}}+\l_p\circlearrowleft_p^{u^{-1}}$, respectively; 
in case of $p,q\in \IP$, they are 
$$(1-\l_p)\circlearrowleft_q^{u}-(1-\l_q)\circlearrowleft_p^{u}, 
\quad \frac{1}{1-\l_q^{-1}}\circlearrowleft_q^{u^{-1}}
-\frac{1}{1-\l_p^{-1}}\circlearrowleft_p^{u^{-1}},
$$
regarded as $0$ and $\ell_q^{u^{-1}}-\ell_p^{u^{-1}}$, respectively.

\item 
Note that 
$$
\cI_h(\d_{p,q}^{u^{-1}},\g_{p,q}^{u})=1-\l_p\l_q
$$
for any $\a_p,\a_q$. 
If it does not vanish, then $\cM_{p,q}$ is the complex reflection 
$$ \g^u\mapsto \g^u-(1-\l_p\l_q)
\frac{\cI_h(\d^{u^{-1}}_{p,q},\g^u)}
{\cI_h(\d^{u^{-1}}_{p,q},\g^u_{p,q})}\g^u_{p,q}
$$
with respect to $\cI_h$.

\item 
Suppose that one of $p$ and $q$ is in $\ID$ and the other is in $\IP$,
and consider the limit $x_p\to x_q$ 
in the upper half space $\H_T$ of $T$.
Then either $\g_{p,q}^u$ is a vanishing cycle of 
$H_1(T,D;\cL)$ or $\d_{p,q}^{u^{-1}}$ is that of $H_1(T^\vee,D^\vee;\cL^\vee)$, 
but not both occur at the same time. 
Though the circuit transformation $\cM_{p,q}$ is characterized by 
vanishing cycles as $x_p\to x_q$ in the other cases of parameters,  
it is not true in this case.
\end{enumerate}
\end{remark}

\begin{proof}
The circuit transformation $\cM_{p,q}$ is studied in \cite[Theorem 5.4]{M3}
except the case one of  $p$ and $q$ is in $\ID$ and 
the other is in $\IP$. In this case, 
we may assume that $p\in \ID$ and $q\in \IP$. 
We  see that the $1$-eigenspace of $\cM_{p,q}$ is 
$$(\circlearrowleft_p^{u^{-1}})^\perp=
\{\g^u\in H_1(T,D;\cL)\mid \cI_h(\circlearrowleft_p^{u^{-1}},\g^u)=0\},
$$
since we can select $1$-chains representing any element ${\g}^{u} \in 
(\circlearrowleft_p^{u^{-1}})^\perp$ so that they are not involved 
in the movement of $x_p$ and $x_q$ caused by $\rho_{p,q}$.
Thus the linear transformation $\cM_{p,q}$ is characterized by 
its image of ${\g'}^u\in H_1(T,D;\cL)$ satisfying 
$\cI_h(\circlearrowleft_p^{u^{-1}},{\g'}^u)=-1$
and its $1$-eigenspace $(\circlearrowleft_p^{u^{-1}})^\perp$.
By tracing the deformation of ${\g'}^u$ along $\rho_{p,q}$, 
we can see that 
$$\cM_{p,q}({\g'}^u)={\g'}^u+\circlearrowleft_q^{u}.$$
On the other hand, we have 
$$
\g^u-\cI_h(\d_{p,q}^{u^{-1}},\g^u)\g_{p,q}^u=\g^u
$$
for any $\g^u\in (\circlearrowleft_p^{u^{-1}})^\perp$, and 
$$
{\g'}^u-\cI_h(\d_{p,q}^{u^{-1}},{\g'}^u)\g_{p,q}^u
={\g'}^u+\circlearrowleft^{u}_{q}$$
since 
$\g^{u}_{p,q}=\circlearrowleft^{u}_{q}$,
$\d^{u^{-1}}_{p,q}=\circlearrowleft^{u^{-1}}_{p}$,  
$\cI_h(\circlearrowleft^{u^{-1}}_{p},\g^u)=0$ 
and $\cI_h(\circlearrowleft^{u^{-1}}_{p},{\g'}^u)=-1$. 
Hence $\cM_{p,q}(\g^u)$ coincides with 
$\g^u-\cI_h(\d^{u^{-1}}_{p,q},\g^u)\g^u_{p,q}$ 
for any $\g^u\in H_1(T,D;\cL)$.
\end{proof}

\begin{cor}
\label{cor:monod-mat}
With respect to the basis $(\g^u_1,\dots,\g^u_{m+1})$ of $H^1(T,D;\cL)$
given in (\ref{eq:hom-basis-1}) or (\ref{eq:hom-basis-2}),  
the representation matrix $M_{p,q}$ of $\cM_{p,q}$ is expressed as 
$$
M_{p,q}=E_{m+1}-H y_{p,q}  z_{p,q},
$$
where the basis $\tr (\d^{u^{-1}}_1,\dots,\d^{u^{-1}}_{m+1})$ 
of $H^1(T^\vee,D^\vee;\cL^\vee)$ is given in 
(\ref{eq:hom-d-basis-1}) or (\ref{eq:hom-d-basis-2}), 
the intersection matrix $H$ with respect to 
$\tr (\d^{u^{-1}}_1,\dots,\d^{u^{-1}}_{m+1})$ 
and $(\g^u_1,\dots,\g^u_{m+1})$ is given in Proposition \ref{prop:dual-basis}, 
and $y_{p,q}$ and $z_{p,q}$ are column and row vectors 
satisfying 
$$\g^u_{p,q}=(\g^u_1,\dots,\g^u_{m+1})y_{p,q},\quad 
\d^{u^{-1}}_{p,q}=z_{p,q}\tr (\d^{u^{-1}}_1,\dots,\d^{u^{-1}}_{m+1}).
$$
\end{cor}

\begin{proof}
Let $y$ be a column vector satisfying 
$$\g^u=(\g^u_1,\dots,\g^u_{m+1})y$$
for any $\g^u\in H^1(T,D;\cL)$. 
Since $\cI_h(\d^{u^{-1}}_{p,q},\g^u)=z_{p,q}Hy$, we have 
\begin{align*}
\cM_{p,q}(\g^u)&=\g^u-\cI_h(\d^{u^{-1}}_{p,q},\g^u)\g^u_{p,q}
=(\g^u_1,\dots,\g^u_{m+1})\Big(y-(z_{p,q}Hy) y_{p,q}\big)\\
&=(\g^u_1,\dots,\g^u_{m+1})\Big(E_{m+1}-y_{p,q}z_{p,q}H\big)y.
\end{align*}
Hence the representation matrix $M_{p,q}$ is obtained.
\end{proof}

\begin{example}
We give examples of circuit matrices $M_{p,q}$ for $m=3$.
\\
(1) $\a_0,\a_5 \notin \Z$, $\a_1,\a_2\in \N_0$, $\a_3,\a_4\in -\N$.\\
In this case, we have $T=\P^1-\{0,x_3,1,\infty\}$, $D=\{x_1,x_2\}$, 
$\l_1=\l_2=\l_3=\l_4=1$, $\l_5=\l_0^{-1}(\ne 1)$. 
We have $H=E_4$,  and list $y_{p,q}$, $z_{p,q}$ and $M_{p,q}$ in 
Table \ref{tab:3-222}.

\begin{table}[hbt]
$$
\begin{array}{|c|c|c|c|}
\hline
p,q &y_{p,q} &z_{p,q} & M_{p,q} \\
\hline
0,1 &\begin{pmatrix} 1\\ 0\\0\\0\end{pmatrix}
& (1-\l_0,0,0,0) & \diag(\l_0,1,1,1)\\
0,2 &\begin{pmatrix}0\\1\\0\\0\end{pmatrix}
& (0,1-\l_0,0,0) & \diag(1,\l_0,1,1)\\
0,3 &\begin{pmatrix}0\\0\\1\\0\end{pmatrix}
& (1,1,1-\l_0,0) & 
\begin{pmatrix} 1 & 0 & 0 & 0\\
0 & 1 & 0 & 0\\
-1 &-1 & \l_0 & 0\\
0 & 0 & 0 & 1\\
\end{pmatrix}
\\
1,2 &\begin{pmatrix}-1\\1\\0\\0\end{pmatrix}
& (0,0,0,0)& E_4
\\
1,3 &\begin{pmatrix}0\\0\\1\\0\end{pmatrix}
& (-1,0,0,0) &
\begin{pmatrix} 1 & 0 & 0 & 0\\
0 & 1 & 0 & 0\\
1 &0 & 1 & 0\\
0 & 0 & 0 & 1\\
\end{pmatrix}
\\
1,4 &\begin{pmatrix}0\\0\\0\\1\end{pmatrix}
& (-1,0,0,0)&
\begin{pmatrix} 1 & 0 & 0 & 0\\
0 & 1 & 0 & 0\\
0 &0 & 1 & 0\\
1 & 0 & 0 & 1\\
\end{pmatrix}
\\
2,3 &\begin{pmatrix}0\\0\\1\\0\end{pmatrix}
& (0,-1,0,0) &
\begin{pmatrix} 1 & 0 & 0 & 0\\
0 & 1 & 0 & 0\\
0 & 1 & 1 & 0\\
0 & 0 & 0 & 1\\
\end{pmatrix}
\\
2,4 &\begin{pmatrix}0\\0\\0\\1\end{pmatrix}
& (0,-1,0,0)& 
\begin{pmatrix} 1 & 0 & 0 & 0\\
0 & 1 & 0 & 0\\
0 & 0 & 1 & 0\\
0 & 1 & 0 & 1\\
\end{pmatrix}
\\
3,4 &\begin{pmatrix}0\\0\\0\\0
\end{pmatrix}
& (0,0,-1,1) & E_4\\
\hline
\end{array}
$$
\caption{List of $y_{p,q}$, $z_{p,q}$ and $M_{p,q}$ for $\a_0,\a_5\notin \Z$, 
$\a_1,\a_2\in \N_0$, $\a_3,\a_4\in -\N$. }
\label{tab:3-222}
\end{table}
\noindent (2) 
$\a_0,\a_3,\a_4,\a_5 \notin \Z$, $\a_1\in \N_0$, $\a_2\in -\N$.\\
In this case, we have $T=\P^1-\{0,x_2,x_3,1,\infty\}$, $D=\{x_1\}$, 
$\l_1=\l_2=1$, $\l_5=\l_0^{-1}\l_3^{-1}\l_4^{-1}(\ne 1)$.
The intersection matrix becomes
$$H=\begin{pmatrix}
1 & 0 & 0 & 0 \\
0 & 1 &\l_3-1 & \l_4-1\\
0 & 0 & \l_3-1 &(1-\l_3^{-1})(\l_4-1)\\
0 & 0 & 0 &\l_4-1
\end{pmatrix}.
$$  
It is easy to express $\g_{p,q}^u$ as a linear combination of 
$\g_{1}^{u},\dots \g_{4}^{u}$, 
To express $\d_{p,q}^{u^{-1}}$ as a linear combination of 
$\d_{1}^{u^{-1}},\dots \d_{4}^{u^{-1}}$, 
we express 
$\frac{1}{1-\l_5^{-1}}\circlearrowleft^{u^{-1}}_{5}
-\frac{1}{1-\l_0^{-1}}\circlearrowleft^{u^{-1}}_{0}$
(regarded as $\ell_{0,5}^{u^{-1}}$) in terms of them. 
Let $\ell_{0,5}^{u^{-1}}$ be expressed as 
$(z_1,\dots,z_4)\tr(\d_1^{u^{-1}},\dots,\d_4^{u^{-1}})$. Then 
we have 
$$(z_1,\dots,z_4)H=
(\cI_h(\ell_{0,5}^{u^{-1}},\g^u_1),\dots,\cI_h(\ell_{0,5}^{u^{-1}},\g^u_4))
=\big(\frac{\l_0^{-1}}{\l_0^{-1}-1},0,\frac{(1-\l_3)\l_0^{-1}}{\l_0^{-1}-1},
\frac{(1-\l_4)\l_0^{-1}}{\l_0^{-1}-1}\big),
$$
and 
$$
(z_1,\dots,z_4)=\big(\frac{\l_0^{-1}}{\l_0^{-1}-1},0,\frac{(1-\l_3)\l_0^{-1}}{\l_0^{-1}-1},
\frac{(1-\l_4)\l_0^{-1}}{\l_0^{-1}-1}\big)H^{-1},$$
$$
\ell_{0,5}^{u^{-1}}=(\frac{1}{1-\l_0},0,\frac{-1}{1-\l_0},
\frac{-1}{(1-\l_0)\l_3})
\tr(\d_1^{u^{-1}},\dots,\d_4^{u^{-1}}).
$$
We list $y_{p,q},z_{p,q}$ and $M_{p,q}$ in Table 
\ref{tab:3-411}.

\begin{table}
$$
\begin{array}{|c|c|c|c|}
\hline
p,q &y_{p,q} &z_{p,q} &M_{p,q}\\
\hline
0,1 &\begin{pmatrix}1\\0\\0\\0\end{pmatrix}
& (1-\l_0,0,0,0)& \diag(\l_0,1,1,1)\\
0,2 &\begin{pmatrix}0\\1\\0\\0\end{pmatrix}
& (1,1-\l_0,-1,-\l_3^{-1}) & \begin{pmatrix}
1 & 0 & 0 & 0\\
-1 & \l_0 & \l_0(\l_3-1) &\l_0(\l_4-1)\\
0 & 0 & 1 & 0\\
0 & 0 & 0 & 1
\end{pmatrix}\\
0,3 &\begin{pmatrix}0\\0\\\frac{1}{1-\l_3}\\0\end{pmatrix}
& (1-\l_3,0,\l_0\l_3-1,1-\l_3^{-1})
&\begin{pmatrix}
1 & 0 & 0 & 0\\
0 & 1 & 0 & 0\\
-1& 0 & \l_0\l_3 &\l_0(\l_4-1)\\
0 & 0 & 0 & 1
\end{pmatrix}\\
1,2 &\begin{pmatrix}0\\1\\0\\0\end{pmatrix}
& (-1,0,0,0) & \begin{pmatrix}
1 & 0 & 0 & 0\\
1 & 1 & 0 & 0\\
0 & 0 & 1 & 0\\
0 & 0 & 0 & 1
\end{pmatrix}\\
1,3 &\begin{pmatrix}-1\\0\\\frac{1}{1-\l_3}\\0\end{pmatrix}
& (-(1-\l_3),0,0,0)&
\begin{pmatrix}
\l_3 & 0 & 0 & 0\\
0 & 1 & 0 & 0\\
1 & 0 & 1 & 0\\
0 & 0 & 0 & 1
\end{pmatrix}\\
1,4 &\begin{pmatrix}-1\\0\\0\\\frac{1}{1-\l_4}\end{pmatrix}
& (-(1-\l_4),0,0,0)&
\begin{pmatrix}
\l_4& 0 & 0 & 0\\
0 & 1 & 0 & 0\\
0 & 0 & 1 & 0\\
1 & 0 & 0 & 1
\end{pmatrix}
\\
2,3 &\begin{pmatrix}0\\-1\\0\\0\end{pmatrix}
& (0,-(1-\l_3),-\l_3,0)&
\begin{pmatrix}
1 & 0 & 0 & 0\\
0 & \l_3 & 1-\l_3 & 0\\
0 & 0 & 1 & 0\\
0 & 0 & 0 & 1
\end{pmatrix}\\
2,4 &\begin{pmatrix}0\\-1\\0\\0\end{pmatrix}
& (0,-(1-\l_4),0,-\l_4)&
\begin{pmatrix}
1 & 0 & 0 & 0\\
0 & \l_4 & (1-\l_3)(1-\l_4) & 1-\l_4\\
0 & 0 & 1 & 0\\
0 & 0 & 0 & 1
\end{pmatrix}\\
3,4 &\begin{pmatrix}0\\0\\\frac{-1}{1-\l_3}\\\frac{1}{1-\l_4}\end{pmatrix}
& (0,0,\l_3(1-\l_4),-\l_4(1-\l_3))&
\begin{pmatrix}
1 & 0 & 0 & 0\\
0 & 1 & 0 & 0\\
0 & 0 & \l_3\l_4-\l_3+1 & 1-\l_4\\
0 & 0 & (1-\l_3)\l_3 & \l_3
\end{pmatrix}\\
\hline
\end{array}
$$
\caption{List of $y_{p,q}$, $z_{p,q}$ and $M_{p,q}$ for 
$\a_0,\a_3,\a_4,\a_5\notin \Z$, 
$\a_1\in \N_0$, $\a_2\in -\N$. }
\label{tab:3-411}
\end{table}
\end{example}

\begin{theorem}
\label{th:reduce-monod}
If there exists an integral parameter $\a_i$, then 
the monodromy representation of $\cF_D(a,b,c)$ is reducible.
In particular, if $\a\in \Z^{m+3}$ and $\#(\ID)=1$ or $\#(\ID)=m+2$ 
then the monodromy representation of $\cF_D(a,b,c)$ becomes trivial.
\end{theorem}
\begin{proof}
(1) $\a\notin \Z^{m+3}$. \\
Suppose that $\a_i\in \Z$. 
If $i\in \IP$ then 
the $1$-dimensional span $\la \circlearrowleft_i^u\ra (\subset H_1(T;\cL)
\subset H_1(T,D;\cL))$ is invariant under any circuit transformation.
If $i\in \ID$ then 
the space $H_1(T;\cL)$ is a non-zero proper subspace of 
$H_1(T,D;\cL)$ and it is invariant under any circuit transformation.

\noindent
(2) $\a\in \Z^{m+3}$. \\
If $\#(\ID)\ne 1$ and $\#(\ID)\ne m+2$ 
then there is an invariant subspace under any circuit transformation
as studied in (1). 
If $\#(\ID)= 1$ then $x_{i_1}\in D$ and 
$H_1(T,D;\cL)$ is generated by 
$m+1$ twisted cycles $\circlearrowleft_{i_2}^u,\dots,
\circlearrowleft_{i_{m+2}}^u$,  
which are element-wise invariant under any circuit transformation since 
$u(t)$ is single valued.
If $\#(\ID)= m+2$ then the space $T$ is $\P^1-\{x_{i_0}\}$, 
which is simply connected. Hence 
the monodromy representation of $\cF_D(a,b,c)$ becomes trivial.
\end{proof}

\section*{Acknowledgment}
The author expresses his gratitude to Professor Tomohide Terasoma 
for valuable discussions and instructions about several cohomology groups.


\begin{thebibliography}{99}
%
\bibitem[AK]{AoKi}
Aomoto K. and Kita M. translated by Iohara K., 
\textsl{
Theory of Hypergeometric Functions
},
Springer Monographs in Mathematics,
Springer Verlag, Now York, 2011. 
%
%
 \bibitem[CM]{CM}
Cho K. and Matsumoto K., 
Intersection theory for twisted cohomologies and 
twisted Riemann's period relations I,   
\textsl{Nagoya Math. J.},  \textbf{139} (1995), 67--86. 

\bibitem[EV]{EV} 
Esnault H. and  Viehweg E. 
\textsl{Lectures on vanishing theorems}, DMV Seminar, 20. Birkh\"auser Verlag, 
Basel, 
1992.
%
\bibitem[GM]{GM} Goto Y. and Matsumoto K., 
Pfaffian equations and contiguity relations
  of the hypergeometric function of type
  $(k+1, k+n+2)$ and their applications, 
\textsl{Funkcial. Ekvac.} \textbf{61} (2018), 
315--347.
%
%
%
\bibitem[IKSY]{IKSY}
Iwasaki K., Kimura H., Shimomura S. and  Yoshida M., 
\textsl{From Gauss to Painlev\'e, 
-A modern theory of special functions-,}  
Aspects of Mathematics, E16. Friedr. Vieweg \& Sohn, Braunschweig, 1991.
%
\bibitem[KY]{KY} Kita M. and Yoshida M., 
Intersection theory for twisted cycles, 
\textsl{Math. Nachr.}, \textbf{166} (1994), 287–304. 
%
\bibitem[M1]{M1} Matsumoto K., 
Intersection numbers for logarithmic $k$-forms, 
\textsl{Osaka J. Math.}, \textbf{35} (1998), 
873--893. 
%
\bibitem[M2]{M2}
\bysame, 
Monodromy and Pfaffian of Lauricella's $F_D$
in terms of the intersection forms of twisted 
(co)homology groups,
\textsl{Kyushu J. Math.}, \textbf{67} (2013),  367--387.
%
\bibitem[M3]{M3}
\bysame, 
The monodromy representations of local systems
associated with Lauricella's $F_D$,
\textsl{Kyushu J. Math.} \textbf{71} (2017), 
329--348.
%
%
%
\bibitem[OT]{OT} Orlik P. and  Terao H., 
\textsl{Arrangements and hypergeometric integrals,} 
Second edition. 
MSJ Memoirs, 9, Mathematical Society of Japan, Tokyo, 2007.
%
%
\bibitem[Y1]{Y1} 
Yoshida M., 
\textsl{Fuchsian Differential Equations}, 
Aspects of Mathematics, E11., 
Vieweg \& Sohn, Braunschweig,  
1987.
%
\bibitem[Y2]{Y2} 
\bysame, 
\textsl{Hypergeometric functions, my love, 
-Modular interpretations of configuration spaces-}, 
Aspects of Mathematics E32.,  
Vieweg \& Sohn, Braunschweig,  
1997

\end{thebibliography}
\end{document}